\pdfoutput=1
\documentclass[11pt]{amsart}
\usepackage[T1]{fontenc}
\usepackage[utf8]{inputenc}
\usepackage{amssymb,amsmath,amsthm,mathrsfs,array,graphicx,bbm,enumerate,wasysym,dsfont,stmaryrd}
\usepackage{lineno}
\usepackage[usenames, dvipsnames]{color}
\usepackage[all]{xy}
\usepackage{fullpage}
\usepackage{float}
\usepackage{verbatim}
\usepackage{soul}
\usepackage{hyperref}

\DeclareFontFamily{U}{mathx}{\hyphenchar\font45}
\DeclareFontShape{U}{mathx}{m}{n}{
	<5> <6> <7> <8> <9> <10>
	<10.95> <12> <14.4> <17.28> <20.74> <24.88>
	mathx10
}{}
\DeclareSymbolFont{mathx}{U}{mathx}{m}{n}
\DeclareFontSubstitution{U}{mathx}{m}{n}
\DeclareMathAccent{\widecheck}{0}{mathx}{"71}
\DeclareMathAccent{\wideparen}{0}{mathx}{"75}

\newtheorem{thm}{Theorem}[section]
\newtheorem{lemma}[thm]{Lemma}
\newtheorem{rem}[thm]{Remark}
\newtheorem{defn}[thm]{Definition}
\newtheorem{prop}[thm]{Proposition}
\newtheorem{cor}[thm]{Corollary}
\newtheorem{claim}[thm]{Claim}

\newtheorem{propA}{Proposition}

\numberwithin{equation}{section}

\newcommand{\Q}{\mathbb Q}
\newcommand{\Z}{\mathbb Z}
\newcommand{\R}{\mathbb R}
\newcommand{\C}{\mathbb C}
\renewcommand{\H}{\mathbb H}

\newcommand{\N}{\mathbb N}
\newcommand{\D}{\mathbb D}
\newcommand{\E}{\mathbb E}

\renewcommand{\P}{\mathbb P}
\renewcommand{\1}{\mathbf 1}
\newcommand{\A}{\mathds A}

\newcommand{\B}{\mathcal B}
\newcommand{\ED}{\operatorname{ED}}

\newcommand{\F}{\mathcal F}

\renewcommand{\epsilon}{\varepsilon}

\newcommand{\An}{\mathbf A}

\renewcommand{\O}{O}
\newcommand{\ta}{\mathbf t}

\definecolor{pakistangreen}{rgb}{0.0, 0.4, 0.0}
\newcommand{\eps}{\epsilon}	
\newcommand{\crad}{\operatorname{CR}}

\renewcommand{\d}{{d}}

\newcommand{\avelio}[1]{{\color{red} #1}}

\begin{document}
\title{Extremal distance and conformal radius of a CLE$_4$ loop}
\date{ }

\author{Juhan Aru \and Titus Lupu \and Avelio Sep\'ulveda}
	
	\address {
		Institute of Mathematics,
		EPFL,
		1015 Lausanne,
		Switzerland}
	\email
	{juhan.aru@epfl.ch}
	
	\address{Univ Lyon, Université Claude Bernard Lyon 1, CNRS UMR 5208, Institut Camille Jordan, 69622 Villeurbanne, France}
	\email
	{sepulveda@math.lyon-1.fr}
	
	\address{CNRS and LPSM, UMR 8001, Sorbonne Université, 4 place Jussieu, 75252 Paris cedex 05, France}
	\email
	{titus.lupu@upmc.fr}


\begin{abstract}
	Consider CLE$_4$ in the unit disk and let $\ell$ be the loop of the CLE$_4$ surrounding the origin. Schramm, Sheffield and Wilson determined the law of the conformal radius seen from the origin of the domain surrounded by $\ell$. We complement their result by determining the law of the extremal distance between $\ell$ and the boundary of the unit disk. More surprisingly, we also compute the joint law of these conformal radius and extremal distance. This law involves first and last hitting times of a one-dimensional Brownian motion. Similar techniques also allow us to determine joint laws of some extremal distances in a critical Brownian loop-soup cluster. 
	\end{abstract}
	
\maketitle

\section{Introduction}

The Conformal Loop Ensembles CLE$_{\kappa}$, $\kappa\in (8/3,4]$ form a one-parameter family of random collections of simple loops in a simply connected planar domain that are conformally invariant in law. They were introduced in 
\cite{SheffieldCLE09,SheffieldWerner2012CLE} as the conjectural scaling limits of the interfaces for various statistical physics models. This conjecture has been confirmed only for some values of the parameter $\kappa$. In particular for $\kappa=4$, it is known that the CLE$_{4}$ appears in the scaling limit of the outer boundaries of outermost sign components of a metric graph GFF \cite{Lupu2015ConvCLE} and thereby has a natural coupling to the 2D continuum Gaussian free field (GFF)
\cite{MS,WaWu,ASW}. Furthermore, the CLE$_4$ is conjectured to be scaling limit of loops 
in the double-dimer model 
\cite{Kenyon14DoubleDimers,Dubedat19DoubleDimers,
BasokChelkak18taufunc}
and in the loop $O(n)$ model with parameters 
$n=2, x=1/\sqrt{2}$
\cite{KagerNienhuis04,PeledSpinka17On}.

The main result of this paper concerns the geometry of the loop of a CLE$_4$ surrounding the origin. We compute the joint law of the extremal distance from this loop to the boundary of the domain, together with the conformal radius seen from the origin of the domain surrounded by the loop.
The conformal radii of CLE$_\kappa$ for 
$\kappa\in (8/3,4]$ have been already identified in
\cite{SheffieldCLE09}, by a different but related method, to first hitting times of particular diffusions. In \cite{SSW09CR} the explicit density of their laws has been further computed.
No identification of the law of the extremal distance has appeared so far.
Our approach for $\kappa=4$ uses the coupling with the GFF.

The laws that appear are related to certain random times of the Brownian motion. For $B$ a standard Brownian motion, define
\begin{align}
\label{Eq T tau}
&T:=\inf\{t\geq 0: |B_t|=\pi\},\\
&\tau:=\sup\{0\leq t \leq T: B_t=0\}.
\nonumber
\end{align}
Given two closed sets $C_1$, $C_2$, denote the extremal distance between them by $\ED(C_1,C_2)$. See Section 
\ref{SubsecED}
for details. The main theorem of this paper is as follows (see also Figure \ref{Fig1}).
\begin{thm}\label{t. main}
	Let $\ell$ be the loop of a CLE$_4$ in the unit disk $\D$ surrounding the origin. Let $\crad(0, \D\backslash \ell)$ be the conformal radius of the origin in the domain surrounded by $\ell$. Then, the law of 
$(2\pi \ED(\ell,\partial \D) , -\log\crad(0, \D\backslash \ell) )$ equals that of $(\tau,T)$.
\end{thm}
\begin{figure}[h!]
	\includegraphics[height=0.25\textwidth]{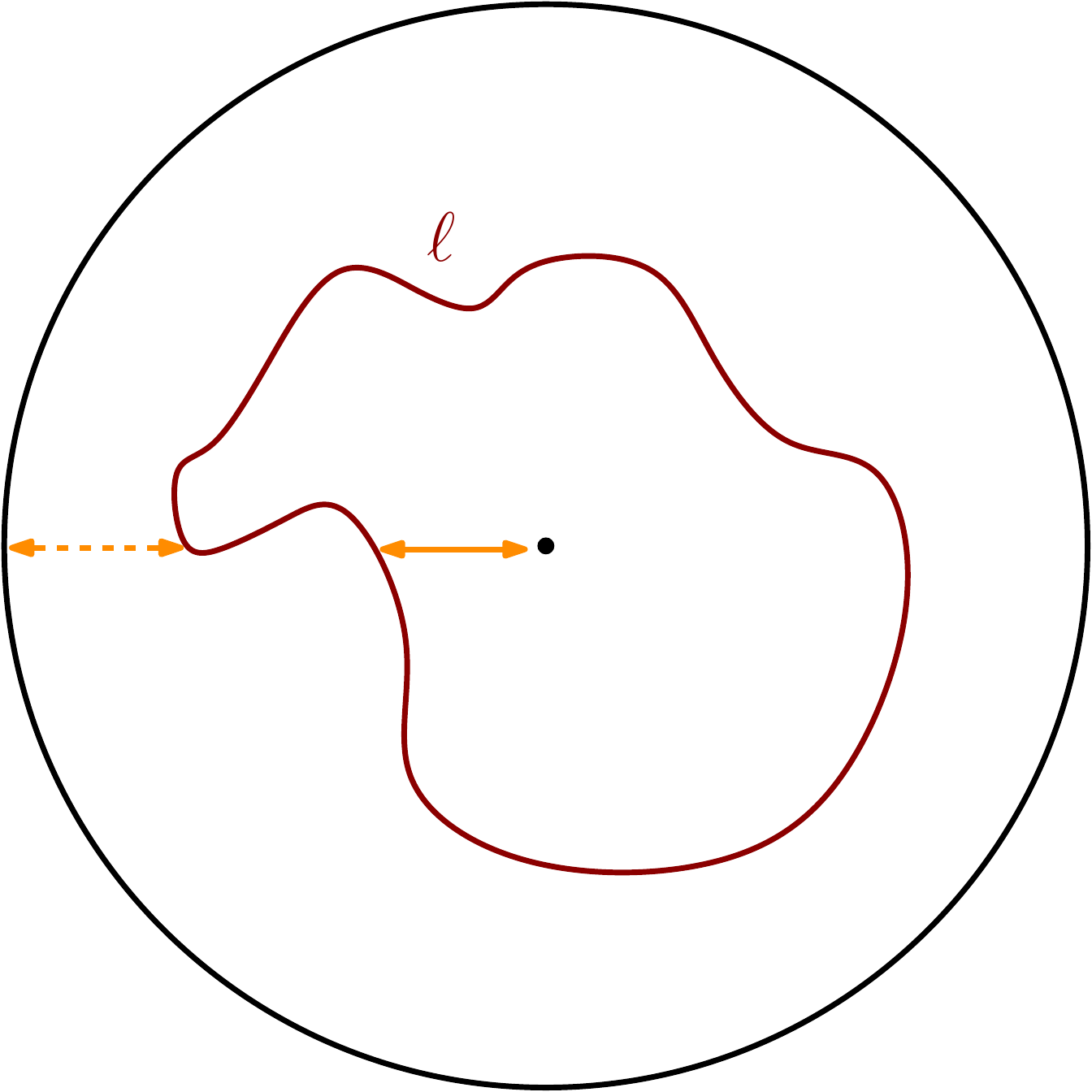}
	\includegraphics[height=0.18\textheight]{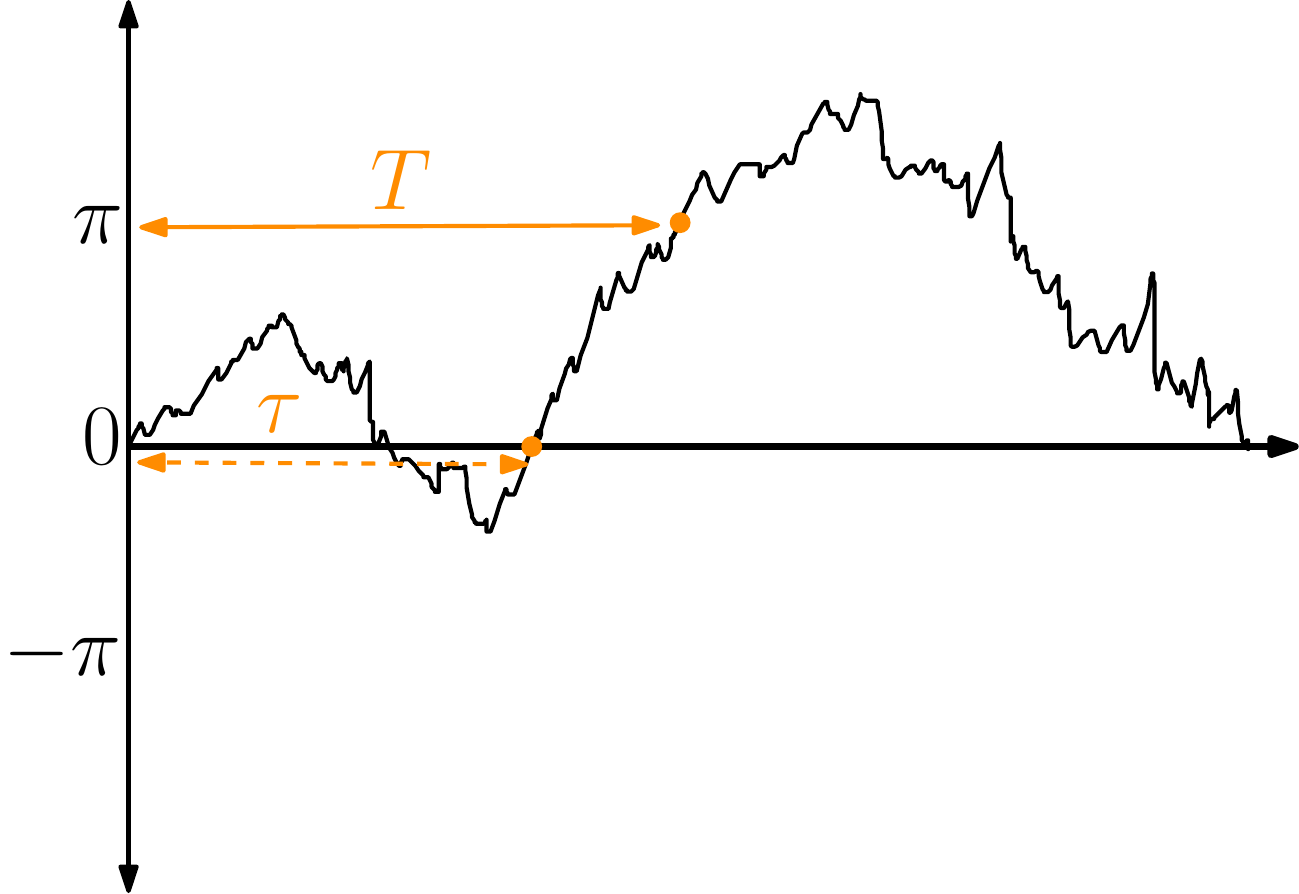}
	\caption{Graphic explanation of Theorem \ref{t. main}. The dashed line represents $2\pi \ED(\ell,\partial \D)$ and the continuous lines represents $-\log\crad(0, \D\backslash \ell)$.}
	\label{Fig1}
\end{figure}

Note that in particular a.s.
$2\pi \ED(\ell,\partial \D)\leq-\log\crad(0, \D\backslash \ell)$. This is actually a deterministic inequality, satisfied by any simple loop disconnecting $0$ from $\partial\D$
(see Corollary \ref{CorCrEDineq}), 
and the equality holds if and only if the loop is a circle centered at $0$.

Using standard distortion bounds one can deduce from 
Theorem \ref{t. main} some information on the size and shape 
the CLE$_{4}$ loop $\ell$ surrounding $0$, and in particular how far it is from being a circle centered at $0$. 
\begin{cor}
\label{CorAnnularExp}
Let $\ell$ be the loop of a CLE$_4$ in the unit disk $\D$ surrounding the origin.
Denote
\begin{align*}
r_{-}(\ell):= d(0,\ell),
	\qquad
	r_{+}(\ell):= \max\{ \vert z\vert : z\in\ell\}.
\end{align*}
Then one has the following exponents:
\begin{align*}
\lim_{R\to +\infty}
\dfrac{\log \mathbb{P}(r_{+}(\ell)^{-1}>R)}{\log R}
=
\lim_{R\to +\infty}
\dfrac{\log \mathbb{P}(r_{-}(\ell)^{-1}>R)}{\log R}
= - \dfrac{1}{8},
\end{align*}
and
\begin{align}
\label{Eq exp ratio r+ r-}
\lim_{R\to +\infty}
\dfrac{\log \mathbb{P}(r_{+}(\ell) / r_{-}(\ell)>R)}{\log R}
= -\dfrac{1}{2}.
\end{align}
Moreover, \eqref{Eq exp ratio r+ r-} also holds if $\ell$ is distributed according to the stationary CLE$_{4}$ distribution in $\C$ introduced in
\cite{KemppainenWerner16NestedCLE}.
\end{cor}

The exponent for $r_{-}(\ell)$ has been known since \cite{SSW09CR} and is related to the dimension of the CLE$_4$ gasket.
The exponents for $r_{+}(\ell)$ and $r_{+}(\ell) / r_{-}(\ell)$ have not appeared in other works, however it is possible that they may also be computable by different methods.

A generalization of Theorem \ref{t. main}
describes the geometry of a
cluster of a Brownian loop-soup
\cite{LW2004BMLoopSoup} involved in the construction of the
CLE$_{4}$ as in \cite{SheffieldWerner2012CLE}. To state the theorem, let us introduce two further random times related to the Brownian motion $B$
\begin{align*}
&\overline T:=\inf\{t\geq T: B_t=0 \},\\
&\overline \tau :=\sup\{0\leq t \leq \overline T- T: B_{t+T}=B_{T}\}.
\end{align*}
We can now state our second main theorem (see also Figure \ref{Fig2}).
\begin{thm}\label{t.loop soup}
	Take a critical Brownian loop-soup in $\D$ and let $\mathcal C$ be the outer-most loop-soup cluster surrounding the origin. Let $\ell^o$ and $\ell^i$ denote the outer and inner boundary of $\mathcal C$ w.r.t. the origin, respectively. Then, 
	\begin{enumerate}
\item		The quadruple  $(2\pi \ED(\partial \D,\ell^o), - \log \crad(0, \D\backslash \ell^i),2\pi \ED(\ell^o,\ell^i), - \log \crad(0, \D\backslash \ell^i)) )$ is equal in law to $(\tau, T, \overline \tau, \overline T)$.
\item Furthermore,  the quadruple $(2\pi \ED(\partial \D,\ell^o), 2\pi \ED(\ell^o,\ell^i), 2\pi \ED(\partial \D,\ell^i), -\log \crad(0, D\backslash \ell^i) )$ has the same law as $(\tau,\overline \tau ,  \overline \tau+T, \overline T)$.
\end{enumerate}
\end{thm}
\begin{figure}[h!]
	\includegraphics[height=0.25\textwidth]{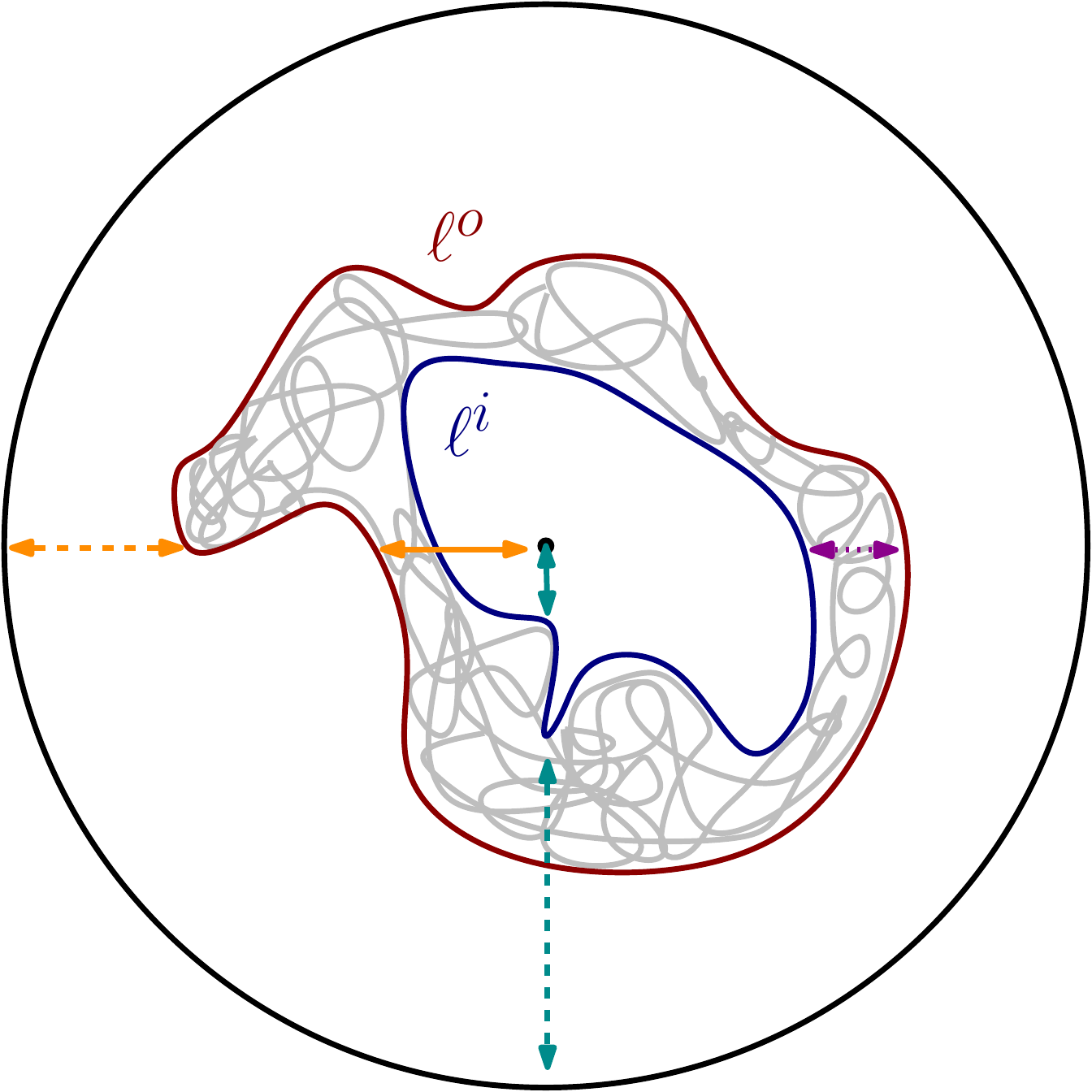}
	\includegraphics[height=0.18\textheight]{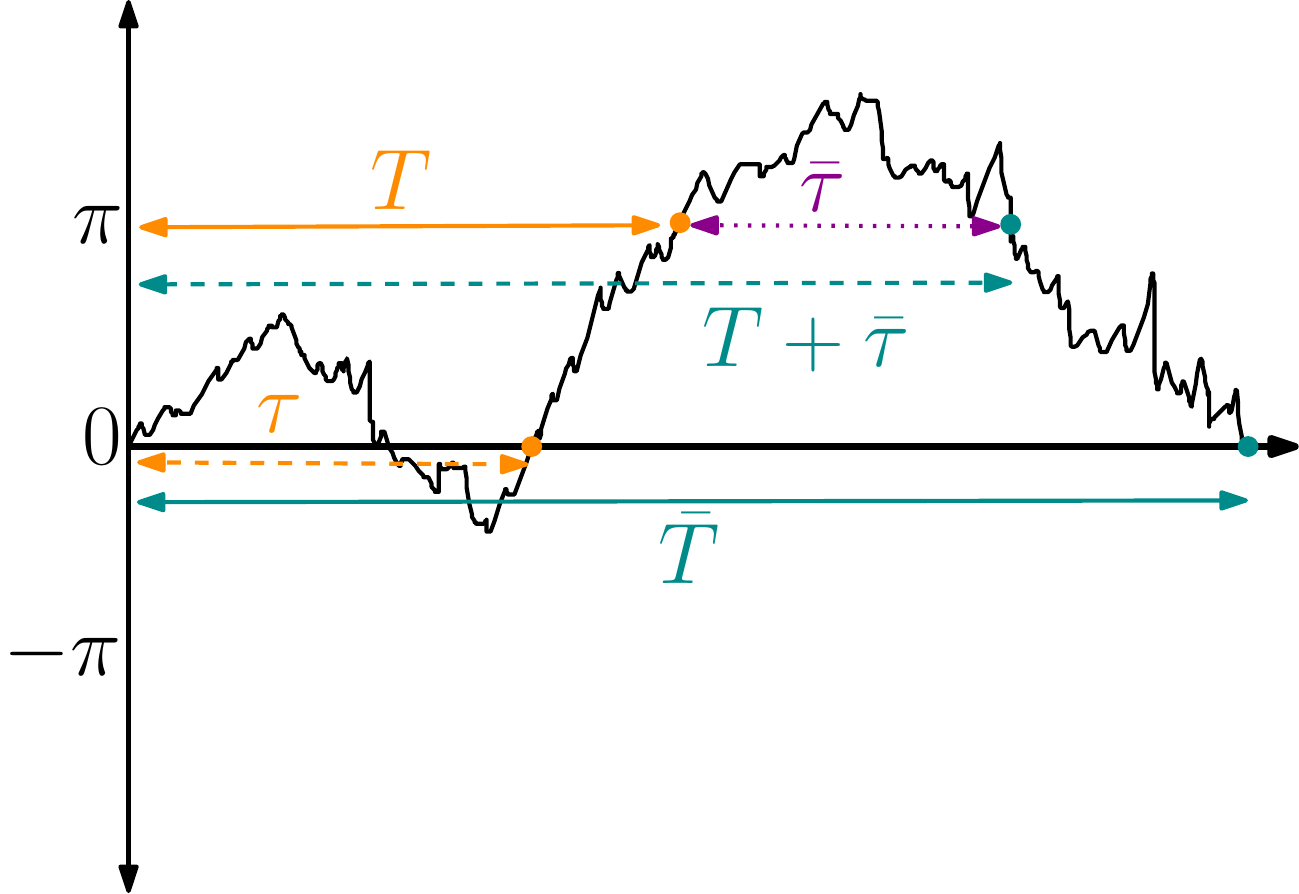}
	\caption{Graphic explanation of Theorem \ref{t.loop soup}. The dashed orange, resp. cyan, line represents $2\pi \ED(\ell^o,\partial \D)$, $2\pi \ED(\ell^i,\partial \D)$,  the continuous  orange, resp. cyan, line represents $-\log\crad(0, \D\backslash \ell^o)$, $-\log\crad(0, \D\backslash \ell^i)$ and the doted magenta line represents $2\pi\ED(\ell^o,\ell^i)$. Note that in this case, the joint law of $-\log(\crad(0, \D\backslash \ell^o))$ and $2\pi \ED(\ell^i,\partial \D))$ is not computed in Theorem \ref{t.loop soup}}.
	\label{Fig2}
\end{figure}

Notice that in both cases, we have only given the joint law for a quadruple of (reduced) extremal distances. The reason is the following:
\begin{itemize}
	\item The joint law of $(-\log \crad(0, \D\backslash \ell^o), 2\pi \ED(\partial \D,\ell^i)) )$ cannot be naturally obtained from the same Brownian motion $B$, i.e. it is not the same as the law of
$(T,\overline{\tau}+T)$. 
\end{itemize}
Indeed, if the joint law were equal to that of $(T, \overline \tau  + T)$, then $-\log \crad(0,\D\backslash \ell^o)< 2\pi \ED(\partial\D,\ell^i)$ would hold almost surely. However, this cannot be true as the following happens with positive probability: the conformal radius seen from $\ell^0$ is very small, hence
$-\log \crad(0,\D\backslash \ell^o)$ very large, and at the same time the extremal distance 
$\ED(\partial\D,\ell^i)$ is very small. See Figure \ref{f. no joint law} for an explanation.
\begin{figure}
	\includegraphics[height=0.25\textwidth]{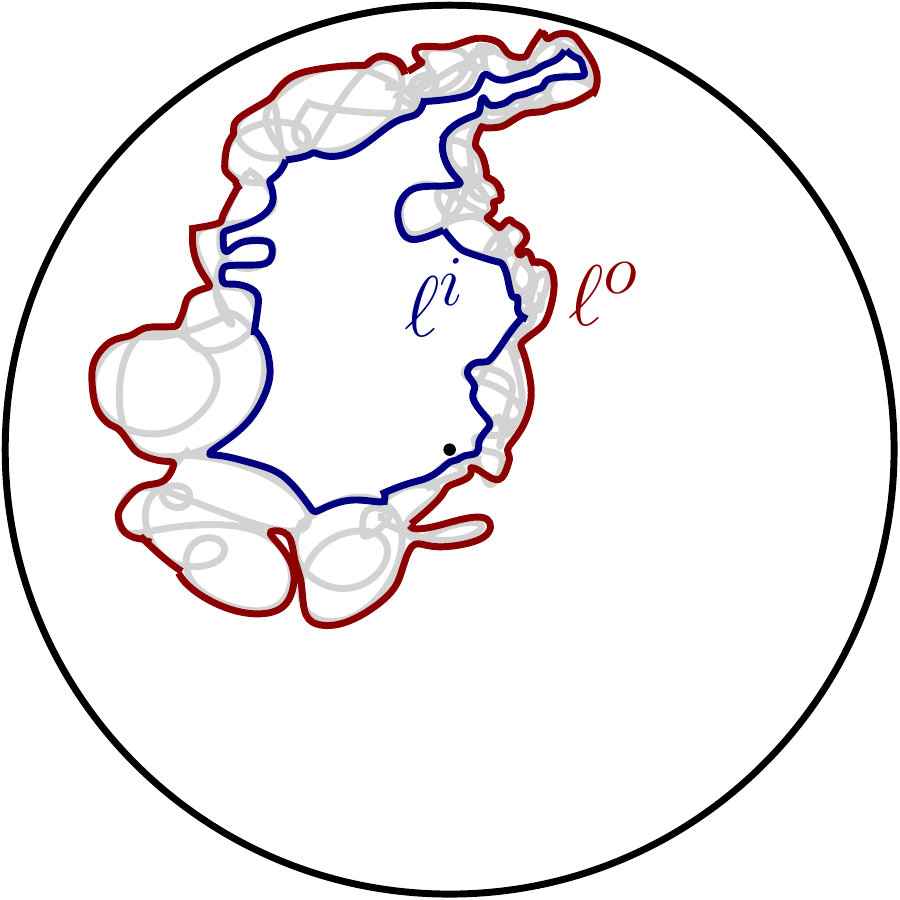}
	\caption{Main reason why we cannot compute the joint law of $(-\log \crad(0, \D\backslash \ell^o), 2\pi \ED(\partial \D,\ell^i)) )$: it is  possible that $\ell^i$ is arbitrarily close to the boundary, while $\ell^o$ is arbitrarily close to the center.}
	\label{f. no joint law}
\end{figure}

\bigskip

Let us now give an idea of the proof of this result. The first step is to couple the CLE$_{4}$ (Miller-Sheffield coupling) and the critical Brownian loop-soup to the continuum GFF. We will use the notion of local sets of the GFF \cite{SchSh2,WWln2} together with some recent developments in this topic \cite{ASW,ALS1,ALS2}. In this language CLE$_4$ is described by a two-valued local set, denoted $\A_{-2\lambda, 2\lambda}$. 

The connection between the 2D continuum GFF and the 1D Brownian motion is obtained from the exploration of the GFF by local sets. Indeed, let $(\eta_{t})_{t\geq 0}$ be a continuous growing family of local sets of a GFF in $\D$ and $h_{\eta_{t}}(0)$ be the harmonic extension in $0$ of the values of the GFF discovered by $\eta_{t}$. Then the process $(h_{\eta_{t}}(0))_{t\geq 0}$ is a continuous local martingale, and thus a time changed Brownian motion. Moreover, the natural time parametrization is such that at each time $t$, we have 
$-(2\pi)^{-1}\log\crad(z,\D\backslash \eta_{t})=t$.

In \cite{ASW}, the authors provide a construction of the local set $\A_{-2\lambda, 2\lambda}$ (that has the law of CLE$_4$) by a  growing family of local sets. This allows them to give a new proof of the fact that
$-\log(\crad(0, \D\backslash \ell) )$ has the law of the exit time of a Brownian motion \cite{ASW} from the interval $[-\pi,\pi]$, reproving the results
of \cite{SheffieldCLE09,SSW09CR} for $\kappa=4$ in a slightly different, but connected way: namely the Brownian motion entering the coupling with the GFF is very much related to the Brownian motion for the case $\kappa = 4$ in \cite{SheffieldCLE09,SSW09CR}.

The derivation of the law of 
$\ED(\ell,\partial \D)$ is considerably more involved than that of the $-\log(\crad(0, \D\backslash \ell) )$. To obtain it we cut a small hole around the origin and work with the GFF in an annular domain rather than in a disk. Then, the CLE$_{4}$ loop around the origin can be approximated, if the hole is sufficiently small, by an interface of the GFF in the annulus separating the inner and the outer boundary. The key point is that in an annulus the same interface can be discovered by explorations by local sets starting both from the outer and the inner boundaries. We call this property 
\textit{reversibility} and we deduce it as consequence of a certain commutativity property of two-valued local sets, reminiscent of commutativity properties of SLE curves coupled with the GFF, first studied in \cite{Dub, MS1}. An exploration from the inner boundary of the annulus, i.e. from the small hole around the origin, will give us the extremal distance of the interface to the outer boundary, i.e. $\partial\D$.

An additional step is needed to obtain the joint law of
$(2\pi \ED(\ell,\partial \D), -\log\crad(0, \D\backslash \ell) ) $. As before, we work in an annulus $\D\backslash r\overline \D$, and show that one can read the joint law of the pair $(2\pi \ED(\ell,\partial \D), 2\pi \ED(\ell,r\partial \D) )$ from the same Brownian motion (that is actually a Brownian bridge in the annulus). To do this, we first show that this joint law can be characterized by studying how it changes when changing the height of one endpoint of the Brownian bridge. We then study how the law of two-valued local sets of a GFF changes, when changing the boundary values of the field on $r\partial \D$, the inner boundary of the annulus, and deduce that the joint law of extremal distances for the loop $\ell$ satisfies the conditions of our characterisation. It should be noted that our argument does not extend to several nested interfaces, and indeed, as mentioned already above and explained on Figure \ref{f. no joint law}, in the case of several interfaces, one cannot simultaneously encode the joint law of all possible extremal distances and conformal radii on the same Brownian sample path, at least not in the obvious way.

We will further provide results analogous to Theorem
\ref{t. main} for some other interfaces induced by local sets of the GFF. Namely the gasket of the CLE$_{4}$ is a particular example of a two-valued local set, a general family of local sets first described in \cite{ASW} and further studied in \cite{AS2,QW,ALS1,SSV}. In Theorem
\ref{t. joint law TVS}, we provide a generalization of 
Theorem \ref{t. main} to the family of two-valued local sets $\A_{-a,b}$ with $a + b \in 2\lambda \N$, where $2\lambda$ is the height gap. A cluster of a critical Brownian loop-soup is related to the first passage sets of a particular level of the continuum Gaussian free field, as explained in \cite{ALS1,ALS2} . In Theorem \ref{t. law FPS} we give the joint law of the extremal distance and the conformal radius of an interface in a first passage set of any level.

\bigskip

Our article is organized as follows. In Section 
\ref{SecPrelim} we recall some notions of conformal geometry, basics of the GFF, its local sets and their relation to CLE$_4$ and the critical Brownian loop soup.
In Section \ref{SecRN} we give the Radon-Nikodym derivative of the two-valued local sets and first passage sets under a change of boundary values of the GFF or a change of domain. In Section \ref{S.Marginal laws} we compute the one-dimensional marginal laws of the extremal distances for a family of GFF interfaces in an annulus. The central step is in Subsection \ref{Ss.Reversibility}, where we prove the reversibility of interfaces.
Finally, in Section \ref{SecLaws} we deal with joint laws - we first characterize the joint laws in terms of the behaviour under a change of the boundary conditions, then we prove Theorems \ref{t. joint law TVS}, \ref{t. law FPS}, and \ref{t. main} in Subsection \ref{SubsecLawsFPSTVS}. 
Theorem \ref{t.loop soup} is proved in \ref{Subsec Thm_1_2}.
We finish by deriving Corollary \ref{CorAnnularExp}
in \ref{SebSec Ann_exp}.

\section{Preliminaries}
\label{SecPrelim}
In this section we collect some preliminaries on annular domains, GFF and its local sets. To start of let us also recall the notion of conformal radius: in any simply-connected domain\footnote{For us a domain is always a connected open subset of $\C$.} $D \subsetneq \C$, there is a conformally  invariant way to measure the size of $D$ seen from an interior point $z$. This notion of size is called the conformal radius of $D$ seen from $z$, is denoted by $\crad(z,D)$ and is defined by:
\begin{align*}
\crad(z,D) := \phi'(z),
\end{align*}
where $\phi: D \to \D$ is a conformal map with $\phi(z) = 0$ and $\phi'(z) > 0$. Such a map exists by the Riemann mapping theorem.
One has the bounds
\begin{align}
\label{EqKoebeQuarter}
\dfrac{1}{4}\crad(z,D) \leq d(z,\partial D)\leq \crad(z,D) .
\end{align}
The upper bound comes from simple monotony. 
The lower bound is the Koebe's one-quarter theorem
(Theorem 5-3 in \cite{Ahlfors2010ConfInv} and the subsequent Corollary).

	\subsection{Annular domains and extremal distance}
	\label{SubsecED}
	
Let $D\subset\C$ be a bounded open two-connected (i.e. with one hole) subset of $\C$. We also assume that the hole
(i.e. the bounded connected component of $\C\backslash D$) is not reduced to a point. Such a $D$ is called an annular domain. We will refer to the boundary of the bounded connected component of $\C \backslash D$ as the inner boundary of $D$ and denote it by $\partial_i$, the other boundary of $D$ is called the outer boundary and denoted $\partial_o$.

For any $0 < r < 1$ we denote by $\An_r$ the open annulus $\An_r := \D \backslash r\overline\D$; we also denote $\An=\An_r$ when we do not need to make reference to $r$. By the Riemann mapping theorem, for any annular domain there exits a unique $0 <r < 1$ and a conformal map such that $\phi: D \to \An_r$ mapping the inner boundary to $\partial_i$ and the outer boundary to $\partial_o$. Notice that $\An_r$ also has a conformal automorphism $z\mapsto r/z$ swapping the boundaries. 

In this way the number $0 < r < 1$ measures the distance between the boundaries in a conformally invariant way and parametrizes the conformal type of an annular domain. In fact, it comes out that it is more convenient to work with the logarithm of $r$: indeed, if we define 
\[\ED( \partial \D,r\partial \D) := \frac{1}{2\pi}\log \frac{1}{r}, \]
then by conformal invariance this defines $\ED(\partial_i, \partial_o)$ for any annular domain. We have that

	\begin{thm}[Theorem 4-5 of \cite{Ahlfors2010ConfInv}]
		\label{thmEL}
		Let $D$ be an annular domain. The inverse of the extremal distance $\ED(\partial_{o},\partial_{i})^{-1}$ is given by the Dirichlet energy 
		$\int_D |\nabla \bar u|^2$ of the harmonic function $\bar u$ equal to $0$ on $\partial_{o}$ and $1$ on 
$\partial_{i}$. Furthermore, if $\partial_i$, resp. $\partial_o$, has a piecewise smooth boundary, then $\int_D |\nabla \bar u|^2$ is equal to $\int_{\partial_i}\partial_n \bar u $, resp. $-\int_{\partial_o}\partial_n \bar u$.
	\end{thm}
We refer the reader to the Section 4 of \cite{Ahlfors2010ConfInv}  for more details on extremal distance, its intrinsic definition using families of curves and its generalization called the extremal length.

\begin{lemma}
\label{LemCREDlim}
Let $\wp$ be a simple loop in $\D$, at positive distance from $0$, surrounding $0$. Denote $\crad(0,\D\backslash\wp)$ the conformal radius seen from $0$ of the domain surrounded by $\wp$. Then
\begin{align*}
\lim_{r\to 0}\big(\ED(\partial\D,r\partial\D)-
\ED(\wp,r\partial\D)\big)
= - \dfrac{1}{2\pi}
\log\crad(0,\D\backslash\wp),
\end{align*}
where in the limit $r$ is small enough so that
$\wp\subset \An_r$.
\end{lemma}

\begin{proof}
This is probably well-known, but for completeness the proof is given in Appendix B.\end{proof}

Next we state a superadditivity property of the extremal distance.
It is referred to as composition law in \cite{Ahlfors2010ConfInv}.
\begin{thm}[Theorem 4-2 of \cite{Ahlfors2010ConfInv}]
\label{ThmSuperAddED}
Let $D$ be an annular domain with outer boundary $\partial_o$ and inner boundary $\partial_i$. Let $\wp$ be a simple non-contractible loop in $D$, that is to say $\partial_o$ and $\partial_i$ are in different connected components of 
$\overline{D}\backslash\wp$. Then
\begin{align}
\label{EqSuperAddED}
\ED(\partial_o,\partial_i)\geq
\ED(\partial_o,\wp)+\ED(\wp,\partial_i).
\end{align}
\end{thm}

Now take $D = \An_r$. We can now use Lemma \ref{LemCREDlim} to rewrite the inequality \eqref{EqSuperAddED} 
in the limit $r\to 0$ to obtain:
\begin{cor}
\label{CorCrEDineq}
Let $\wp$ be a simple loop in $\D$, at positive distance from $0$, surrounding $0$. Denote $\crad(0,\D\backslash\wp)$ the conformal radius seen from $0$ of the domain surrounded by $\wp$. Then
\begin{align*}
\ED(\partial\D,\wp)
\leq - \dfrac{1}{2\pi}
\log\crad(0,\D\backslash\wp).
\end{align*}
\end{cor}

Let $\wp$ be a simple loop in $\D$, at positive distance from $0$, surrounding $0$. Denote
	\begin{align*}
	r_{-}(\wp):= d(0,\wp),
	\qquad
	r_{+}(\wp):= \max\{ \vert z\vert : z\in\wp\}.
	\end{align*}
	Similarly to \eqref{EqKoebeQuarter}, one has bounds for
	$r_{+}(\wp)$ involving the extremal distance
	$\ED(\partial\D,\wp)$.

	\begin{prop}
		\label{PropRplusED}
		For $\wp$ a simple loop in $\D$, at positive distance from $0$, surrounding $0$,
		\begin{align*}
		e^{-2\pi \ED(\partial\D,\wp)}\leq r_{+}(\wp)
		\leq 4 e^{-2\pi \ED(\partial\D,\wp)}.
		\end{align*}
	\end{prop}
	\begin{proof}
		This follows from Sections
	4-11 and 4-12 in \cite{Ahlfors2010ConfInv}.
	For completeness the derivations are given in Appendix B.
	\end{proof}
	The ratio $r_{+}(\wp)/r_{-}(\wp)$ parametrizes the shape of the minimal annulus centered at $0$ containing $\wp$. 
	By combining Proposition \ref{PropRplusED} with \eqref{EqKoebeQuarter},
	one gets bounds for this ratio in terms of $\ED(\partial\D,\wp)$
	and $\crad(0,\D\backslash\wp)$. The constant 16 that appears below may be non-optimal.
	\begin{cor}
		\label{CorRatio}
		For $\wp$ a simple loop in $\D$, at positive distance from $0$, surrounding $0$,
		\begin{align*}
		e^{-2\pi \ED(\partial\D,\wp)}\crad(0,\D\backslash\wp)^{-1}
		\leq
		\dfrac{r_{+}(\wp)}{r_{-}(\wp)}
		\leq 
		16 e^{-2\pi \ED(\partial\D,\wp)}\crad(0,\D\backslash\wp)^{-1}.
		\end{align*}
	\end{cor}

	\subsection{The continuum GFF}
	\label{SubsecGFF}
	
	The (zero boundary) Gaussian Free Field (GFF) in a domain $D$ 
	is a centered Gaussian process $\Phi$ (we also sometimes write  $\Phi^D$ when we the domain needs to be specified) 
	indexed by the set of continuous functions with compact support in $D$, with covariance given by 
	$$ \E [(\Phi,f_{1}) (\Phi,f_{2})]  =  
	\iint_{D\times D} f_{1}(z) G_D(z,w) f_{2}(w) \d z \d w, $$ 
	where $G_D$ is the Green's function of Laplacian (with Dirichlet boundary conditions) in $D$, normalized such that $G_D(z,w)\sim (2\pi)^{-1} \log(1/|z-w|)$ as $z \to w$. In other words, the Gaussian free field is the normal random variable whose Cameron-Martin space is $H_0^1$, the completion of the set of smooth compactly supported functions for the inner product given by
	\[(f,g)_{\nabla}:=\int_{D} \nabla f (z) \nabla g(z) dz.\]

	For this choice of normalization of $G_D$ (and therefore of the GFF), we set\footnote{Sometimes, other normalizations are used in the literature: if ${G_D (z,w) \sim c \log(1/|z-w|)}$ as $z \to w$, then $\lambda$ should be taken to be $(\pi/2)\times \sqrt {c}$.}
	\begin{equation}\label{eq::lambda}
		\lambda=\sqrt{\pi/8}.
	\end{equation}
	It can be shown that the GFF has a version that lives in some space of generalized functions (Sobolev space $H^{-1}$), which justifies the notation $(\Phi,f)$ for $\Phi$ tested against the function $f$ (see for example \cite{Dub}).
	
	In this paper, $\Phi$ always denotes the zero boundary GFF. We also consider GFF-s with non-zero Dirichlet boundary conditions - they are given by $\Phi + u$ where $u$ is some bounded harmonic function.

	\subsection{Local sets: definitions and basic properties}
	\label{SubsecLocSet}
	Let us now introduce more thoroughly the local sets of the GFF. We only discuss items that are directly used in the current paper. For a more general discussion of local sets we refer to 
	\cite{SchSh2,WWln2}.
	
	\begin{defn}[Local sets]\label{d. local sets}
		Consider a random triple $(\Phi, A,\Phi_A)$, where $\Phi$ is a  GFF in $D$, $A$ is a random closed subset of $\overline D$ and $\Phi_A$ a random distribution that can be viewed as a harmonic function when restricted to
		$D \backslash A$.
		We say that $A$ is a local set for $\Phi$ if conditionally on $(A,\Phi_A)$, $\Phi^A:=\Phi - \Phi_A$ is a  GFF in $D \backslash A$. 
		\end {defn}
		
		Throughout this paper, we use the notation $h_A: D\rightarrow\R$ for the function  that is equal to $\Phi_A$ on $D\backslash A$ and $0$ on $A$.

		Let us list a few properties of local sets (see for instance \cite {SchSh2,Aru,AS} for derivations and further properties). 
		\begin{lemma}\label{BPLS}    $\ $
			\begin {enumerate}
			\item Any local set can be coupled in a unique way with a given GFF: Let $(\Phi,A,\Phi_A,\widehat \Phi_A)$ be a coupling, where $(\Phi,A,\Phi_A)$ and $(\Phi,A,\Phi'_A)$ satisfy the conditions of this definition. Then, a.s. $\Phi_A= \Phi'_A$. Thus, being a local set is a property of the coupling $(\Phi,A)$, as  $\Phi_A$ is a measurable function of $(\Phi,A)$. 
			\item If $A$ and $B$ are local sets coupled with the same GFF $\Phi$, and $(A, \Phi_A)$ and $(B, \Phi_B)$ are conditionally independent given $\Phi$, then $A \cup B$ is also a local set coupled with $\Phi$ and the boundary values of $\Phi_{A \cup B}$ agree with those of $\Phi_B$ or $\Phi_A$ at every point of the boundary of $A \cup B$ that is of positive distance of $A$ or $B$ respectively\footnote{We say that $\Phi_{A \cup B}$ agrees with $\Phi_A$ at a point $x\in \partial(A\cup B) \cap \partial A $ if for any sequence of $x_n \notin A\cup B$ converging to $x$, $\Phi_{A \cup B}(x_n)-\Phi_{A} (x_n) \to 0$ as $n\to \infty$.}. Additionally, $B\backslash A$ is a local set of $\Phi^A$ with $(\Phi^A)_{B\backslash A} = \Phi_{B\cup A}-\Phi_{A}$  . 
			\item Let $(\Phi, (A_n)_{n\in \N},(\Phi_{A_n}))_{n\in \N}$ a sequence of conditionally independent local sets coupled with the same GFF $\Phi$. Furthermore, assume that $A_n$ is increasing. Then $A_\infty = \overline{\bigcup_{n\in \N} A_n}$ is a local set. Furthermore, if a.s. for all $n\in \N$, $A_n$ is connected to the boundary, then a.s. $\Phi_{A_n}\to \Phi_{A}$.
		\end{enumerate}
	\end{lemma}

	We will now argue that a local set remains a local set under a Cameron-Martin shift of the underlying field.  Let $g$ a function in the Sobolev space $H_0^1(D)$ and $\Phi$ be a Gaussian free field under the probability measure $\P$. Define
	\begin{equation}\label{e.change measure}
	d\widetilde \P= \exp\left ((\Phi, g)_\nabla-\frac{1}{2}(g,g)_\nabla\right )d\P.
	\end{equation}
	According to the Cameron-Martin theorem, under the measure $\widetilde \P$, $\Phi-g$ is a GFF. Something similar is true for local sets.
	\begin{thm}\label{t. change of measure local sets}
		Let $(\Phi,A,\Phi_A)$ be a local set coupling under the measure $\P$. Then, under the law $\widetilde \P$, $(\Phi-g, A, \Phi_A -g_A)$ is a local set coupling, where $g_A$  is the orthogonal projection in $H_0^1(D)$ of $g$ on the subspace of functions that are harmonic in $D\backslash A$.
	\end{thm}
    The proof of Theorem \ref{t. change of measure local sets} is essentially contained in Proposition 13 of \cite{ASW}. As that proposition is stated slightly differently, we provide the proof for the sake of completeness in Appendix B.
	\begin{rem}Note that the subspace of harmonic functions in $D\backslash A$ is closed in $H^1_0(D)$. Furthermore, when $A$ has empty interior and $g$ has a bounded trace on $A$, then $g_A$ is equal to the unique bounded harmonic function in $D\backslash A$ with boundary values $g$ on $A$ and $0$ on $\partial D$.
	\end{rem}

	\subsection{Parametrizing local sets}
	\label{SubSec Param Loc Set}
	
	Often one is interested in a growing family of local sets, which we call local set processes.
	
	\begin{defn}[Local set process]
		We say that a coupling $(\Phi,(\eta_t)_{t\geq 0})$ is a local set process if $\Phi$ is a GFF in $D$, $\eta_0\subseteq \partial D$, and $\eta_t$ is an increasing continuous family (for the Haussdorf topology) of local sets such that for all stopping time $\tau$ of the filtration $\F_t:=\sigma(\eta_s:s\leq t)$, 
		$(\Phi,\eta_\tau)$ is a local set.
	\end{defn}
	
	Let us note that in our definition $\eta_t$ is actually a random set. In other words, in our notation $\eta_t=\eta([0,t])$. In the rest of the paper, we are mostly interested in local set processes that (mostly) evolve as continuous curve. In those cases, if appropriate, we denote by $\eta(t)$ the tip of the curve at time $t$. 
	
	In a disk, local processes can be naturally parametrized using the conformal radius. In an annulus local set process can be parametrized by its extremal distance to a whole boundary component. 
	
	\begin{prop}[Proposition 2.7 of \cite{ALS1}] \label{BProcess} Let $D$ be an annular domain and $(\Phi,\eta_t)$ be a local set process with $\Phi $ a GFF in $D$. Then, if $\eta_t$ is parametrized by the inverse of its extremal distance, i.e., 
	\[t=\ED(\B, \partial D \backslash \B)-\ED(\B, \mathcal (\partial D \cup \eta_t) \backslash \B),\]
	 the process 
	\[\widehat{B}_t := \ED(\B, \mathcal (\partial D \cup \eta_t) \backslash \B)\int_{\B} \partial_n h_{\eta_t},\]
	 has (a modification with) the law of a Brownian bridge from 0 to 0 with length $\ED(\B, \mathcal \partial D \backslash \B)$.
	\end{prop}
		\subsection{Two-valued local sets}
		\label{SubsecTVS}
		First, it is convenient to review a larger setting, that of bounded type local sets (BTLS) introduced in \cite{ASW}. These sets are thin local set $A$, for which its associated harmonic function $h_A$ remains bounded. 
		Let us introduce the definition of a thin 
local set when $h_A$ is integrable\footnote{For the general definition of a thin set, see \cite{Se}. }.
		\begin{defn}
We say that a local set $A$ is thin, if $h_A$ belongs to $\mathbb L^1(D \backslash A)$ and for any smooth function $f$
\begin{align*}
(\Phi_A,f)= \int_{D\backslash A} h_A(x)f(x)dx.
\end{align*}
		\end{defn}
			The following proposition provides a sufficient condition to show that a local set $A$ is thin.
		\begin{prop}[Proposition 4.3  of \cite{Se}]\label{p.thin} Let $A$ be a local set.
 If $h_A$ is $\mathbb L^1(D\backslash A)$  and for any compact set $K\subseteq D$, the Minkowski dimension of $A\cap K$ is strictly smaller than 2, then $A$ is thin. 
		\end{prop} 		 
		Now, we can define the bounded type local sets.		 
		\begin{defn}[BTLS]\label{BTLSCND}
			Consider  a  closed subset $A$ of $\overline D$  and $\Phi$ a GFF in $D$ defined on the same probability space.
			Let $K>0$, we say that $A$ is a $K$-BTLS for $\Phi$ if the following four conditions are satisfied: 
			\begin {itemize}
			\item $A$ is a thin local set of $\Phi$.
			\item Almost surely, $|h_A| \le K$ in $D \backslash A$.  
			\item Almost surely, each connected component of $A$ that does not intersect $\partial D$ has a neighborhood that does not intersect any other connected component of $A$. 
			\end {itemize}
			If $A$ is a $K$-BTLS for some $K$, we say that it is a BTLS.
			\end {defn}

	\begin{figure}[ht!]	   
		\centering
		\includegraphics[width=5.0in]{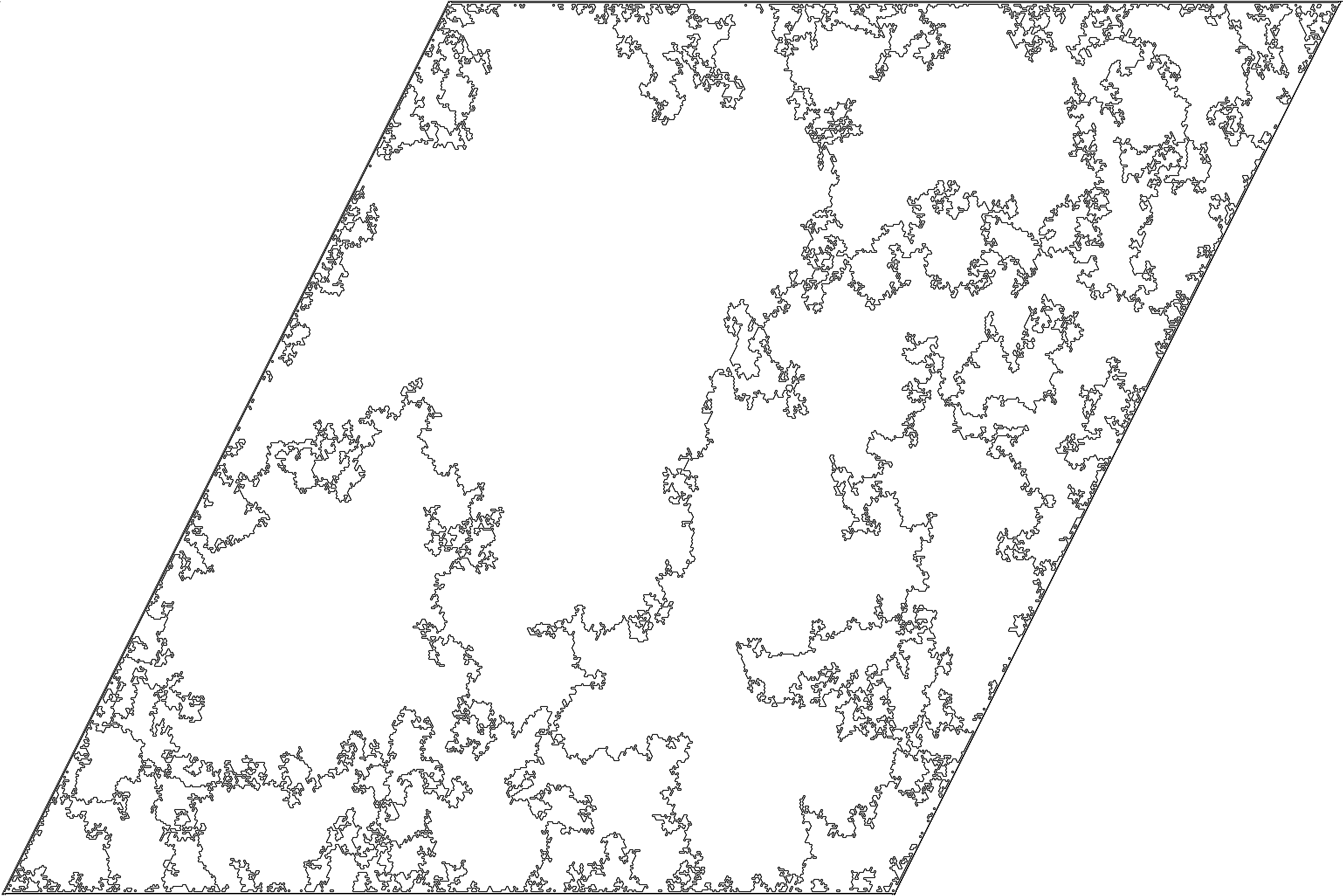}
		\caption{Simulation of $\A_{-\lambda,\lambda}$ done by B. Werness.}\label{SimALE}
	\end{figure}
	
	One family of useful BTLS is that of two-valued local sets. In \cite{ASW}, two-valued local sets of the zero boundary GFF were introduced in the simply connected case. In \cite{ALS1} two-valued local sets were generalized to multiply-connected domains and to more general, piece-wise constant boundary values. 
	
	So let $u$ be a bounded harmonic function whose boundary values are constant in each connected component of the boundary $\partial D$. We denote $u^{-a,b}$ the connected components of $\partial D$ where the values of $u$ are outside of $[-a,b]$.The two-valued set $\A^{u}_{-a,b}$ in n-connected domains is then a BTLS such that in each connected component $O$ of $D\backslash \A_{-a,b}^u$ the bounded harmonic function $h_{\A_{-a,b}^u}$ satisfies the following \hypertarget{tvs} {conditions:} 
	\begin{enumerate}
		\item[(\twonotes)] On every boundary component of $\partial O\backslash u^{-a,b}$ the harmonic function $h_{\A_{-a,b}^u}+u$ takes constant value $a$ or $-b$, and in $\partial O \cap u^{-a,b}$ it takes the value $u$. 
		\item[(\twonotes \twonotes)] Additionally, we require that in every connected component of $\partial O$  either $h_{\A_{-a,b}^u}+u\leq -a$ or $h_{\A_{-a,b}^u}+u\geq b$ holds. 
	\end{enumerate}

	The next theorem gives the existence and uniqueness of the TVS (in fact more general boundary conditions are allowed in \cite{ALS1}, but we state it in a form relevant to this paper).
	\begin{thm}[Theorem 3.21 of \cite{ALS1}]
	\label {t.tvs}
	Consider a bounded harmonic function $u$ equal to a constant on any boundary component of $\partial D$ as above. If $[\min(u),\max(u)] \cap (-a,b) \neq \emptyset$ and  $a+b \geq 2\lambda$, then  it is possible to construct $\A_{-a,b}^u\neq \emptyset$ coupled with a GFF $\Phi$ . Moreover, the sets $\A_{-a,b}^u$ are
	\begin{itemize} 
		\item unique in the sense that if $A'$ is another BTLS coupled with the same $\Phi$,  such that a.s. it satisfies the conditions above,	then $A' = \A^u_{-a, b}$ almost surely;  
		\item measurable functions of the GFF $\Phi$ that they are coupled with;
		\item 	monotone in the following sense: if $[-a,b] \subset [-a', b']$ with $b+a \ge 2\lambda$,  then almost surely, $\A^u_{-a,b} \subset \A^{u}_{-a', b'}$. 
		
	\end{itemize}
	\end {thm}
	
	Let us now concentrate on TVS in the annulus. We claim, that in this case TVS is either connected and all the components of its complement are simply-connected, or it has two connected components, one corresponding to each boundary component.
	\begin{cor} \label{c.tvs from a boundary}Consider an annulus $\An_r$ and $a,b$ as in the previous statement. Define $\A_{-a,b}^{u,\partial_o}$ as the connected component of $\A_{-a,b}^u\cup \partial \An_r$ that contains $\partial_o$ on its boundary. 	Then $\A_{-a,b}^{u,\partial_o}$ is a BTLS satisfying condition (\twonotes). Moreover, if $\A_{-a,b}^{u,\partial_o}$ touches the boundary, it is equal to $\A_{-a,b}^u$. A similar claim holds for $\A_{-a,b}^{u,\partial_i}$. 
	\end{cor}

	\begin{proof}
	Both claims follow directly from the construction of TVS in non-simply connected domains in Theorem 3.21 of \cite{ALS1}. 
	\end{proof}
	It is possible to calculate explicitly the extremal distance in the annular component of $\An_r \backslash \A_{-a,b}^{u,\partial_o}$ - i.e. the extremal distance between the non-trivial loop of $\A_{-a,b}^{u,\partial_o}$ and $\partial_i$. This comes from a slightly simplified version of Proposition  4.13 in \cite{ALS1}.
	
	\begin{prop} \label{p.LawELBddTVS}
		Let $a, b$ be positive with $a+b \geq 2\lambda$ , $\An_r$ be an annulus. Let $u_v$ be a bounded harmonic function equal to a constant $v$ on the inner boundary $\partial_i$ and $0$ on the outer boundary $\partial_o$. Let $\widehat{B}_t$ be a Brownian bridge from $0$ to $v$ with length $\ED(\partial_i, \partial_o)$. 
		Then 
		\[\ED(\partial_o, \partial_i) - \ED(\A_{-a,b}^{u_v,\partial_o} , \partial_i)\]
		is equal in law to the first hitting time of $\{-a,b\}$ by 
		$\widehat{B}_t$. Notice that if $\{-a,b\}$ is not hit, this time is equal to $\ED( \partial_o,\partial_i)$.
	\end{prop}

	In particular, for any $|a+b| \geq 2\lambda$, we see that the event $E$ that $\A_{-a,b}^{u_v,\partial_o}$ does not intersect the inner boundary has positive probability. On this event $\An_r \backslash \A_{-a,b}^{u_v,\partial_o}$ contains an annular component with $\partial_i$ on its boundary.

We will also need a result regarding the intersection of the connected components of $\An_r\backslash \A_{-a,b}^{u,\partial_o}$ with the outer boundary itself - i.e. this corresponds to the case when the extremal distance of the non-contractible loop of $\A_{-a,b}^{u,\partial_o}$ (when it exists) to the outer boundary $\partial_o$ is just $0$. Results of this type were studied in \cite{AS2} in the simply connected case, and the same ideas work also in the non-simply connected domains.
	\begin{prop}\label{p.intersection_boundary}
		Take $a\in \R$ and let $u$ be a bounded harmonic function with constant boundary values $\beta_o$ in $\partial_o$ with $-a\leq \beta_o\leq -a+2\lambda$. Then, for any connected component $O$ of $\An_r\backslash\A_{-a,-a+2\lambda}^{u,\partial_o}$ we have that $\partial O$ intersects the boundary of $\An_r$.
	\end{prop}
	\begin{proof}
		The proof follows directly from the level line construction found in Section 3.3 of \cite{ALS1}. See also Remark 3.2 of \cite{AS2} for the argument in the simply-connected setting.
	\end{proof}
	\subsection{CLE$_4$ as the two-valued local set $\A_{-2\lambda, 2\lambda}$}
	
	Maybe the shortest way to define CLE$_4$ in $\D$ is to define it as the collection of outer boundaries of the outermost clusters of a Brownian loop-soup at the critical intensity \cite{SheffieldWerner2012CLE}. However, for us the important and useful part is the connection of CLE$_4$ to SLE$_4$ process and to the two-valued local sets of the Gaussian free field.
	
	The latter connection was first discovered by Miller \& Sheffield \cite{MS}, based on the work of Schramm and Sheffield \cite{SchSh2}, and it says that CLE$_4$ can be coupled as a local set of the GFF. In \cite{ASW} this was rephrased in the language of two-valued sets - the set $\A_{-2\lambda,2\lambda}$ has the law of a CLE$_4$ carpet. 
	
	\begin{thm}\label{BPCLE}
		Let $\Phi$ be a GFF in $\D$ and $\A_{-2\lambda,2\lambda}$ be its TVS of levels $-2\lambda$ and $2\lambda$. Then $\A_{-2\lambda, 2\lambda}$ has the law of CLE$_4$ carpet. Moreover, it satisfies the following properties:
		\begin{enumerate}
			\item The loops of $\A_{-2\lambda,2\lambda}$ (i.e. the boundaries of the connected components of $\D \backslash \A_{-2\lambda, 2\lambda}$) are continuous simple loops. $\A_{-2\lambda, 2\lambda}$ is the closure of the union of all loops.
			\item The collection of loops of $\A_{-2\lambda,2\lambda}$ is locally finite, i.e. for any $\eps>0$  there are only finitely many loops that have diameter bigger than $\eps$.
			\item Almost surely no two loops of $\A_{-2\lambda,2\lambda}$ intersect, nor does any loop intersect the boundary.
			\item The conditional law of the labels of the loops of $\A_{-2\lambda,2\lambda}$ given $\A_{-2\lambda,2\lambda}$ is that of i.i.d random variables taking values $\pm 2\lambda$ with equal probability.
		\end{enumerate}
	\end{thm}
	
	\begin{rem}
		For the fact that $\A_{-2\lambda,2\lambda}$ has the law of a CLE$_4$, see Section 4 of \cite{ASW}. 
		Then the first three properties stem from the basic properties of CLE$_4$ (see \cite{SheffieldWerner2012CLE}).  For (4) see, for example, the last comment in Section 4.3 of \cite{ASW}.
	\end{rem}

	\subsubsection{Exploring the loops of $\A_{-2\lambda,2\lambda}$ using SLE$_4(-2)$}\label{ss.coupling rgamma}
	In Section 4 of \cite{ASW}, it is explained how to couple a radial SLE$_4(-2)$ in $\D$ targeted at fixed point $z$ with a GFF as a local set process that discovers loops of $\A_{-2\lambda, 2\lambda}$. We will not repeat this construction here and redirect the reader to \cite{ASW}. Instead, we will instead just summarize its trajectorial properties, and the properties of the coupling with the GFF needed in the current paper.
	
	\begin{prop}\label{p.GFFcpl}[Existence and properties of radial SLE$_4(-2)$] 
		Let $\Phi$ be a GFF in $\D$ and take $x\in \partial \D, z \in \D$. Then there is a random continuous (in Hausdroff topology) growth process $\eta: [0,\infty) \to \overline \D$ starting in $x$, coupled with a GFF $\Phi$ and defined up to some random time $0 < T_z < \infty$ such that:
		\begin{itemize}
			\item $z \notin \eta([0,T_z])$ a.s.
			\item For all $t\geq T_z$, $\eta(t_1)=\eta(T_z)$.
			\item For any rational time $q<T_z$, there exists a time $\underline q < q$ such that $\eta_q\backslash \eta_{\underline q}$ is a simple curve in the connected component of $\D\backslash \eta_{\underline q}$ containing the point $z$.
			\item The process $\eta_t$ is a local set process and thus for any fixed $t$, $\eta_t$ is a local set.
		\end{itemize}
		Furthermore, $h_{\eta_t}$ restricted to a connected component $O$ of $\D\backslash \eta_t$ is
		\begin{itemize}
			\item Constant equal to either $\{-2\lambda,0,2\lambda\}$, if $z\notin O$.
			\item Constant equal to either $\{-2\lambda,2\lambda\}$, when $z\in O$ and $t\geq T$.
			\item When $z\in O$ and $t<T_z$, take a rational time $q>t$ such that $\underline q\leq t$. If such a $q$ does not exists, then the harmonic function $h_{\eta_t}$ equals $0$ (then at $t$ the process $\eta_t$ has just finished a loop). If such a $q$ does exist, there are two cases
			\begin{itemize}
				\item The bounded harmonic function is taking constant value equal to $2\lambda$ to the right side of the simple curve $\eta_t\backslash \eta_{\underline q}$ and $0$ elsewhere.
				\item The bounded harmonic function is taking constant value equal to $-2\lambda$ to the left side of the simple curve $\eta_t\backslash \eta_{\underline q}$ and $0$ elsewhere.
			\end{itemize}			
		\end{itemize}
	\end{prop}

It is convenient to separate some further properties into a different proposition.
\begin{prop}\label{p.GFFcpl2}[Radial SLE$_2(-2)$ and the GFF]
	Let $\eta_t$ be as in Proposition \ref{p.GFFcpl}. We have further that $\eta_t \subseteq \A_{-2\lambda,2\lambda}(\Phi)$ for all $t\leq T_z$ and moreover each connected component of the complement of $\eta_{T_z}$ with boundary condition equal to $\pm 2\lambda$ is also a connected component of $\D \backslash \A_{-2\lambda,2\lambda}(\Phi)$. In the other direction, the loop of $\A_{-2\lambda, 2\lambda}(\Phi)$ around $z$ is a subset of $\eta_t$ for all $t \geq T_z$. 
	\end{prop}
	
	The evolution of this process when it is tracing a loop can be described using generalized level lines. Let us recall the definition here, and see Section 3.3 in \cite{ALS1} for more details. By conformal invariance the explicit choice of the two-connected domain below plays no role.
	
	\begin{defn}[Generalized level line]\label{deflevell}
		Let $D:=\H$ or $D:=\H\backslash B(x,r)$ for some $x\in \H$, $r<\operatorname{Im}(x)$. Further, let $u$ be a harmonic function in $D$. We say that $\eta(\cdot)$, a curve parametrized by half plane capacity, is the generalized level line for the GFF $\Phi + u$ in $D$ up to a stopping time $\ta$ if for all $t \geq 0$:
		
		\begin{description}
			\item[$(**)$]The set $\eta_{t}:=\eta[0, t \wedge \ta]$ is a BTLS of the GFF $\Phi$, with harmonic function satisfying the following properties: $h_{\eta_{t}} + u$ is a harmonic function in $D \backslash \eta_t$ with boundary values $-\lambda$  on the left-hand side of $\eta_t$, $+ \lambda$ on the right side of $\eta_t$, and with the same boundary values as $u$ on $\partial D$. 
		\end{description}
	\end{defn}
	
	From Proposition \ref{p.GFFcpl}, we obtain the following lemma.	
	\begin{lemma}\label{Rlevelline}
		We work in the context of Proposition \ref{p.GFFcpl}. At each time $t$ such that there is a rational $\underline q$ such that $\eta_t \backslash \eta_{\underline q}$ is a simple curve in $\D \backslash \eta_{\underline q}$ (i.e. at each time $t$ when $\eta_t$ is tracing a loop), $\eta_t$ is tracing a generalized level line.
	\end{lemma}
	
	The importance of working with generalized level lines comes from the fact that one can control the behaviour of such local set processes when they approach different boundaries. An example of such a result is Lemma 16 of \cite{ASW}, that we restate here and that is useful for us later on.	
	\begin{lemma}[Boundary hitting of generalized level lines, Lemma 16 \cite{ASW}] \label{notouch}
	Let $D$ be as above and $\eta$ be a generalized level line of $\Phi + u$ in $D$. Suppose $u \geq \lambda$ in $J\cap \partial D$, with $J$ some open set of  $D$. Let $\tau$ denote the first time at which $\inf \{ d ( \eta_s, J\cap \partial D), s< t \} = 0$. Then, the probability that $\tau < \infty$ and that $\eta(t)$ accumulates in a point in $J$ as $t \to \tau^-$ is equal to $0$. This also holds if $u \leq -\lambda$ in $J\cap \partial D$ for $J$ an open set of $D$.
	\end{lemma}
	
	\subsubsection{A construction of the non-contractible loop of $\A_{-2\lambda,2\lambda}^{\partial_o}$ in an annulus}\label{ss.construction SLE}
	In this section, we provide a construction of the non-contractible loop of $\A^{\partial_o}_{-2\lambda,2\lambda}$ in an annulus $\An$ using an SLE$_4(-2)$ type process starting from a point on $\partial_o$. By conformal invariance the same naturally holds when we replace $\partial_o$ by $\partial_i$.
	
	Now, fix some $\epsilon>0$. Then Corollary 14 in \cite{ASW} (restated in the current article as Proposition \ref{p.change measure boundary}) implies that via a change of measure argument, we can define a local set process $\nu$ starting from
	a point $x \in \partial_o$, stopped at the first time $\tau_\eps$ when it gets $\epsilon$-close to $\partial_i$, that has the same pathwise properties as the process $\eta$ of Proposition \ref{p.GFFcpl} targeted at $0$. 
	
	In particular, the boundary values of this local set process are the same as described by Proposition \ref{p.GFFcpl}. To recap - they are either $\pm 2\lambda$ inside any finished loops, $\pm 2\lambda$ and $0$ on any loop that is in the process of being traced, and zero elsewhere.
	
	We also have the equivalent of Proposition \ref{p.GFFcpl2}. Indeed, following the construction of TVS in Section 3.4 of \cite{ALS1}, having explored $\nu_{\tau_\eps}$, we can further explore level lines in every connected component of $\An \backslash \nu_{\tau_\eps}$ to complete the construction of a local set connected to $\partial_o$ and having boundary conditions equal to $\pm 2\lambda$. By uniqueness of (connected components of TVS), given in Corollary \ref{c.tvs from a boundary}, we can thus construct $\A_{-2\lambda, 2\lambda}^{\partial_o}$ starting from $\nu_{\tau_\eps}$. In particular, this implies that $\nu_{\tau_\eps} \subseteq \A_{-2\lambda, 2\lambda}^{\partial_o}$. Moreover, as we only explore further level lines in connected components of $\nu_{\tau_\eps}$ where the boundary conditions are not constant equal to $\pm 2\lambda$, we see that all loops of $\nu_{\tau_\eps}$ are also loops of $\A_{-2\lambda, 2\lambda}$. In particular, this is true also for a possible non-contractible loop of $\nu_{\tau_\eps}$ that would then be the unique non-contractible loop of $\A_{-2\lambda, 2\lambda}^{\partial_o}$.
		
	Furthermore, when we take $\epsilon\to 0$ there are one of two possibilities that may arise
	\begin{itemize}
		\item $\nu_t$ creates a loop separating $\partial_i$ from $\partial_o$. In this case, this is the only loop of $\A_{-2\lambda,2\lambda}^{\partial_o}$ that separates $\partial_i$ from $\partial_o$.
		\item $\nu_t$ intersects the boundary $\partial_i$. This can only happen if $\A_{-2\lambda, 2\lambda}$ connects $\partial_i$ and $\partial_o$.
	\end{itemize}
	
	Finally, as in this construction for any $\epsilon>0$, $(\nu_t)_{t \leq \tau_\eps}$ is absolutely continuous w.r.t. to the process $(\eta_t)_{t \leq \tau_\eps}$ of Proposition \ref{p.GFFcpl} started from $x \in \partial_o$. In particular, at each time $t$ for which $\ED(\eta_t, \partial_o) > 0$ we have that when $\nu_t$ is tracing a loop, it is tracing a generalized level line.	Let us combine all of the above in a proposition for further reference.	
	\begin{prop}\label{p.GFFcpl2ann} 
		
	In an annulus $\An$, and for any $x \in \partial_o$ there is a local set process $\nu_t$ of a GFF $\Phi$ starting from $x$, defined until a stopping time $\tau$ when it either finishes tracing a loop that separates $\partial_o$ from $\partial_i$ or intersects $\partial_i$. Moreover, we have that 
	
	\begin{itemize}
		\item $\nu_t \subseteq \A_{-2\lambda,2\lambda}^{\partial_o}(\Phi)$ for all $t\leq \tau$ and each connected component of the complement of $\eta_{T_z}$ with boundary condition equal to $\pm 2\lambda$ is also a connected component of $\D \backslash \A_{-2\lambda,2\lambda}(\Phi)$. 
		\item If $\A_{-2\lambda, \lambda}^{\partial_o}(\Phi)$ has a non-contractible loop separating $\partial_i$ and $\partial_o$, then this loop is a subset of $\eta_\tau$ and at $\tau$ the process $\eta$ finishes tracing this loop.
		\item When $\nu_t$ is tracing a loop, it is tracing a generalized level line.
	\end{itemize}
	\end{prop}

	\subsection{First passage sets of the 2D continuum GFF}
	\label{SubsecFPS}
	
	The aim of this section is to recall the definition of first passage sets introduced in \cite{ALS1} of the 2D continuum GFF, and state the properties that will be used in this paper.
	
	The set-up is as follows: $D$ is simply connected or annular domain and $u$ is a bounded harmonic function whose boundary values are constant in each connected component of the boundary.
	\begin{defn}[First passage set]\label{Def ES}
		Let $a\in \R$ and $\Phi$ be a GFF in $D$. We define the first passage set of $\Phi$ of level $-a$ and boundary condition $u$ as the local set of $\Phi$ such that $\partial D \subseteq \A^u_{-a}$, with the following properties:
		\begin{enumerate}
			\item Inside each connected component $O$ of $D\backslash \A_{-a}^u$, the harmonic function $h_{\A_{-a}^u}+u$ is equal to $-a$ on $\partial \A_{-a}^u \backslash \partial D$ and equal to $u$ on $\partial D \backslash \A_{-a}^u$ in such a way that $h_{\A_{-a}^u}+u \leq -a$. 		
			\item $\Phi_{\A^u_{-a}}-h_{\A_{-a}^u}\geq 0$, i.e., for any smooth positive test function $f$ we have 
			$(\Phi_{\A^u_{-a}}-h_{\A_{-a}^u},f) \geq 0$. 
		\end{enumerate}	
	\end{defn}
		
 The key result is the following.
	\begin{thm}[Theorem 4.3 and Proposition 4.5 of \cite{ALS1}]\label{Thm::FPS}For all $a\geq 0$,  the first passage set, $\A_{-a}^u$, of $\Phi$ of level -a and boundary condition $u$ exists and satisfies the following properties:
		\begin{enumerate}
			\item Uniqueness: if $A'$ is another local set coupled with $\Phi$ and satisfying Definition \ref{Def ES}, then a.s. $A'=\A_{-a}^u$.
			\item Measurability: $\A_{-a}^u$ is a measurable function of $\Phi$.
			\item Monotonicity: If $a\leq a'$ and $u\leq u'$, then $\A_{-a}^u\subseteq \A_{-a'}^{u'}$ .
		\end{enumerate}
	\end{thm}
	
	In this setup, there is also an analogue of Corollary \ref{c.tvs from a boundary}.
		\begin{cor} \label{c.FPS from a boundary} Consider an annulus $\An_r$. Define $\A_{-a}^{u,\partial_o}$ as the connected component of $\A_{-a}^u\cup \partial \An_r$ that contains $\partial_o$, then $\A_{-a}^{u,\partial_0}$ is a local set such that $h_{\A_{-a}^{u,\partial_o}}+u$ is the bounded harmonic function in $\An_r\backslash\A_{-a}^{u,\partial_o}$ with values $-a$ in $\partial \A_{-a}^{u,\partial_o}$ and $u$ in $\partial_i$. Furthermore, $\Phi_{\A_{-a}^{u,\partial_0}}-h_{\A_{-a}^{u,\partial_o}}$ is the positive measure $\Phi_{\A_{-a}^{u}}-h_{\A_{-a}^{u}}$ restricted to $\A_{-a}^{u,\partial_o}$. The same holds when we swap the roles of $\partial_i$ and $\partial_o$ and consider $\A_{-a}^{u,\partial_o}$. 
		\end{cor}
	
	The FPS in the annulus has a similar description to that of the TVS - its complement has at most one non-simply connected component that needs to be an annulus. Moreover, one can similarly calculate the extremal distance of the annulus in the complement of $\A_{-a}^{u,\partial_o}$.
			
	\begin{prop} \label{p.LawELBddFPS}
		Let $a> 0$, $\An_r$ be an annulus. Let $u_v$ be a bounded harmonic function equal to $v$ on the inner boundary $\partial_i$ and $0$ on the outer boundary $\partial_o$. Let $\widehat{B}_t$ be a Brownian bridge from $0$ to $v$ with length $\ED(\partial_i, \partial_o)$. 
		Then 
		\[\ED(\partial_o, \partial_i) - \ED(\partial_o \cup\A_{-a}^{u_v,\partial_0} , \partial_i)\]
		is equal in law to the first hitting time of $-a$ by 
		$\widehat{B}_t$. Notice that if the Brownian bridge stays strictly above $-a$, this time is equal to $\ED(\partial_i, \partial_o)$. In this case $\A_{-a}^{u_v}$ connects $\partial_0$ with $\partial_i$, i.e. $\A_{-a}^{u_v,\partial_0} =\A_{-a}^{u_v}$.
	\end{prop}

	\subsection{Connection between the critical Brownian loop-soup and the GFF}
\label{SubSec loop soup}
Here we recall how to construct the outermost clusters of a critical Brownian loop-soup out of local sets of a GFF. For details, see \cite{ALS2}.

	First, the measure on Brownian loops is constructed as follows.
	For $z\in\C$ and $t>0$, let $\mathbb{P}^{t}_{z,z}$ denote the bridge probability measure from $z$ to $z$ in time $t$ associated to the standard Brownian motion in $\C$. Following \cite{LW2004BMLoopSoup}, the Brownian loops measure in $\D$ is
	\begin{displaymath}
	\mu^{\D}_{\rm loop}(d\gamma)=
	\int_{\D}\int_{0}^{+\infty}
	\1_{\gamma \text{ stays in } \D}\mathbb{P}^{t}_{z,z}(d\gamma)
	\dfrac{1}{2\pi t^{2}}
	\, dt \, dz.
	\end{displaymath}
	The critical Brownian loop-soup 
	$\mathcal{L}^{\D}_{1/2}$
	is the Poisson point process of loops in $\D$ of intensity 
	$\frac{1}{2}\mu^{\D}_{\rm loop}$. Two loops in
	$\mathcal{L}^{\D}_{1/2}$ are in the same cluster if there is a finite chain of intersecting loops in $\mathcal{L}^{\D}_{1/2}$ joining them. One sees the clusters of $\mathcal{L}^{\D}_{1/2}$ as random subsets of $\D$ obtained by taking the union of ranges of loops in the same cluster. The outer boundaries of outermost clusters (not surrounded by other clusters) are distributed as
	a CLE$_{4}$ loop ensemble \cite{SheffieldWerner2012CLE}.
	More precisely, let $\mathcal C$ be the outermost cluster of
	$\mathcal{L}^{\D}_{1/2}$ surrounding the origin. Let $\ell_1$
	be the outermost boundary of $\mathcal C$ and
	$\ell_2$  the inner boundary of $\mathcal C$
	surrounding the origin. Then $\ell_1$ is distributed as the CLE$_{4}$ loop surrounding the origin. As explained above, this loop and in fact the whole of CLE$_4$ can be also constructed as a two-valued local set of the GFF. In \cite{ALS2} it was further shown that the whole
	$\overline{\mathcal C}$ can be seen as local set of the GFF as follows.

Let $\Phi$ be a zero boundary GFF on $\D$ and
	$\A_{-2\lambda,2\lambda}$ a TVS of $\Phi$.
	Let $O$ be the connected component of
	$\D\backslash \A_{-2\lambda,2\lambda}$ containing $0$. Let
	$\alpha\in\{-2\lambda,2\lambda\}$ denote the random boundary value of the GFF $\Phi$ on $\partial O$.
	$\Phi^{\A_{-2\lambda,2\lambda}}\vert_O$ will denote the restriction of the conditional GFF $\Phi^{\A_{-2\lambda,2\lambda}}$ to
	$O$. We now define a local set $\widecheck{A}_{0}$ of the GFF $\Phi$ as follows. On the event $\alpha=2\lambda$, we set
	\begin{displaymath}
	\widecheck{A}_{0} = \A_{-2\lambda,2\lambda} \cup
	\A_{0}^{2\lambda}(\Phi^{\A_{-2\lambda,2\lambda}}\vert_ O),
	\end{displaymath}
	where $\A_{0}^{2\lambda}(\Phi^{\A_{-2\lambda,2\lambda}}
	\vert_ O) \subset \overline{O}$ is an FPS of 
	$\Phi^{\A_{-2\lambda,2\lambda}}\vert_O$.
	On the event $\alpha= -2\lambda$,
	\begin{displaymath}
	\widecheck{A}_{0} = \A_{-2\lambda,2\lambda} \cup
	\A_{0}^{2\lambda}(-\Phi^{\A_{-2\lambda,2\lambda}}\vert_O),
	\end{displaymath}
	where one flips the sign of the conditional GFF 
	$\Phi^{\A_{-2\lambda,2\lambda}}\vert_O$.

	\begin{prop}[Proposition 5.3 in \cite{ALS2}]
		\label{PropClusterFPS}
		The closed critical Brownian loop-soup cluster $\overline{\mathcal C}$ has same law of as
		$\overline{\widecheck{A}_{0}\backslash \A_{-2\lambda,2\lambda}}$.
	\end{prop}

\subsection{Stationary distribution of a single CLE$_{4}$ loop surrounding 0}
\label{SubSecStationary}

	In this subsection, we recall the notion of stationary measure on a CLE$_{4}$
	loop in whole $\C$, surrounding $0$, constructed in \cite{KemppainenWerner16NestedCLE}\footnote{The results of this section hold for all CLE$_\kappa$, with $\kappa \in (8/3,4]$ but we present them only in the case $\kappa=4$ for more clarity.}.
	Consider the space of simple loops $\wp$ in $\C$, at positive distance from $0$, surrounding $0$, such that
	\begin{displaymath}
	\crad (0,\C\backslash\wp)=1,
	\end{displaymath}
	where $\crad (0,\C\backslash\wp)$ denotes the conformal radius seen from $0$ of the interior surrounded by $\wp$. On this space consider the following Markov chain $(\wp_{j})_{j\geq 0}$.
	To go from $\wp_{j}$ to $\wp_{j+1}$, one first samples a 
	CLE$_{4}$ loop ensemble in the simply connected domain surrounded
	by $\wp_{j}$. Then one takes $\tilde{\wp}$ the CLE$_{4}$ loop surrounding $0$ and scales it by the factor
	$\crad (0,\C\backslash\tilde{\wp})^{-1}$ to get
	$\wp_{j+1}=\crad (0,\C\backslash\tilde{\wp})^{-1}\tilde{\wp}$.
	By construction, $\crad (0,\C\backslash\wp_{j+1})=1$.
	According to Proposition 2 in \cite{KemppainenWerner16NestedCLE}
	(see also Section 3.2 in \cite{KemppainenWerner16NestedCLE}), 
    there is
	a unique probability measure that is stationary for the Markov chain
	$(\wp_{j})_{j\geq 0}$. We will denote it by
	$\mathbb{P}^{\rm stat}_{\operatorname{CLE}_{4}}$.
	We will further need the following result.

	\begin{thm}[Corollary 2 in \cite{KemppainenWerner16NestedCLE}]
		\label{ThmConvStatCLE}
		Consider the CLE$_{4}$ loop ensemble in the unit disk $\D$. Let $\ell$ denote the loop in CLE$_{4}$ that surrounds $0$. The law of the loop
		\begin{align*}
		\crad(0,\D\backslash\ell)^{-1}\ell
		\end{align*}
		conditionally on the event
		\begin{displaymath}
		\ED(\ell,\partial \D)>L
		\end{displaymath}
		converges as $L\to +\infty$ to the stationary probability measure 
		$\mathbb{P}^{\rm stat}_{\operatorname{CLE}_{4}}$.
	\end{thm}

\section{Explicit Radon-Nikodym derivatives for FPS and TVS}
\label{SecRN}

In this section, we will study how the laws of the TVS and FPS change when we change boundary conditions of a fixed domain, or when we change the domain itself. 

\subsection{Laws of local sets after a change of measure}
Consider $g\in H_0^1(D)$, and let $\Phi$ be a (zero-boundary) GFF. Define
\begin{equation}\label{e.change measure 1}
d\widetilde \P= \exp\left ((\Phi, g)_\nabla-\frac{1}{2}(g,g)_\nabla\right )d\P.
\end{equation}
Then by the Girsanov theorem,
$\widetilde \Phi: = \Phi - g$ is a (zero-boundary) GFF under 
$\widetilde \P$.
Let now $A$ be a local set of $\Phi$. By $g_A$ we denote the orthogonal projection of $g$ in $H^{1}_{0}(D)$
on the subspace of functions that are harmonic in $D\backslash A$.
Set $\widetilde \Phi_A := \Phi_A - g_A$.
By Theorem \ref{t. change of measure local sets}, under $\widetilde \P$,
$( \widetilde \Phi, A,\widetilde \Phi_A)$ is a local set coupling. Moreover, we can explicitly calculate how the law of these local sets is modified under this change of measure: 
\begin{lemma}\label{lem::RN}
Let $\Q$ be the law of $(A,\Phi_A)$ under $\P$ and $\widetilde \Q$ be the law of $(A, \widetilde \Phi_A+g_A)=(A,\Phi_A)$ 
under $\widetilde \P$. 
Then
\[\dfrac{d\widetilde \Q}{d \Q}= 
\exp\Big((\Phi_A, g)_\nabla-\frac{1}{2}(g_A,g_A)_{\nabla}\Big).\]
\end{lemma}

\begin{rem}
One way to give a rigorous sense to $(\Phi_A, g)_\nabla$ is to set for any $g\in H_0^1(D)$, $(\Phi_A,g)_{\nabla}:= \E\left[(\Phi,g)_{\nabla}\mid (A,\Phi_A) \right]$. 
\end{rem}

\begin{proof}
	Take $F$ a bounded measurable
function from the space of closed subsets of $\overline{D}$ times $H^{-1}(D)$ to $\R$. Then,
	\begin{align*}
	\widetilde \Q\left[F(A,\widetilde \Phi_A+g_A) \right]&= 
	\E\left[F(A,\Phi_A)\exp\Big((\Phi, g)_{\nabla}
	-\frac{1}{2}(g,g)_{\nabla}\Big) \right] \\
	&=\E\left[F(A,\Phi_A)\E\left[ 
	\exp\Big((\Phi, g)_{\nabla}-\frac{1}{2}(g,g)_{\nabla}\Big)
	\Big\vert (A,\Phi_A)\right] \right].
	\end{align*}
	Now, conditional on $A$, the law of $\Phi$ is that of the sum of two independent fields $\Phi_A$ and $\Phi^A$, where the law of $\Phi^A$ is that of a GFF in $D\backslash A$ and $\Phi_A$ is harmonic in $D\backslash A$. We can write a similar decomposition for $g=g_A+g^A$, where $g_A$ is harmonic in $D\backslash A$, and $g^A$ is supported in $D\backslash A$. 
Furthermore, the given decomposition is orthogonal in $H_0^1(D)$, i.e. 
$(g,g)_\nabla = (g_A,g_A)_{\nabla}+ (g^A,g^A)_{\nabla}$. 
Moreover, 
$(\Phi,g)_\nabla = (\Phi_A, g)_\nabla + (\Phi^A, g^{A})_\nabla$
a.s. Indeed, conditional on $A$,
$(\Phi^A, g_A)_\nabla$ is a Gaussian r.v. and its variance is zero since $g_A$ is orthogonal to $H_0^1(D\backslash A)$.

Thus,
	\begin{multline*}
	\E\left[ 
	\exp\Big((\Phi, g)_{\nabla}-\frac{1}{2}(g,g)_{\nabla}\Big)
	\Big\vert (A,\Phi_A)\right]=
	\\
	\exp\Big((\Phi_A, g)_\nabla-\frac{1}{2}(g_A,g_A)_{\nabla}\Big)
	\E\left[\exp(\Phi^A, g^A)_\nabla-\frac{1}{2}(g^A,g^A)_{\nabla}\Big		\vert A \right].
	\end{multline*}
	Finally, 
	\begin{equation*}
	\E\left[\exp\Big((\Phi^A, g^A)_\nabla
	-\frac{1}{2}(g^A,g^A)_{\nabla}\Big)
	\Big\vert A \right]=1,
	\end{equation*}
	and the result follows.
\end{proof}

\subsection{Radon-Nykodim derivative for local set in an annulus}\label{ss.RN for annulus}
Now consider an annulus 
$\An_r$ defined as $\D\backslash r\overline{\D}$, 
with $r\in(0,1)$. We denote $\partial_i= r \partial \D$ and 
$\partial_o=\partial \D$ the inner an outer boundary of $\An_r$ respectively. 
Recall from Section \ref{SubsecED} that $\ED(\partial_o,\partial_i)=
(2\pi)^{-1}\log(r^{-1}).$

For $v\in\R$ let $u_v$ be the harmonic function in $\An_r$ 
that takes value $v$ on $\partial_i$ and zero on $\partial_o$.
Let $A^v =\A^{u_v,\partial_o}_{-a,b}$
 with ($a,b >0, a+b\geq 2\lambda$)
denote the connected component of the 
two-valued local set $\A^{u_v}_{-a,b}$ (in $\An_r$) 
containing the outer boundary $\partial_{o}$.
On the event
$A^{v}\cap\partial_{i}=\emptyset$, 
define $\ell^{u_{v}}_{-a,b}$ 
to be the boundary of the connected component 
$\O_{i}^{u_{v}}$ of $\An_r\backslash A^{v}$ with 
$\partial_{i}\subset \overline{\O_{i}^{u_{v}}}$, i.e. we set 
$\ell^{u_{v}}_{-a,b}:=A^{v}\cap \overline{\O_{i}^{u_{v}}}$. Then
$\ell^{u_{v}}_{-a,b}$ is a simple loop, non-contractible in 
$\overline{\An_r}$, separating $\partial_o$ and $\partial_i$\footnote{Locally, the loop $\ell^{u_{v}}_{-a,b}$  looks like an SLE$_4$ loop, however we will not use this in this paper.}. On the event $A^{v}\cap\partial_{i}\neq\emptyset$, we define
$\ell^{u_{v}}_{-a,b}$ to be $\partial_{i}$.

We will show that on the event 
$\ED(\ell^{u_{v}}_{-a,b},\partial_{i})>0$
(i.e. $A^{v}\cap\partial_{i}=\emptyset$), the conditional law of
$\ED(\ell^{u_{v}}_{-a,b},\partial_{o})$ given
$\ED(\ell^{u_{v}}_{-a,b},\partial_{i})
= \ED(A^{v},\partial_{i})$
and the label $-a$ or $b$ of $\ell^{u_{v}}_{-a,b}$, does not depend on the value
$v$.

To do this, we calculate explicitly the 
Radon-Nikodym derivative of $(A^v, \Phi_{A^v}+u_v)$ with respect to 
$(A^0, \Phi_{A^0})$.

\begin{prop}\label{p.change of measure TVS}
	Let $\Q_{v}$ be the law of $(A^v, \Phi_{A^v}+u_v)$ on the event where $A^{v}$ does not intersect $\partial_{i}$, i.e. for any measurable bounded function $F$,
	\begin{equation*}
	\Q_v[F(A^v, \Phi_{A^v}+u_v)]= \E\left[ F((A^v, \Phi_{A^v}+u_v)) \1_{A^v \cap \partial_i=\emptyset}\right].
	\end{equation*}
	Further, on the event $A^{v}\cap\partial_{i}=\emptyset$, denote by
	$\alpha^{u_{v}} =\alpha^{u_{v}}_{-a,b}$ 
	the constant boundary value in
	$\{-a,b\}$ of $u_{v}+h_{A^{v}}$ on
	$\ell^{u_{v}}_{-a,b}$, seen as the boundary of $\O_{i}^{u_{v}}$ side.
	Then, we have that
	\begin{equation}\label{e.RND1}
	\frac{d \Q_v}{d \Q_0}= 
	\exp\left (
	-\frac{v^{2}}{2}
	(\ED(A^0, \partial_i)^{-1}-\ED(\partial_o, \partial_i)^{-1})
	+\alpha^{0} v \ED(A^0, \partial_i)^{-1}
	\right ).
	\end{equation}
In particular, if $\ED(\ell^{u_{v}}_{-a,b},\partial_{i})>0$, then the law of $\ED(\ell^{u_{v}}_{-a,b},\partial_{o})$ conditional on
$\alpha^{u_{v}}$ and on $\ED(\ell^{u_{v}}_{-a,b},\partial_{i})$ 
only depends on $(\alpha^{u_{v}},\ED(\ell^{u_{v}}_{-a,b},\partial_{i}))$ and not on the inner boundary value $v$.
	\end{prop}

\begin{proof}
	
Define $u_{v}^\epsilon: \overline{\An_r}\rightarrow\R$ as the unique function continuous on $\overline{\An_r}$ and harmonic
in $\An_r\backslash (r+\epsilon)\partial \D$, with boundary values
	\begin{equation}
	u_v^\epsilon= \left\{\begin{array}{l l}
	u_v & \text{ on } (r+\epsilon)\partial \D,\\
	0 & \text{ on } \partial \An_r.
	\end{array} \right. 
	\end{equation}
	Note that $u_v^\epsilon$ is equal to $u_v$ in $\An_{r+\epsilon}$.
		
Let $\P$ denote the law of $\Phi$, a (zero-boundary) GFF in $\An_r$. 
Define $\widetilde \Phi:= \Phi - u^\epsilon_v$ and
	\begin{equation}\label{e.change measure 2}
	\frac{d\widetilde \P}{d\P}= 
	\exp\left ((\Phi, u_v^\epsilon)_\nabla
	-\frac{1}{2}(u_v^\epsilon,u_v^\epsilon)_\nabla\right ),
	\end{equation}
so that $\widetilde \Phi$ is a GFF under $\widetilde \P$.
	
	Let us first prove the following claim.
	\begin{claim}\label{c.new measure}
		The event $\{d(A^0(\Phi),\partial^i)\geq \epsilon \}$ is a.s. equal to the event $\{d(A^v(\widetilde \Phi),\partial^i)\geq \epsilon \}$. Furthermore, on this event a.s. $A^0(\Phi)=A^v(\widetilde \Phi)$ and the boundary values of $h_{A^0}$ and $h^{A^v}+u_v$ coincide everywhere except on $\partial_i$.
	\end{claim}
	\begin{proof}[Proof of the Claim]
We construct $\A_{-a,b}^{u_{v},\partial_o}$ (of the GFF $\Phi$) using the  level line construction given in the proof of Proposition 3.9 of \cite{ALS1}.  Let $A^{v}_{\eps}(\Phi)$ be the local set obtained when one stops this construction the first time a level line gets $\eps$ close to the inner boundary. Let us verify that, under $\widetilde \P$, 
$A^0_\eps(\Phi)$ is equal to $A^v_\eps(\widetilde \Phi)$. First, Theorem \ref{t. change of measure local sets} ensures that $A^0_\epsilon(\Phi)$ is a BTLS of 
$\widetilde\Phi$. Moreover, on the event $A^0$ is at distance $\epsilon$ from $\partial_i$, on the boundary of $A^0$  we have that 
$\widetilde\Phi_{A^0_\eps(\Phi)}+u_v$ takes values in $\{-a,b\}$; and on the complement of this event, again on the boundary of $A^0$, the function
$\widetilde\Phi_{A^0_\eps(\Phi)}+u_v$ takes values in $[-a,b]$, changing only finitely many times on the boundary of each connected component. 

Thus, following the cited construction of \cite{ALS1} on the event that $A^0$ does get $\eps$ close, one can complete $A^0_\eps(\Phi)$ to build a local set 
$\widehat A$ of $\widetilde \Phi$ that comes at distance $\eps$ from $\partial_i$ and such that 
$\widetilde\Phi_{\widehat A}+u_v$ takes values in $\{-a,b\}$ on the boundary of $\widehat A$. By uniqueness of the TVS, 
$\widehat A = A^v(\widetilde \Phi)$. Note that from the construction it follows that on the event when $A^0(\Phi)$ remains at distance $\eps$ from the boundary a.s. $A^0(\Phi)=A^v(\widetilde \Phi)$. Furthermore, if $d(A^0(\Phi), \partial_i)\leq \epsilon$, we have that also $d(A^v(\tilde \Phi),\partial_i)\leq \epsilon$.
	\end{proof}
	
The claim implies that the measure $\Q_v$ restricted to the event that $A^v$ is at distance $\epsilon$ from $\partial_i$ is absolutely continuous with respect to the measure $\Q_0$ 		restricted to the event that $A^0$ is at distance $\epsilon$ from $\partial_i$. Furthermore, the Radon-Nykodim derivative is given by the conditional expectation of
		\eqref{e.change measure 2} given $(A^{0}, \Phi_{A^{0}})$. Thus, it remains to compute the conditional expectation of
\eqref{e.change measure 2} to obtain the result. 
Let $\Q_0$, respectively $\widetilde\Q_0$ be the laws of
$(A^{0},\Phi_{A^{0}})$ under
$\P$, respectively $\widetilde\P$, restricted to the event that $A^0$ stays at a positive distance from $\partial_i$.
By Lemma \ref{lem::RN} we have that 
	\begin{equation}\label{e.RND2}
	\frac{d \widetilde \Q_0}{d \Q_0}
	= \exp\left ((\Phi_{A^0}, 
	u_v^\epsilon)_\nabla
	-\frac{1}{2}((u_v^\epsilon)_{A^0},
	(u_v^\epsilon)_{A^0})_{\nabla}\right ),
	\end{equation}
where $(u_v^\epsilon)_{A^0}$ is the orthogonal projection
of $u_v^\eps$ in $H^{1}_{0}(\An_r)$ on the subspace of functions that are harmonic in $\An_r\backslash A^{0}$.
On the event $d(A^{0},\partial_{i})>\eps$, we can explicitly calculate the terms inside the exponential. 
Indeed, denote by 
$I_{A^0}:\overline{\An_r}\rightarrow\R$ 
the function that is continuous on $\overline{\An_r}$, harmonic in $\An_r \backslash A^0$, and takes value $0$ on 
$A^0\cup \partial_o$, and value $1$ on $\partial_i$.
	
	Observe that on the event $d(A^{0},\partial_{i})>\eps$,
$(u_v^\eps)_{A^0} + vI_{A^0} = u_v$ on the whole of $\An_r$. Moreover, observe that 
$((u_v^\eps)_{A^0}, u_v)_\nabla = 0$ as $(u_v^\eps)_{A^0}$ is zero on $\partial\An_r$ 
and $u_v$ is harmonic in $\An_r$. 
	Thus, 
	\begin{align*}
	(vI_{A^0}, vI_{A^0})_\nabla 
	= (u_v - (u_v^\eps)_{A^0}, 
	u_v - (u_v^\eps)_{A^0})_\nabla 
	= (u_v, u_v)_\nabla + 
	((u_v^\eps)_{A^0},
	(u_v^\eps)_{A^0})_{\nabla}.
	\end{align*}
	From Theorem \ref{thmEL}, it then follows that 
	\begin{align*}
	((u_v^\eps)_{A^0},
	(u_v^\eps)_{A^0})_{\nabla} 
	= -v^2\ED(\partial_o, \partial_i)^{-1}+v^2\ED(A^0, \partial_i)^{-1}.
	\end{align*}
	
Let us now compute 
$(\Phi_{A^{0}},u_v^\eps)_\nabla$. By definition, and recalling that $u_v$ is harmonic in $\An_r$,
\begin{align}\label{e.phi A_0}
(\Phi_{A^{0}},u_v^\eps)_\nabla = (\Phi_{A^{0}},-\Delta u_v^\eps)=(\Phi_{A^{0}},-\Delta (u_v^\eps-u_v)).
\end{align}
Now, note that $u_v^\eps-u_v$ is supported in $\An_r\cap (r+\epsilon) \D$, and furthermore that $\Phi_{A^0}$ is harmonic in $\An_r\cap (r+\epsilon) \D$. Using integration by parts we obtain that on the event $d(A^{0},\partial_{i})>\eps$ \eqref{e.phi A_0} is equal to 
\begin{align*}
-\int_{O_i^0} h_{A^0} \Delta (u_v^\eps-u_v)&= \int_{O_i^0} \nabla h_{A^0} \nabla (u_v^\eps-u_v)=-v\alpha^0 \int_{\partial_i} \partial_n(1-I_{A^0})=v\alpha^0 \ED(A^0,\partial_i)^{-1}.
\end{align*}
Here we used that when restricted to $O_i^0$, $h_{A^0}=\alpha^0(1-I_{A^0})$ and Theorem \ref{thmEL}.

We conclude the proof of the identity \ref{e.RND1} by taking 
$\eps\to 0$.

Finally, notice that the non-dependence on $v$ of the conditional law 
of $\ED(\ell^{u_{v}}_{-a,b},\partial_{o})$ simply comes from the fact that the Radon-Nykodim derivative in \eqref{e.RND1} is measurable with respect to
$(\alpha^{0},\ED(\ell^{0}_{-a,b},\partial_{i}))$.
\end{proof}

A similar proof gives us the following result for an FPS. 
Let $a>0$. Consider the FPS $\A^{u_v}_{-a}$ in
$\An_r$ and let $\A^{u_v,\partial_{o}}_{-a}$ denote the connected component of
$\A^{u_v}_{-a}$ containing $\partial_{o}$.
On the event that 
$\A^{u_v,\partial_{o}}_{-a}\cap\partial_{i}=\emptyset$, let
$\ell^{u_{v}}_{-a}$ be the only loop delimited by 
$\A^{u_v,\partial_{o}}_{-a}$ which is not contractible in
$\overline{\An_r}$. In other words,
$\ell^{u_{v}}_{-a}\cup\partial_{i}$ is the boundary of the unique connected component of $\An_r\backslash\A^{u_v,\partial_{o}}_{-a}$
which is topologically an annulus. 
We have that 
$\ED(\A_{-a}^{u_{v},\partial_{o}},\partial_{i})=
\ED(\ell^{u_{v}}_{-a},\partial_{i})$.
On the event
$\A^{u_v,\partial_{o}}_{-a}\cap\partial_{i}\neq\emptyset$, we set
$\ell^{u_{v}}_{-a}=\partial_{i}$.

\begin{prop}\label{p.change of measure FPS}
	Let $\Q_{v}$ be the law 
of $\A^{u_v,\partial_{o}}_{-a}$ on the event where $\A^{u_v,\partial_{o}}_{-a}$ does not intersect $\partial_{i}$. Then, we have that
	\begin{equation}\label{e.RND3}
	\frac{d \Q_v}{d \Q_0}= 
	\exp\left (	-\frac{v^{2}}{2}
	(\ED(\A^{u_v,\partial_{o}}_{-a}, \partial_{i})^{-1}
	-\ED(\partial_{o}, \partial_{i})^{-1})
	-av \ED(\A^{u_v,\partial_{o}}_{-a}, \partial_{i})^{-1}\right).
	\end{equation}
	In particular, if $\ED(\ell^{u_{v}}_{-a},\partial_{i})>0$, the law of 
$\ED(\ell^{u_{v}}_{-a},\partial_{o})$, conditionally on
$\ED(\ell^{u_{v}}_{-a},\partial_{i})$ only depends on 
$\ED(\ell^{u_{v}}_{-a},\partial_{i})$ and not on $v$.
\end{prop}

Finally, let us extend this proposition also to the local set that gives us the law of outermost clusters of the Brownian loop soup, as explained in Subsection \ref{SubSec loop soup}. We will need to work in a slightly more general setting than described in that subsection, so let us describe the construction of this local set $\widecheck{A}^{u_{v}}_{0}$ again in detail. 

First we take
$\A^{u_{v},\partial_{o}}_{-2\lambda,2\lambda}$.
If 
$\A^{u_{v},\partial_{o}}_{-2\lambda,2\lambda}\cap\partial_{i}
\neq\emptyset$, then we set
$\widecheck{A}^{u_{v}}_{0}:=\A^{u_{v},\partial_{o}}_{-2\lambda,2\lambda}$. 

Otherwise, let $O^{u_{v}}_{i}$ denote the connected component of
$\An_{r}\backslash \A^{u_{v},\partial_{o}}_{-2\lambda,2\lambda}$ such that $\partial_{i}\subset\overline{O^{u_{v}}_{i}}$. 
$O^{u_{v}}_{i}$ is an annular domain. 
Let 
$\ell^{u_v}=\ell^{u_{v}}_{-2\lambda,2\lambda}:= 
\A^{u_{v},\partial_{o}}_{-2\lambda,2\lambda}
\cap \overline{O^{u_{v}}_{i}}$ 
denote the outer boundary of $O^{u_{v}}_{i}$. Let
$\alpha^{u_{v}}=\alpha^{u_{v}}_{-2\lambda, 2\lambda}
\in\{-2\lambda,2\lambda\}$ be the constant boundary value of $u_{v}+h_{\A^{u_{v},\partial_{o}}_{-2\lambda,2\lambda}}$ on
$\ell^{u_{v}}$ from the inner side.
On the event 
$\alpha^{u_{v}}=2\lambda$, we consider the local set
$\A^{u_{v}+h_{\A^{u_{v},\partial_{o}}_{-2\lambda,2\lambda}},
\ell^{u_{v}}}_{0}
(\Phi^{\A^{u_{v},\partial_{o}}_{-2\lambda,2\lambda}}\vert_ {O^{u_{v}}_{i}})$, which is the connected component containing
$\ell^{u_{v}}$ of the first passage set of level $0$ of the conditional GFF inside $O^{u_{v}}_{i}$ with boundary values $v$ on $\partial_{i}$ and $\alpha^{u_{v}}$ on
$\ell^{u_{v}}$. 

We finally set
\begin{displaymath}
\widecheck{A}^{u_{v}}_{0}:=
\A^{u_{v},\partial_{o}}_{-2\lambda,2\lambda}\cup
\A^{u_{v}+h_{\A^{u_{v},\partial_{o}}_{-2\lambda,2\lambda}},
\ell^{u_{v}}}_{0}
(\Phi^{\A^{u_{v},\partial_{o}}_{-2\lambda,2\lambda}}
\vert_{ O^{u_{v}}_{i}}).
\end{displaymath}
On the event $\alpha^{u_{v}}= -2\lambda$, we set
\begin{displaymath}
\widecheck{A}^{u_{v}}_{0}:=
\A^{u_{v},\partial_{o}}_{-2\lambda,2\lambda}\cup
\A^{-u_{v}-h_{\A^{u_{v},\partial_{o}}_{-2\lambda,2\lambda}},
\ell^{u_{v}}}_{0}
(-\Phi^{\A^{u_{v},\partial_{o}}_{-2\lambda,2\lambda}}
\vert _{O^{u_{v}}_{i}}),
\end{displaymath}
where we take the opposite of the conditional GFF
$\Phi^{\A^{u_{v},\partial_{o}}_{-2\lambda,2\lambda}}
\vert_{O^{u_{v}}_{i}}$ because $\alpha^{u_{v}}<0$. Constructed that way,
$\widecheck{A}^{u_{v}}_{0}$ is a local set of $\Phi$ and as mentioned, it is related to a cluster of a Brownian loop-soup. We now prove the analogue of Propositions \ref{p.change of measure TVS} and \ref{p.change of measure FPS} for this local set. As the proof is similar, it is omitted. 

\begin{prop}\label{p.change of measure cluster}
	Let $\Q_{v}$ be the law of 
$(\widecheck{A}^{u_{v}}_{0}, \Phi_{\widecheck{A}^{u_{v}}_{0}}+u_v)$ on the event where 
$\widecheck{A}^{u_{v}}_{0}$ does not intersect $\partial_{i}$, i.e. for any measurable bounded function $F$,
	\begin{equation*}
	\Q_v[F(\widecheck{A}^{u_{v}}_{0}, 
	\Phi_{\widecheck{A}^{u_{v}}_{0}}+u_v)]= 
	\E\left[ F(\widecheck{A}^{u_{v}}_{0}, 
	\Phi_{\widecheck{A}^{u_{v}}_{0}}+u_v) 
	\1_{\widecheck{A}^{u_{v}}_{0} \cap \partial_i=\emptyset}\right].
	\end{equation*}
	Further, on the event
	$\widecheck{A}^{u_{v}}_{0}\cap\partial_{i}=\emptyset$ denote by 
	$\check{\ell}^{u_{v}}_{0}$
	the outer boundary of the annular connected component of 
	$\An_r\backslash \widecheck{A}^{u_{v}}_{0}$ whose inner boundary is
	$\partial_{i}$.
	Then, we have that
	\begin{equation}\label{e.RNDcluster}
	\frac{d \Q_v}{d \Q_0}= 
	\exp\left (	-\frac{v^{2}}{2}
	(\ED(\widecheck{A}^{0}_{0}, \partial_{i})^{-1}
	-\ED(\partial_{o}, \partial_{i})^{-1})
	\right).
	\end{equation}
In particular, if $\ED(\check{\ell}^{u_{v}}_{0},\partial_{i})>0$, 
then the law of 
$\ED(\check{\ell}^{u_{v}}_{0},\partial_{o})$ conditional on
$\ED(\check{\ell}^{u_{v}}_{0},\partial_{i})$ 
only depends on $\ED(\check{\ell}^{u_{v}}_{0},\partial_{i})$ 
and not on the inner boundary value $v$.
	\end{prop}

\subsection{Radon-Nykodim derivative for local set for different domains}
Let $\widehat  D\subseteq D$ be two domains and $\widehat \Phi$ and 
$\Phi$ be two GFFs in $\widehat D$ and $D$ respectively. Define $C$ to be the closure of $D\backslash\widehat D$, and note that because of the domain Markov property, $\Phi= \Phi_C+\Phi^C$, where both terms are independent and $\Phi^C$ has the same law as $\widehat \Phi$. The harmonic part of $\Phi_C$ is denoted as usual $h_C$.

From Corollary 14 in \cite{ASW}, we know that if $A$ is a local set of  $\Phi$ and $A\subseteq \widehat D$ a.s., then $A$ is also a local set of $\Phi^C$, and $(\Phi^C)_A= \Phi_{A\cup C}-\Phi_C $. We restate it carefully here.

\begin{prop}\label{p.change measure boundary}
	Take $(\Phi,A,\Phi_A)$ a local set coupling such that a.s. $d(A,C)\geq \epsilon>0$.
Let $C_\eps$ be some neighbourhood of $C$ in $D$, whose boundary in $D$ is given by a continuous curve, such that $C_\eps \subseteq (C+\eps\overline{\D})\cap D$. Define $h_C^\epsilon : \widehat D\rightarrow\R$ 
as the unique function harmonic in 
$\widehat D\backslash \partial C_\eps$ with boundary value $0$ on $\partial \widehat D$ and $h_C$ on 
$\partial C_\epsilon\cap \widehat D$. 
Consider the change of measure
	\begin{equation}\label{e.change of measure domain}
	d\widehat \P_C^\eps
	=\exp\left (-(\Phi^C,h_C^\epsilon)_{\nabla}
	-\frac{1}{2}(h_C^\epsilon,h_C^\eps)_{\nabla}\right ) 
	d\widehat \P, 
	\end{equation}
where $\widehat\P$ is the law of $\widehat\Phi=\Phi^C$,
and $(\cdot,\cdot)_{\nabla}$ denotes here an integral over
$\widehat{D}$.	
We have that under $\widehat \P^{\eps}_C$, 
	$\widehat \Phi^{\eps}:= \Phi^C +h_C^\epsilon$ is a GFF in $\widehat D$ and  
$(\widehat \Phi^{\eps}, A, 
\widehat {\Phi}^{\eps}_A:=(\Phi^C)_A+(h_C^\epsilon)_A)$ is a local set coupling.
\end{prop}
\begin{proof}
	This follows directly from Theorem \ref{t. change of measure local sets} and the fact that due to the condition
$d(A,C)\geq \epsilon$, $(h_C)_A=(h_C^\epsilon)_A$, where $(f)_A$ 
denotes the function that coincides with $f$ on $A$, is zero on
$\partial \widehat D$, and is extended harmonically to
$\widehat D\backslash A$.
\end{proof}

\begin{rem}\label{r.change measure boundary}
	Note that in Proposition \ref{p.change measure boundary} the boundary values of
the harmonic part of $\widehat{\Phi}^{\eps}_A$ coincide
on $A$ with that of 
	$\Phi_A$,  and are $0$ on $\partial\widehat D$. 
\end{rem}

Let $\Phi$ denote the zero boundary GFF in the disk
$\D$ and $\Phi^r$ the zero boundary GFF in the annulus
$\An_r$.
Let $\A_{-a,b}^{\partial_o}(\Phi^r)$
be the connected component, containing $\partial_o$, of the two-valued set of a GFF $\Phi^r$. We now prove a proposition stating that there is a convergence in total variation of 
$\A_{-a,b}^{\partial_o}(\Phi^r)$ to 
$\A_{-a,b}(\Phi)$, as $r\to 0$.

\begin{prop}\label{prop:coupling}
	For every $\delta>0$, there exists $r_0$ such that for all $r<r_0$ there exists a coupling between 
$\Phi$ and $\Phi^r$, such that
	\[\P\left(\A_{-a,b}(\Phi)=
	\A_{-a,b}^{\partial_o}(\Phi^r) \right)\geq 1-\delta. \]
	Furthermore, the analogue is true for the FPS $\A_{-a}$.
\end{prop}

\begin{proof}
	We prove the result for the TVS $\A_{-a,b}$ as the proof is the same for FPS $\A_{-a}$. 
Let us work in the context of Proposition \ref{p.change measure boundary} with $C=r\overline \D$. We show that there exists a local set $A$, an event $\widetilde \Omega$ and $\epsilon>0$ such that $\P(\widetilde \Omega)=1-\delta/2$, a.s. on $\widetilde \Omega$ we have that  
$\A_{-a,b}(\Phi)=\A_{-a,b}^{\partial_o}(\Phi^C+h^\epsilon_{C})=A$ and for all $\omega \in \tilde \Omega$
	\begin{align}\label{e.rnd to 1}
	\left |\frac{d\widehat \P^\epsilon_r}{d\P}(\omega)-1\right |\leq \delta/2,
	\end{align}
	 where
	  \begin{align}\label{e.RND change of measure r}
	  d\widehat \P^\epsilon_r=\exp\left (-\left (\Phi^{C},h_{C}^\epsilon\right )_{\nabla}
	  -\frac{1}{2}\left (h_{C}^\epsilon,h_{C}^\eps\right )_{\nabla}\right ) d\P.
	  \end{align}
	 
	We start as in Claim \ref{c.new measure}, by defining $A$ as the set obtained from the construction of 
$\A_{-a,b}(\Phi)$ when stopped the first time it hits 
$\epsilon\D$. We choose $\epsilon>0$ such that
	\[\P\left (\widetilde \Omega_1:=\{d(\A_{-a,b}(\Phi),0)> \epsilon\}\right )>1-\delta/2.\]
	We then have the analogue of Claim \ref{c.new measure}
	\begin{claim}\label{c.newmeasurecd} 
	On the event $d(A,0)> \epsilon$, a.s. 
	$A=\A_{-a,b}(\Phi)
	=\A_{-a,b}^{\partial_o}(\Phi^C+h^\epsilon_{C})$.		
	\end{claim}
	The proof is basically the same as that of Claim \ref{c.new measure} and is thus omitted.
		
	Now, let us see that there is an $r>0$ small enough such that \eqref{e.rnd to 1} is true. From \eqref{e.RND change of measure r}, we just have to show that $h_{C}^\epsilon$ converges to $0$ in $H^1_0(\D)$, as $r\to 0$. As $h^\epsilon$ is bounded and harmonic in $\D\backslash (\epsilon \partial \D \cup r \D)$ and taking value $0$ in $\partial \D$ and $r\partial \D$ and $h_{C}$ in $\epsilon \partial \D$, the convergence follows from the following claim.
	\begin{claim}
		Define $\hat h_r$ as the restriction of $h_{C}$ to $(\epsilon/2)\partial \D$ and see it as a continuous function from the circle to $\R$. Then as $r\to 0$,  $\hat h_r$ converges in probability for the topology of $\mathbb{L}^2$ to $0$.
	\end{claim}
	\begin{proof}
		Note that when restricted to $\D\backslash r\D$, $\Phi= \Phi^{C}+ h_{C}$, where both terms of the sum are independent. This implies that $h_{C}$ is a centred Gaussian process with covariance
		\begin{align*}
		\E\left[ h_{C}(x)h_{C}(y)\right]=G_{\D}(x,y)-G_{\D\backslash r\D}(x,y).
		\end{align*}
		Note that this extends continuously when $x\to y$, to $g_r(x,x)$. Here, $g_r(x,\cdot)$ is the unique bounded harmonic function in $\D\backslash r\D$ with values $0$ in $\partial \D$ and $-(2\pi)^{-1}\log(\|x-\cdot\|)$ in $r\partial \D$.
		
		Thus,
		\begin{align*}
		\E\left[\int_{(\epsilon/2)\partial \D} (\hat h_r(x))^2 dx \right] = \int_{(\epsilon/2)\partial \D} g_r(x,x) dx\leq \frac{\log(1/\epsilon)}{\log(1/r)}\frac{\log(1/(\epsilon + r))}{2\pi}\to 0, \ \ \text{as $r\to 0$.}
		\end{align*}
		Here, the inequality just follows from the estimate of the probability that a Brownian motion started at $x$ hits $r\partial \D$ before
		 $\partial \D$.
	\end{proof}	 
	\end{proof}

\section{The marginal laws in an annulus} \label{S.Marginal laws}

Consider an annulus $\An$ with outer and inner boundaries denoted by $\partial_o$ and $\partial_i$, and $\A_{-2\lambda,2\lambda}$ the two-valued set of a zero-boundary GFF on the annulus $\An$. 

Recall from Corollary \ref{c.tvs from a boundary} the definition of $\A_{-2\lambda,2\lambda}^{\partial_o}$ as the connected component of $\A_{-2\lambda, 2\lambda}$ connected to $\partial_o$. We saw that either $\A_{-2\lambda, 2\lambda}$ intersects $\partial_i$, or $\An \backslash \A_{-2\lambda,2\lambda}^{\partial_o}$ has a connected component that is annular and has $\partial_i$ as one of its boundaries. We set $\ell$ to be the other boundary of this annular component, i.e. the non-trivial loop of $\A_{-2\lambda,2\lambda}^{\partial_o}$. In the case when $\A_{-2\lambda,2\lambda}^{\partial_o}$ intersects $\partial_i$, we set $\ell = \partial_i$. We are interested in the marginal law of the extremal distances $\ED(\ell, \partial_o)$ and $\ED(\ell, \partial_i)$ (see Figure \ref{f.EDinAnnulus}).
\begin{figure}[h!]
	\includegraphics[height=0.28\textwidth]{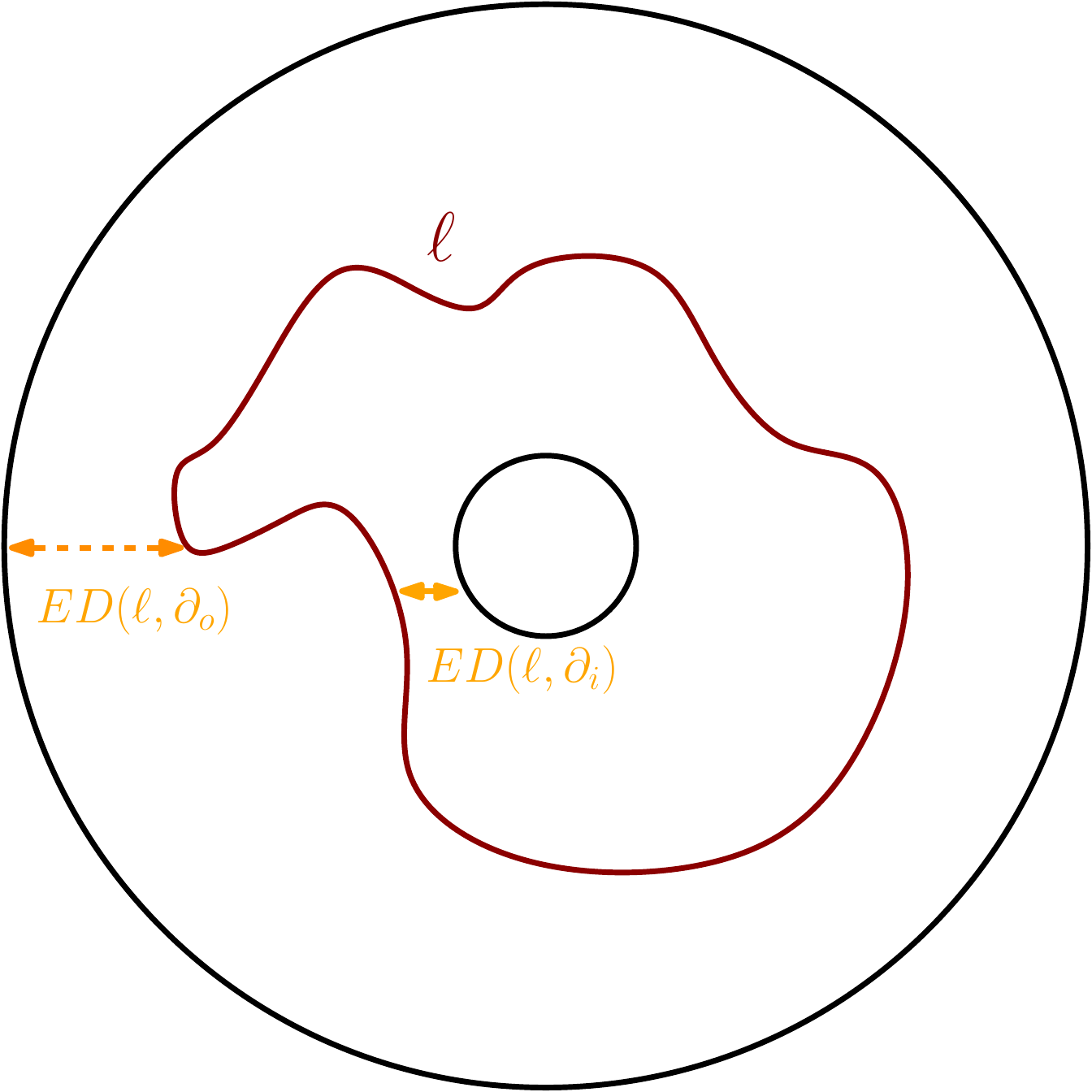}
	\caption{In the figure you can see the non-trivial loop $\ell$ together with the two-random variables we are interested in.}
	\label{f.EDinAnnulus}
\end{figure} 

In fact, Proposition \ref{p.LawELBddTVS} already gives us the law of $\ED(\ell, \partial_i)$. The main result of this section is the computation of the law of $\ED(\ell, \partial_o)$. To describe the resulting law, let $\widehat B$ be a standard Brownian bridge of time-duration $L$, that denotes the extremal distance of $\An$. Define
\begin{align}
\nonumber&\widehat T_{-2\lambda,2\lambda}:=\inf\{s\geq 0: |\widehat B_s|=2\lambda \} \wedge L,\\
\label{e.hattau}&\widehat \tau_{-2\lambda,2\lambda}:=\left\{\begin{array}{l l}
\sup\{0\leq s \leq \widehat T_{-2\lambda,2\lambda}:\widehat B_s=0\} & \text{ if } \widehat T_{-2\lambda,2\lambda}<L,\\
L & \text{ if not.}
\end{array}\right.
\end{align}

\begin{prop}\label{mlaw0}
The law of $\ED(\ell, \partial_o)$ equals that of $\widehat \tau_{-2\lambda,2\lambda}$.
\end{prop}

The proof proceeds by providing a way of constructing $\ell$ via a local set exploration from the inner boundary, and then basically using a correspondence with the Brownian motion, similarly to Proposition \ref{p.LawELBddTVS}. The main input is a certain reversibility statement, saying that non-contractible loops generated by iterating two-valued sets starting from the outer boundary, agree in inverse order with the non-contractible loops generated by iterating two-valued sets from the inner boundary. 
\subsection{A reversibility statement and Proposition \ref{mlaw0}}\label{ssrv}

Consider, a zero boundary GFF $\Phi$ on an annulus $\An$ with outer and inner boundaries denoted by $\partial_o$ and $\partial_i$. We now describe an exploration from the outer boundary $\ell_0 := \partial_o$ towards the inner boundary, where each step consists of sampling the connected component of $\A_{-2\lambda, 2\lambda}$ connected to the outer boundary, denoted $\A^{\partial_o}_{-2\lambda, 2\lambda}$. 

\begin{enumerate}
	\item First, construct the component $A_1 := \A^{\partial_o}_{-2\lambda, 2\lambda}$. If $A_1$ touches the inner boundary, we stop. Otherwise one of the connected components of $\An \backslash A_1$ is an annulus $\An^1$, still containing $\partial_i$ as one of its boundaries. We denote by $\ell_1$ its other boundary, i.e. $\ell_1 = A_1 \cap \partial \An^1$. 
	\item As $(\Phi,A_1,h_{A_1})$ is a local set, conditionally on $(A_1,h_{A_1})$ the GFF, $\Phi$, restricted to $\An^1$ is equal to $\Phi^{A_1} + h_{A_1}$, where the boundary conditions of $h_{A_1}$ on $\ell_1$ (towards $\partial_i$) are equal to either $-2\lambda$ or $2\lambda$ (denote this value by $\alpha_1$) and zero on $\partial_i$. We can now construct the component $A_2$ of the two valued-set $\A_{-2\lambda +\alpha_1, 2\lambda+ \alpha_1}^{h_{A_1},\ell_1}$ of the GFF $\Phi^{A_1}$ restricted to $\An^1$, and with boundary condition $h_{A_1}$. In this case, by Proposition \ref{p.LawELBddTVS}, $A_2$ necessarily cuts out another annulus $\An^2$ (see Figure \ref{f.iteration}), whose outer boundary we denote by $\ell_2$. 
	\item It is now clear how to iterate, to obtain the \textit{iterated} $\A^{\partial_o}_{-2\lambda, 2\lambda}$, and the sequence of contours $\ell_0 = \partial_o, \ell_1, \ell_2, \dots, \ell_n, \ell_{n+1} = \partial_i$.
	\item We call $A_\infty$ the closure of the union of all sets $A_j$ explored.
\end{enumerate}

\begin{figure}[h!]
	\includegraphics[width=0.5\textwidth]{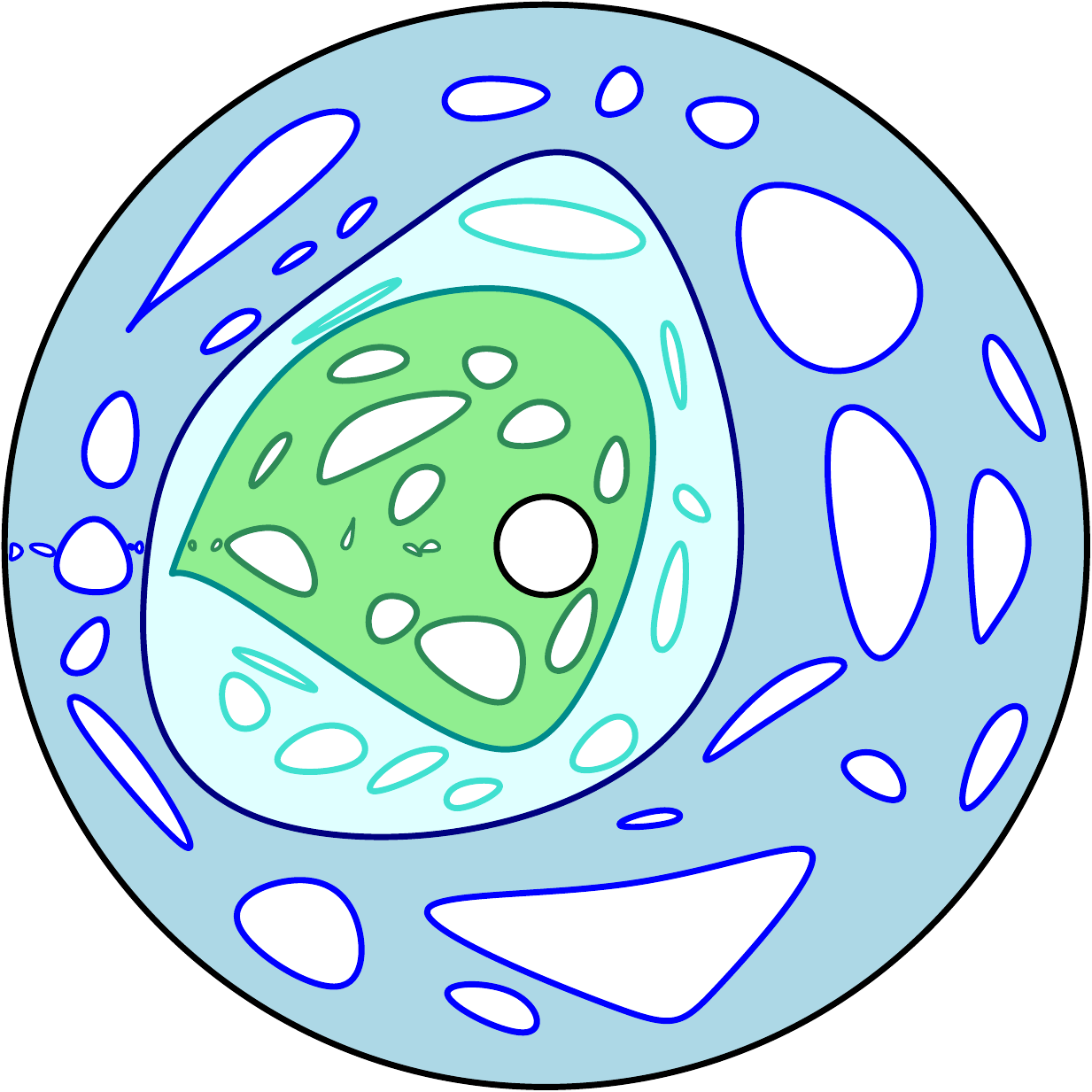}
	\caption{A representation of the iteration produced. In this case $n=2$ and $\ell_1$ is the curve separating blue from cyan, and $\ell_2$ is the curve separating cyan from green }
	\label{f.iteration}
\end{figure}

This construction would work exactly in the same way, if instead of starting with $0$-boundary condition, we would start with boundary conditions given by $u_N$, a bounded harmonic function with value $0$ in $\partial_o$ and $2\lambda N$ in $\partial_i$, with $N \in \Z$. In other words, we could take $A_1^N:=\A^{u_N,\partial_o}_{-2\lambda, 2\lambda}$ and then iterate starting from step (2). We denote the loops obtained this way by $\ell^N_0 = \partial_o, \ell^N_1, \ell^N_2, \dots, \ell^N_n, \ell^N_{n+1} = \partial_i$.

We have the following basic claim about the sequence of loops and their boundary values:

\begin{lemma}\label{lembasic}
Take $N\in \Z$. The sequence $\ell^N_0 = \partial_o, \ell^N_1, \ell^N_2, \dots, \ell^N_n, \ell^N_{n+1} = \partial_i$ is almost surely finite and the boundary condition on the last boundary $\ell_n^N$ before $\partial_i$ is equal to $2\lambda N$. 
\end{lemma}

\begin{proof}
They key input for this lemma is a correspondence of this sequence of non-contractible loops with the stopping times of a Brownian bridge, stemming from Proposition \ref{p.LawELBddTVS}.

Indeed, let $\alpha_j^N$, denote the labels of the contours $\ell_j^N$, i.e. the boundary values of $\Phi_{A_j^N}+u_N$ on $\ell_j^N$ inside $\An^{j,N}$, the connected component of $\An\backslash A_j^N$ that has $\partial_i$ in its boundary. Consider also a Brownian bridge $\widehat B_t$  of time-length $L := \ED(\partial_i,\partial_o)$ going from $0$ to $2N\lambda$ and stopping times $\widehat T^{(j)}$ defined recursively by $\widehat T^{(0)} = 0$, $$\widehat T^{(j)} = \inf \{t\geq \widehat T^{(j-1)}:\widehat B_t = \widehat B_{\widehat T^{(j)}} \pm 2\lambda\} \wedge L.$$ 

Now, by Proposition  \ref{p.LawELBddTVS}, the labels $\alpha_j^N$ are exactly equal to the values of this Brownian bridge at the stopping times $\widehat T^{(j)}$. Thus, as the Brownian bridge has a.s. only finitely many jumps of size $2\lambda$, we deduce that $n$ is almost surely finite. Moreover, notice that the last jump is to $2\lambda N$, because the Brownian bridge ends at $2\lambda N$ and almost surely visits this endpoint after the last visit to $2\lambda (N\pm1)$
\end{proof}
As a simple corollary of the proof of this lemma, we see that
\begin{cor}\label{corlst}
The extremal distance between $\ell_n$ and $\partial_i$ has the same law as the random time $\widehat \tau_{-2\lambda, 2\lambda}$ described in \eqref{e.hattau}.
\end{cor}

Now, one can also explore similar contours starting from the inner boundary of a GFF with boundary values given by $u_N' := u_N-2N\lambda$, i.e. setting the values to $0$ on $\partial_i$ and to $-2N\lambda$ on $\partial_o$. Indeed, we could start with $A_1^N:=\A^{u_N-2N\lambda,\partial_i}_{-2\lambda, 2\lambda}$ and then iterate as above, with the only difference that the exploration goes from interior towards the outer boundary. Let us denote by $\overleftarrow\ell^N_0 = \partial_i, \overleftarrow\ell^N_1, \overleftarrow\ell^N_2, \dots, \overleftarrow\ell^N_n, \overleftarrow\ell^N_{n+1} = \partial_o$ the non-contractible loops obtained this way, and let $\overleftarrow{\alpha_i}$ denote the boundary values of $\Phi_{\overleftarrow A_j}$ on $\overleftarrow\ell_j^N$ towards $\partial_o$. 

The following theorem is the central result of this section and states that the contours described above coincide in the reverse order. This theorem is proved in the next subsection.

\begin{thm}[inside $\to$ outside = outside $\to$ inside]\label{thmRV}
	Let $\Phi$ be a GFF on $\An$ with zero boundary conditions. 
	Then we have that almost surely $$(\ell_n^N, \ell_{n-1}^N, \dots, \ell_1^N) = (\overleftarrow\ell_1^N, \overleftarrow\ell_2^N, \dots, \overleftarrow \ell_n^N).$$ 
\end{thm}

Let us remark that for us the relevant part of the theorem is the equality in law.
\begin{cor}[Reversibility in law]\label{correv}
Let $\varphi$ be a conformal automorphism of $\An$ that swaps the two boundary components. Then, the sequence $(\varphi(\ell^N_n), \varphi(\ell^N_{n-1}), \dots, \varphi(\ell_1^N))$ has the same law as $(\ell_1^N, \ell_2^N, \dots, \ell_n^N)$. 
\end{cor}
\begin{proof}
	Due to the fact that a GFF $\Phi$ is conformally invariant, we have that $\widehat \Phi:=\Phi\circ \varphi$ is a GFF in $\An$ where the roles of the boundary components have been swapped. Thus, by Theorem \ref{thmRV} $(\varphi(\overleftarrow\ell_1^N), \varphi(\overleftarrow\ell_2^N), \dots, \varphi(\overleftarrow \ell_n^N))$ corresponds to $(\hat \ell_1^{-N}, \hat \ell_2^{-N},..., \hat \ell_n^{-N})$. The result follows from the fact that this last $n$-tuple has the same law as $( \ell_1^{N}, \ell_2^{N},..., \ell_n^{N})$ since $-\Phi$ has the same law as $\Phi$.
\end{proof}

In fact, Proposition \ref{mlaw0} follows directly from this corollary.

\begin{proof}[Proof of Proposition \ref{mlaw0}]
By Corollary \ref{correv} the law of $\ED(\ell, \partial_o)=\ED(\ell_1, \partial_o)$ equals the law of $\ED(\overleftarrow \ell_n, \partial_o)$. But this in turn equals the law of $\ED(\varphi(\ell_n), \partial_i)$, which by conformal invariance of the GFF and Corollary \ref{corlst} equals the law of 
$\widehat \tau_{-2\lambda, 2\lambda}$.
\end{proof}

It remains to prove Theorem \ref{thmRV}, which is the content of the next subsection.

\subsection{The proof of reversibility}\label{Ss.Reversibility}
In this subsection, we prove a commutativity statement about the iterated TVS exploration described at the beginning of Section \ref{ssrv}. Theorem \ref{thmRV} is an immediate consequence of this commutativity result. 

We follow the same convention as in Section \ref{ssrv}. That is to say, consider an annulus $\An$ and a GFF with boundary conditions $0$ on the outer boundary and $2\lambda N$ on the inner boundary. Let $A_j^N$ denote the $n$ times iterated $\A_{-2\lambda, 2\lambda}^{u_N,\partial_o}$ and let $\overleftarrow A_k^N$ denote the $k$ times iterated $\A_{-2\lambda, 2\lambda}^{u_N-2N\lambda,\partial_i}$. By convention $A_0 = \partial_o$ and $\overleftarrow A_0 = \partial_i$. Furthermore, let $A_{j,k}^{N}$ be the union of $A_j^N$ and $\overleftarrow A_k^N$. 

Let us note that by definition $A_{j,k}^N$ is increasing in the sense that $A_{j,k}^N \subseteq A_{j',k'}^N$ whenever $j \leq j'$ and $k \leq k'$. Additionally if $A_{j,k}^N$ does not connect $\partial_o$ with $\partial_i$, there is exactly one connected component of $\An_r \backslash A_{j,k}^N$ that has the topology of an annulus. 

The main proposition of this section states that there is a unique way of connecting $\partial_i$ and $\partial_o$ using $A_{j,k}$. 
\begin{prop}[Commutativity]\label{p.comm1}
Suppose that $n \in \N$ is the smallest number so that $A_{n+1}^N$ is connected to $\partial_i$. Then $A_{j,k}^N = A_{n+1,0}^N$ for all $j, k\geq 0$ such that $j + k = n+1$. Moreover, in this case \begin{equation}\label{e.intersectionA}
\overleftarrow A_{j}^N \cap A_{k}^N = \ell_{j}^N = \overleftarrow\ell_{k}^N.\end{equation}
\end{prop}
\begin{rem}
	Notice that Theorem \ref{thmRV} follows directly from the equality in \eqref{e.intersectionA}.
\end{rem}

The proof of this proposition is based on two commutativity lemmas for TVS. These lemmas state some conditions under which we can construct $\overleftarrow A_1^N$ in a different way: by first sampling $A_1^N$ and then constructing a particular two-valued local set of $\Phi^{A_1^N}$.

The first of the two commutativity lemmas deals with all types of TVS starting from different connected components of the boundaries. Let $a,a',b,b'$ such that $a+b,a'+b'\geq 2\lambda$ and let $u$ be a harmonic function with constant boundary values $v_o$ and $v_i$ in $\partial_o$ and $\partial_i$ respectively. Assume furthermore that $-a<v_o < b$ and $-a'<v_i< b'$. We will construct a local set $A'$ connected to $\partial_i$ and state sufficient conditions for it to be equal to $\A_{-a',b'}^{u, \partial_i}$:
\begin{itemize}
	\item Start by sampling $A=\A_{-a,b}^{u,\partial_o}$. If $A$ does not intersect $\partial_i$, denote by $O$ the unique connected component of $\An\backslash O$ that has the topology of an annulus.
	\item 	Recall that $\Phi+u$ restricted to $O$ has the law of $\Phi^A\mid_O+u'$, where $\Phi^A\mid_O$ is a GFF in $O$ and  $u'$ is the bounded harmonic function in $O$ with boundary values $v_i$ in $\partial_i$ and constant either $-a$ or $b$ in $A$.
	\item Now explore $A'=\A_{-a',b'}^{u',\partial_i}(\Phi^A\mid_O)$, the connected component of the local set $\A_{-a',b'}^{u'}(\Phi^A\mid_O)$ containing $\partial_i$. We set $A'=\emptyset$ if $O=\emptyset$.
\end{itemize}

The first commutativity lemma roughly says that when $A'$ does not touch $A=\A_{-a,b}^{u,\partial_o}$, then it indeed equals $\A_{-a',b'}^{u, \partial_i}$.
\begin{lemma}\label{l.comm0} Up to a null-set
	\[\{\A_{-a,b}^{u,\partial_o}\cap \A_{-a',b'}^{u,\partial_i}=\emptyset \neq A'\}=\{\A_{-a,b}^{u,\partial_o} \cap A' =\emptyset\neq A' \}. \]
	Furthermore, on the event $\{\A_{-a,b}^{u,\partial_o} \cap A' =\emptyset\neq A' \}$ we have that $A'=\A_{-a',b'}^{u,\partial_i}$.
\end{lemma}

\begin{proof}

	The proof is similar in spirit to that of Claim \ref{c.new measure}, so we will be brief here.
	
	Indeed, we first construct the local set $A = A_{-a,b}^{u,\partial_o}$. We only need to work on the event that it does not touch $\partial_i$. On this event, we define $\widehat \An$ as the only connected component of $\An\backslash A$ that is not simply connected. 
	
	Now, consider the level line construction of $A_{-a',b'}^{u, \partial_i}$, as described in Section 3.4 of \cite{ALS1}. Similarly to the proof of Claim \ref{c.new measure}, we can stop this construction once it gets to distance $\eps > 0$ of $A$. Note that by Lemma \ref{BPLS} (2), these are level lines of both $\Phi + u$ and of $\Phi^A +u + h_A$ on $\widehat \An$ and thus we obtain a local set $A_\eps$ for both fields. By construction, we have that $A \subseteq A_\eps \subseteq A \cup A_{-a',b'}^{u, \partial_i}$. Further, as these level lines are also exactly the level lines used in the construction of $\A_{-a',b'}^{u+h_A, \partial_i}(\Phi^A\mid_{ \widehat \An})$ for the field $\Phi^A +u + h_A$ (in Section 3.4 of \cite{ALS1}), we also have that $A_\epsilon \backslash A\subseteq \A_{-a',b'}^{u+h_A, \partial_i}(\Phi^A\mid_{ \widehat \An})$. Finally, if $A_\epsilon\backslash A$ is at distance strictly bigger than $\epsilon$ from $A$ we have that, by construction, $ \A_{-a',b'}^{u, \partial_i}(\Phi)=A_\epsilon \backslash A = \A_{-a',b'}^{u+h_A, \partial_i}(\Phi^A\mid_{ \widehat \An})$. We conclude by the fact that the above holds for all $\epsilon>0$.

\end{proof}

 When $A'$ and $\A_{-a,b}^{u,\partial_o}$ do intersect, it is not generally true that the local sets $A'\cup \A_{-a,b}^{u,\partial_o}$ and $\A_{-a',b'}^{u,\partial_i} \cup \A_{-a,b}^{u,\partial_o}$ are equal. The second lemma of commutativity gives a condition under which this commutativity does hold.

\begin{lemma}\label{l.comm2}
We work in the same context as Lemma \ref{l.comm0}. Take $v_o=0$ and $v_i \in 2\lambda \Z$. Furthermore, let $a=b=2\lambda$, $a'=-2\lambda+v_i$ and $b'=2\lambda+v_i$. We have that a.s. \[A'\cup \A_{-2\lambda,2\lambda}^{u,\partial_o}=  \A_{-a',b'}^{u,\partial_i} \cup \A_{-2\lambda,2\lambda}^{u,\partial_o}.\] 
Furthermore, if $A'\neq \emptyset$ and $\tilde \ell=\A_{-a',b'}^{u,\partial_i}\cap \A_{-2\lambda,2\lambda}^{u,\partial_o}\neq \emptyset$, we have that $\tilde \ell$ is equal to $\ell^N_1$, the outer boundary of the annular connected component of $D\backslash \A_{-2\lambda,2\lambda}^{u,\partial_o}$. This can only happen when $v_i=\pm 2\lambda$.
\end{lemma}

\begin{rem}
Lemma \ref{l.comm2} is the main reason why in our commutativity results the maximum generality we can hope is to have $a,b$ with $a+b=2K\lambda$. This lemma is in fact not true, when $a+b\neq 4\lambda$.
\end{rem}
\begin{proof}[Proof of Lemma \ref{l.comm2}:]
By Lemma \ref{l.comm0} we just need to study the cases where $\{\A_{-2\lambda,2\lambda}^{u,\partial_o}\cap \A_{-a',b'}^{u,\partial_i}=\emptyset\}$ or $\{A'=\emptyset\}$. We start by separating into cases according to the value of \avelio{$v_i=2K\lambda$}. 

\textbf{Case $|K|\geq 2$}: In this case, Proposition \ref{p.LawELBddTVS} implies that a.s. $\A_{-2\lambda,2\lambda}^{u,\partial_o}\cap \A_{-a',b'}^{u,\partial_i}=\emptyset$ and thus the result follows from Lemma \ref{l.comm0}.

\textbf{Case $K=0$}: Note that in this case we have that $a'=b'=2\lambda$ and thus, if $\A_{-2\lambda,2\lambda}^{u,\partial_o}\cap \A_{-2\lambda,2\lambda}^{u,\partial_i}\neq \emptyset$ we have that $\A_{-2\lambda,2\lambda}^{u,\partial_o}\cap \A_{-a',b'}^{u,\partial_i}=\A_{-2\lambda,2\lambda}$, which implies that the result is true. Furthermore, let us note that in this case $A'=\emptyset$.

\textbf{Case $|K|=1$}: This is the core of the lemma. WLOG we can consider $K=1$, i.e. an annulus with boundary conditions $0$ on the outer boundary and $2\lambda$ on the inner boundary. Thus, $\A_{-2\lambda,2\lambda}^{u,\partial_o}=A_1^1$ and $\A_{-a',b'}^{u,\partial_i}=\A_{0,4\lambda}^{u,\partial_i}=\overleftarrow A^1_1$. Furthermore, when it is necessary we can assume that we are on the event $\A_{-a',b'}^{u,\partial_i}\cap \A_{-2\lambda,2\lambda}^{u,\partial_o}\neq \emptyset$.

Observe first that by Proposition \ref{p.LawELBddTVS}, $A_1^1 \cap \partial_i = \emptyset$ and $\overleftarrow A_1^1 \cap \partial_o = \emptyset$. Thus both $\An \backslash A_1^1$ and $\An \backslash \overleftarrow A_1^1$ have one annular component, denoted by $\An^{1,1}$ and $\overleftarrow\An^{1,1}$, respectively. In particular $A'\neq \emptyset$. Recall that we denote the outer boundary of $\An^{1,1}$ by $\ell_1^1$ and the inner boundary of $\overleftarrow\An^{1,1}$ by $\overleftarrow\ell_1^1$.

The rest of the lemma is easy, once we establish that $A_1^1$ and $\overleftarrow A_1^1$ share their only non-contractible loop.

\begin{claim} \label{c.intersection_equality}
If $A_1^1\cap \overleftarrow A_1^1\neq \emptyset$, then $\ell_1^1 =\overleftarrow\ell_1^1$. 
\end{claim}

Let us first show how the lemma follows from this claim. Indeed, given the claim, Lemma \ref{BPLS} property (2) implies that $\overleftarrow A_1^1$ is a local set in $ \An^{1,1}$ of the GFF $\Phi^{\A_{-2\lambda,2\lambda}^{u,\partial_o}}$ restricted to $ \An^{1,1}$. We need to prove that $\overleftarrow A_1^1$ is equal to $\overleftarrow A_1' := \A_{0,4\lambda}^{u',\partial_i}(\Phi^{\A_{-2\lambda,2\lambda}^{u,\partial_o}}\mid_{\An^{1,1}})$, where $u'=\Phi_{\A_{-2\lambda,2\lambda}^{u,\partial_o}} + u$ restricted to $\An^{1,1}$. We know that this is the case when $\overleftarrow A_1^1\cap A_1^1= \emptyset$, so we can assume $\overleftarrow A_1^1\cap A_1^1\neq \emptyset$. Note that this can only happen when $u'\equiv2\lambda$. 

Now, the Minkowski dimension of $\overleftarrow A_1^1$ is strictly smaller than $2$ and thus Proposition \ref{p.thin} implies it is thin. 
So by Theorem \ref{t.tvs} it remains to prove that $\overleftarrow A_1^1$ satisfies the condition \hyperlink{tvs}{(\twonotes)} for $\Phi^{\A_{-2\lambda,2\lambda}^{u,\partial_o}}\mid_{\An^{1,1}}$. 

First, note that by Claim \ref{c.intersection_equality}, the connected components of $\An_1^1 \backslash \overleftarrow A_1^1$ are just all the simply connected components of $\An \backslash \overleftarrow A_1^1$. Now, we have that $\Phi_{\overleftarrow A_1^1}$ restricted to any such component $O$ is constant equal to $0$ or $4\lambda$. From the construction of TVS in \cite{ALS1}, we know that the boundary of each component $O$ is disjoint from $\overleftarrow \ell_1^1 = \ell_1^1$. Thus by part (2) of Lemma \ref{BPLS} we have that $\Phi_{A_1^1\cup \overleftarrow A_1^1} = \Phi_{\overleftarrow A_1^1}$ for all such connected components $O$ and we conclude that $\overleftarrow A_1^1$ satisfies condition \hyperlink{tvs}{(\twonotes)} for $\Phi^{\A_{-2\lambda,2\lambda}^{u,\partial_o}}\mid_{\An^{1,1}}$.

Let us now turn to the proof of the claim.
\begin{proof}[Proof of Claim \ref{c.intersection_equality}:]
We use the SLE$_4(-2)$ based construction of $\ell_1^1$ and $\overleftarrow \ell_1^1$ detailed in Section \ref{ss.construction SLE}. Let these SLE$_4(-2)$ type of processes constructing $A_1^1$, resp. $\overleftarrow A_1^1$ be denoted by $\nu_t$, resp. 
$\overleftarrow \nu_t$. 

 Take $\tau$ any stopping time of $\nu$ where $\nu$ closes a contractible loop and where $\nu_\tau$ does not hit $\partial_i$ and let $\An_\tau$ denote the  only connected component of $\An\backslash \nu_\tau$ that is not simply connected. Let us now recall that $\nu_\tau$ is a local set where $h_{\nu_\tau}+u$ restricted to $\An_\tau$ is the unique bounded harmonic function  with boundary conditions $0$ on $\partial_o \cup \nu_\tau$ and $2\lambda$ on $\partial_i$. Thus by Proposition \ref{p.LawELBddTVS}, we see that the local set $\A_{0, 4\lambda}^{\partial_i}$ sampled in $\An \backslash \nu_t$ stays at a positive distance from $\nu_t$ almost surely. Moreover, exactly as in the proof of Claim \ref{c.new measure}, we can argue that in fact this local set is equal to $\overleftarrow A_1^1$. As a consequence, we obtain that a.s. $\overleftarrow A_1^1$ stays at positive distance of $\nu_t$ for any $t$ which corresponds to finishing a contractible loop. Then $\overleftarrow A_1^1$ is at positive distance of $\nu_t$ for any $t < T$, where $T$ is the time where $\nu$ starts tracing the non-trivial loop $\ell_1^1$. Let $T_\ell$ be the time when $\nu$ finishes tracing $\ell_1^1$  
 
 By exactly the same argument, if we denote $S$ to be the time when $\overleftarrow \nu$ starts tracing the non-trivial loop $\overleftarrow \ell_1^1$, then $A_1^1$ stays at positive distance of $\overleftarrow \nu_s$ for any $s < S$. 

First, let us argue that $\overleftarrow \nu$ cannot enter in the interior of the annulus $\An \backslash \An_1^1$. To do this, first notice that no non-trivial loop of $\overleftarrow \nu$ is inside $\An \backslash \An_1^1$ - indeed, any such loop is finished before the time $S$. Thus $\overleftarrow \nu$ can only enter $\An \backslash \An_1^1$ when it is tracing the non-contractible loop $\overleftarrow \ell_1^1$. 

Now take $s$ a rational time such that with positive probability $\nu(s)$ is inside the interior of the annulus $\An\backslash \An_1^1$. By the argument just above, it then has to enter some connected component $O$ of $\An \backslash \nu_{T_\ell}$ contained in the annulus $\An\backslash \An_1^1$. By Proposition \ref{p.GFFcpl2ann} $(\nu_t)_{t\geq s}$ is in the process of tracing a generalized level line with boundary values $2\lambda$ and $0$ or $2\lambda$ and $4\lambda$. From Claim 17 of \cite{ASW} we thus see that if $\ell^1$ enters any such connected component $O$ with positive probability, then it only intersects its boundary at one point corresponding both to its starting and endpoint. But now recall that $\overleftarrow \ell_1^1$ does not intersect $\nu_t$ for any $t<T$. Thus it means that the only possibility for $\ell^1$ to form a non-contractible loop in $\An \backslash \An_1^1$ is depicted in Figure \ref{f.nli}, as for any other $O$ the closure does not disconnect $\partial_i$ and $\partial_o$.

 \begin{figure}[h!]
 	\includegraphics[width=0.3\textwidth]{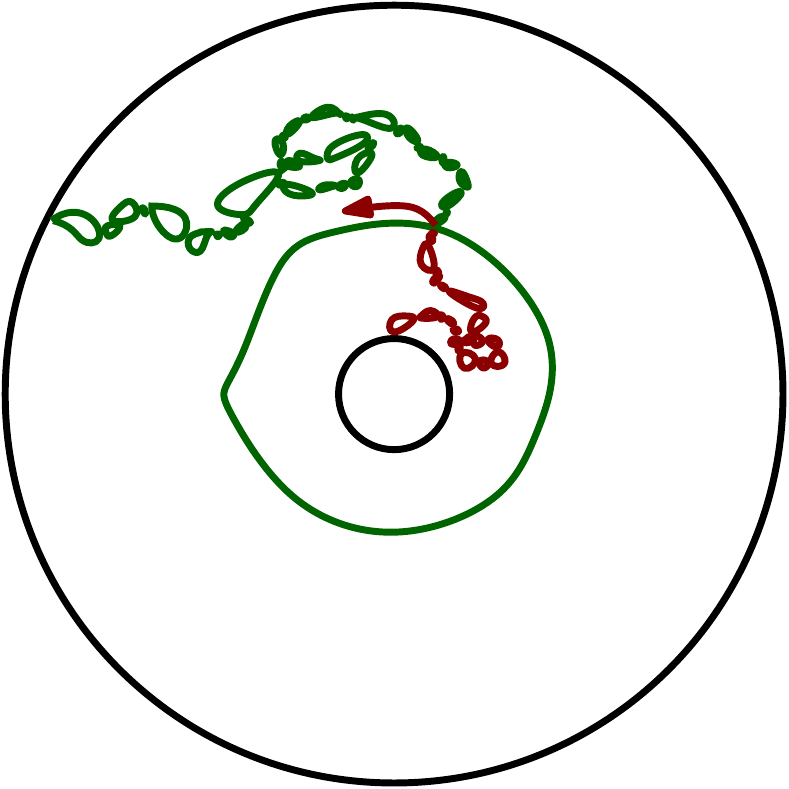}
 	\caption{The green set represents $A_1^1$ and the brown set represents $ \protect\overleftarrow \nu_s$}
 	\label{f.nli}
 \end{figure}
 Now, we note that the generalized level line $\overleftarrow \nu_t$ enters the interior of a component $O$ by one prime end of the boundary. Thus by the proof of Lemma 10 in \cite{ASW} \footnote{In that statement `point' should read as a `prime end'} the boundary conditions near the other prime end will be $0$ at any rational time when $\nu_t$ is inside $\An \backslash \An_1^1$. But then by Lemma \ref{notouch}, the generalized level line cannot exit through this other prime end.

 The above discussion implied that $\overleftarrow\ell_1^1$ is contained in the closure $\An_1^1$. We still need to prove that $\overleftarrow \ell_1^1$ does not intersect the interior of $\An_1^1$. However, this can be proved by exactly the same argument. Indeed, by hypothesis $\overleftarrow \ell_1^1$ needs to intersects $A_1^1$. Now recall that $\overleftarrow \nu_t$ is disjoint from $\An_1^1$ for all $t < T$. Moreover, when $\overleftarrow \nu_T$ is also disjoint from $\An_1^1$, then as between the times $T$ and $T_\ell$, the process $\overleftarrow \nu_t$ is tracing a generalized level, it cannot hit $\ell_1^1$ that has boundary values $\pm 2\lambda$. Thus the only possibility for $\overleftarrow \ell_1^1$ to intersect $A_1^1$ is depicted by Figure \ref{f.nli2}. But now we can conclude that this cannot happen by arguing exactly as in the paragraph above. 
 \begin{figure}[h!]
 	\includegraphics[width=0.3\textwidth]{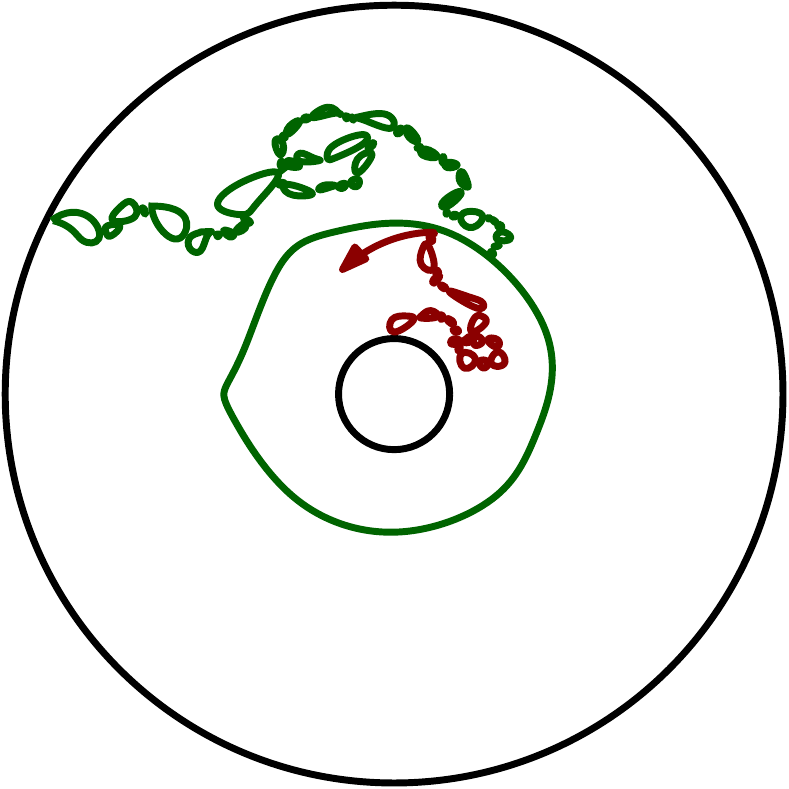}
 	\caption{The green set represents $A_1^1$ and the brown set represents $\nu_s$}
 	\label{f.nli2}
 \end{figure}
 
\end{proof}

\end{proof}

The next corollary is a certain iteration of the result.
\begin{cor}\label{cor.iterate}
Suppose that $A_{j,k}^N$ does not connect the inner and outer boundary and consider the annular connected component $\widetilde \An$ of $\An \backslash A^N_{j,k}$. Let $K_{j,k}$ be equal to $(2\lambda)^{-1}(\alpha_j - \overleftarrow{\alpha_k})$ \footnote{Recall that $\alpha_j$, resp. $\overleftarrow \alpha_k$ is the boundary value on $\ell_j^N$, resp. $\overleftarrow\ell_k^N$, of $h_{A_j^N}+u_N$, resp. $h_{\overleftarrow A_k^N}+u_N$ restricted to $\An_j^N$, resp. $\overleftarrow\An_k^N$.}. Then $A_{j+1,k}^N \backslash A_{j,k}^N$ is equal to $A_{1}^{K_{j,k}}$ of the GFF $\Phi^{A^N_{j,k}}$ in $\widetilde \An$.   Similarly, we have that $A_{j,k+1}^N\backslash A_{j,k}^N$ is equal to $\overleftarrow A_1^{K_{j,k}}$ of the GFF $\Phi^{A_{j,k}}$ in $\widetilde \An$. 
\end{cor}

\begin{proof}[Proof of Corollary \ref{cor.iterate}:]
We will prove this by induction on $j+k\geq 0$. The claim is true for $j = k = 0$ by definitions. The claim for $j=1, k = 0$ (and then similarly for $j = 0$ and $k=1$) follows from Lemma \ref{l.comm2}: indeed, the lemma directly implies that $A_{1,1}^N \backslash A_{1,0}^N$ equals $A_{0,1}^0 = \overleftarrow A_1'$ of the relevant GFF. Moreover, $A_{2,0}^N \backslash A_{1,0}^N$ equals $A_{1,0}^0 = A_1'$ just by construction of the iterated TVS.

Now suppose that the claim holds for all $j+k\leq m$, with $m > 1$ and let us consider $j+k = m+1$. WLOG we can assume $j \geq 1$. As $A_{j,k}^N$ is assumed not to connect inner and outer boundaries, by (an iterated version of) Lemma \ref{l.comm0}, it follows that $A_{j,k}^N\backslash A_{1,0}^N$ equals in law with $A_{j-1,k}^{N'}$ of the GFF in the non-simply-connected component of $\An \backslash A_{1,0}^N$, where $N' = N \pm 1$, depending on the label of $A_{1,0}$ towards the inner boundary. As by construction $A_{j+1,k}^N\backslash A_{1,0}^N$ has the law of $A_{j,k}^{N'}$, we can now use the induction hypothesis for $A_{j-1,k}^{N'}$ to conclude. 
\end{proof}

We are, finally, ready to prove the proposition.

\begin{proof}[Proof of Proposition \ref{p.comm1}:]

We will prove the proposition again by induction on $n$. The base case is $n=0$: by Proposition \ref{p.LawELBddTVS}, $A_{1,0}^N = \A_{-2\lambda, 2\lambda}^{u_N, \partial_o}$ can only hit the inner boundary if $N = 0$. In this case $A_{1,0}^0 = \A_{-2\lambda,2\lambda} =  A_{0,1}^0$ by Corollary \ref{c.tvs from a boundary} and $\overleftarrow A_1^0 \cap A_0^0 = \partial_o$ and $\overleftarrow A_0^0 \cap A_1^0 = \partial_i$.

So suppose now that the statement is true for $n \leq m$ and we want to prove it for $n= m+1$. We will show that $A_{m+1, 0}^N = A_{m,1}^N$, and by a similar argument it then follows that $A_{j, m+1-l}^N = A_{j-1, m+2-j}^N$. 

As by hypothesis $A^N_{m,0}$ does not connect the two boundaries, by Corollary \ref{cor.iterate} we see that $A_{m+1,0}^N\backslash A_{m,0}^N$ equals $A_{1,0}^{K_{m,0}}$ of the relevant GFF $\Phi^{A_{m,0}^N}$ in $\An	 \backslash A_{m,0}^N$. Because $A_{m+1,0}^N$ connects the two boundaries, Proposition \ref{p.LawELBddTVS} implies that $K_{m,0}$ has to be equal to $0$. Similarly $A_{m,1}^N\backslash A_{m,0}^N$ equals $\A_{0,1}^{0}$ in $\An_r \backslash A_{k,0}^N$ for the same GFF as above. But by the induction base $A_{1,0}^0 = A_{0,1}^0$ a.s. and we conclude that $A_{m+1,0}^N = A_{m,1}^N$. As moreover $A_m^N = A_{m,0}^N$ and $\overleftarrow A_1^N = A_{0,1}^0$, the claim on the boundaries also follows from the base case $n=1$. 
\end{proof}

\subsection{Extensions to general boundary conditions and to first passage sets}\label{Ss.Extensions to non-zero}

In this subsection we will first generalise Proposition \ref{mlaw0} to boundary conditions differing on the inner and outer boundaries. This allows us prove the result also for general two-valued sets with $a + b \in 2\lambda \N$. We finally deduce the result for first passage sets.

\subsubsection{General boundary conditions for the TVS $\A_{-2\lambda, 2\lambda}$}

Consider again an annulus $\An$ with outer and inner boundaries denoted by $\partial_o$ and $\partial_i$.
Denote by $\Phi$ a zero-boundary GFF on $\An$ and for $v_i,v_o \in \R$, let $u_{v_i,v_o}$ denote the bounded harmonic function that is equal to $v_0\in (-2\lambda,2\lambda)$ on $\partial_o$ and $v_i$ on $\partial_i$. Consider now the two-valued set $\A^{u_{v_o,v_i}}_{-2\lambda,2\lambda}$ of $\Phi + u_v$ on the annulus $\An$. 

Let $\ell^{v_o,v_i}$ be the boundary of the connected component of $\A_{-2\lambda, 2\lambda}^{u_{v_o,v_i}}$ connected to $\partial_o$ (in case there is no non-trivial loop surrounding $\partial_i$ we again set $\ell^{v_o,v_i} = \partial_i$). We want to calculate $\ED(\ell^{v_o,v_i}, \partial_o)$ as before. 

To do this, let 
$\widehat{B}^{v_0,v_i}$ 
be a Brownian bridge from 
$v_0$ to $v_i$, whose time-duration $L$ is given by the extremal distance of $\An$. Define
	\begin{equation*}
	\widehat{T}^{v_o, v_i}_{-2\lambda,2\lambda}
	:=\inf\{s\geq 0: |\widehat B^{v_o,v_i}_s|=2\lambda\} \wedge L
	\end{equation*}
	If 
	$\widehat{T}^{v_o, v_i}_{-2\lambda,2\lambda}= L$, 
	we set 
	$\widehat{\tau}^{v_o,v_i}_{-2\lambda,2\lambda} = L$. If, however 
	$\widehat{T}^{v_o, v_i}_{-2\lambda,2\lambda}<L$, 
	we define
	\begin{equation*}
	\widehat{\tau}^{v_o,v_i}_{-2\lambda,2\lambda}:=
	\sup\{0\leq s \leq 
	\widehat{T}^{v_o, v_i}_{-2\lambda,2\lambda}:
	\widehat B^{v_o,v_i}_s=0\} \vee 0.
	\end{equation*}
We then generalize Proposition \ref{mlaw0} as follows.
\begin{prop}\label{mlawv}
	The law of $\ED(\ell^{v_o,v_i}, \partial_o)$ equals that of $\widehat{\tau}^{v_o,v_i}_{-2\lambda,2\lambda}$.
\end{prop}	

We start again from a result on reversibility. For simplification, in this stage we will take $v_o=0$, $v_i=v\in (2\lambda(N-1), 2N\lambda]$ and $u_v=u_{0,v}$. As before, we denote the iterations of $\A_{-2\lambda,2\lambda}^{u_v,\partial_o}$ from the outer boundary by $A_1^{u_v}, A_2^{u_v}, \dots, A_{n}^{u_v}$, where $A_n^{u_v}$ is the last iteration that does not touch $\partial_i$. We also define $\ell_0^{u_v}=\partial_o,\ell_1^{u_v},..,\ell_n^{^{u_v}},\ell_{n+1}^{u_v}=\partial_i$ as the non-contractible contours.

To construct the iterations from the interior boundary, let us first define $\overleftarrow A^{u_v}_0 := \A_{2\lambda(N-1), 2\lambda N}^{u_v,\partial_i}$. Now, if $\overleftarrow A^{u_v}_0$ intersects $\partial_o$ we finish the construction. If $\overleftarrow A^{u_v}_0\cap \partial_o\neq \emptyset$, define $\overleftarrow\An^{u_v}_0$ as the connected component of $\An\backslash \overleftarrow A^{u_v}_0$ that has the topology of the annulus, furthermore define $\overleftarrow\ell_0^{u_v}$ as the inner boundary of $\overleftarrow\An_0^{u_v}$. As a consequence of Proposition \ref{p.intersection_boundary}, $\overleftarrow \ell_0^{u_v}$ always intersects $\partial_i$. Now, let us work in $\overleftarrow \An_n^0$. Like in the Subsection \ref{ssrv}, we define $\overleftarrow \ell^{u_v}_1, \dots, \overleftarrow \ell^{u_v}_{n+1} = \partial_o$, as $ \overleftarrow \ell^{\tilde N}_1, \dots, \overleftarrow \ell^{\tilde N}_{n+1} = \partial_o$ of the GFF $\Phi^{\overleftarrow A^{u_v}_0}$, where $\tilde N\in \{N-1, N\}$ is $(2\lambda)^{-1}$ times the boundary value on $\ell_0^{u_v}$ of $h_{\overleftarrow A^{u_v}_0}+u_v$.

\begin{prop}\label{rev:bg}
If $\overleftarrow A^{u_v}_0$ intersects $\partial_o$ then $A_1^{u_v}$ intersects $\partial_i$. Furthermore, we have that almost surely $$(\ell_n^{u_v}, \ell_{n-1}^{u_v}, \dots, \ell_1^{u_v}) = (\overleftarrow\ell_1^{u_v}, \overleftarrow\ell_2^{u_v}, \dots, \overleftarrow \ell_n^{u_v}).$$ 
\end{prop}

\begin{proof}
For the first part, let us note that if $\overleftarrow A_0^{u_v}$ intersects $\partial_o$ then $v\in(-2\lambda,2\lambda)$, and thus $\overleftarrow A_0^{u_v}$ is equal to either $\A_{-2\lambda,0}^{u_v}$ and $\A_{0,2\lambda}^{u_v}$. As in both cases $\overleftarrow A_0^{u_v}\subseteq\A_{-2\lambda,2\lambda}^{u_v}=A_1^{u_v}$, we conclude.

We work now in the case where $\overleftarrow A_0^{u_v}$ does not touch $\partial_o$. In this case, recall that $\Phi + u_v$ restricted to $\overleftarrow \An_0^{u_v}$ is equal to $\Phi^{\overleftarrow A_0^{u_v}}\mid_{\overleftarrow \An_0^{u_v}}+u' $ where $u'$ is the unique bounded harmonic function in $\overleftarrow A_0^{u_v}$ with boundary values $0\in \partial_0$ and $\tilde N$ in $\overleftarrow \ell_0^{u_v}$. For simplification, we write $\Phi^{\overleftarrow A_0^{u_v}}=\Phi^{\overleftarrow A_0^{u_v}}\mid_{\overleftarrow \An_0^{u_v}}$ the GFF in $\overleftarrow \An_0^{u_v}$. We now apply the Proposition \ref{p.comm1} to the iterated TVS of the GFF $\Phi^{\overleftarrow A_0^{u_v}}$. Now, as long as the iterated $\A_{-2\lambda, 2\lambda}^{u_v,\partial_o}$ from the outer boundary does not touch $\overleftarrow A_0^{u_v}$, we can deduce from (an iterated version of) Lemma \ref{l.comm0}, that the iterated $\A_{-2\lambda, 2\lambda}^{u_v,\partial_o}$ of the initial GFF $\Phi$ agrees with that of $\Phi^{\overleftarrow \An_0}$. And in particular, that the non-trivial loops coincide in reverse order.

It remains to argue that the first iteration where $\A_{-2\lambda, 2\lambda}^{u_v,\partial_o}$ touches $\overleftarrow A_0^{u_v}$ is also the first iteration when it touches $\partial_i$. To show this, it suffices to show that if the iterated $\A_{-2\lambda, 2\lambda}^{\partial_o}$ does not touch the boundary, then it does not touch $\overleftarrow A_0^{u_v}$. To see this, assume that $A_k^{u_v}$ does not intersect $\partial_i$ and, as above Lemma \ref{l.comm0}, define $A'$ to be the $\A_{-2\lambda(N-1), 2\lambda N}^{u',\partial_o}$ of the GFF $ \Phi^{A_k^{u_v}}$ restricted to $\An_k^{u_v}$. Here $u'$ is the unique bounded harmonic function with value $\alpha_k^{u_v}$ in $\ell_k^{u_v}$ and $v$ in $\partial_i$. From (an iterated version of) Lemma \ref{l.comm0} we know that if $A'$ does not intersect $\ell_k^{u_v}$, then neither does $\overleftarrow A_0^{u_v}$. But we have that $\alpha_k^{u_v}\in 2\lambda \Z$ and thus by Proposition \ref{p.LawELBddTVS} $A'$ does not touch $\ell_k^{u_v}$ and we conclude. 
\end{proof}

We will now prove Proposition \ref{mlawv}.
\begin{proof}[Proof of Proposition \ref{mlawv}]
Notice that by symmetry we may assume that $0<v_0<2\lambda$.

	As in the last section, whereas the loop of the TVS is constructed using an outside to inside exploration, the Brownian motion stems from an inside to outside exploration via Proposition \ref{p.LawELBddTVS}. Again, the two are joined by a reversibility argument.
		
	Our aim is reduce to boundary conditions of the form $2\lambda N$ on both boundaries. To do this, explore $A_o:=\A_{0,2\lambda }^{u_{v_o,v_i},\partial_o}$, the connected component of $\A_{0,2\lambda}^{u_{v_o,v_i}}$ connected to $\partial_o$. Notice that by monotonicity of TVS, we have that $A_o \subseteq \A_{-2\lambda, 2\lambda}^{u_{v_o,v_i},\partial_o}$.
	
	We divide the study in three cases.
	
	\textbf{Case 1:} All connected components of $\An\backslash A_o$ are simply connected. Then so are those of $\A_{-2\lambda, 2\lambda}^{u_{v_o,v_i}}$. Thus $\ell^{u_{v_o,v_i}}=\partial_i$. 
	
	Now, we are going to present the two other cases. As $A_o$ does not touch $\partial_i$ in these cases, there is unique a non-simply-connected component in $\An \backslash A_0$, which we denote by $\An_0$. Proposition \ref{p.intersection_boundary} implies that $\partial \An_0 \backslash \partial_i$ touches $\partial_o$.
	Notice that the boundary condition of $h_{A_0}+u_{v_0,v_1}$ restricted to $\An_i$ on the boundary $\partial \An_0 \backslash \partial_i$ is equal to either $0$ or $2\lambda$, which correspond to two further cases we are going to study.
	
	\textbf{Case 2:} The boundary condition is $2\lambda$. In this case, we can further explore $\A_{-2\lambda, 2\lambda}$ of $\Phi^{A_0}$ in all simply-connected components of $\An \backslash A_0$ with boundary conditions $0$. As a result we obtain a local set that is connected to $\partial_o$ and has boundary conditions $\pm 2\lambda$ on all its boundaries. Using uniqueness of TVS, it follows that this local set is in fact the connected component of $\A_{-2\lambda, 2\lambda}^{u_{v_o,v_i},\partial_o}$ and thus $\partial \An_0 \backslash \partial_i$ equals $\ell^{u_{v_o,v_i}}$.
	
	\textbf{Case 3:} The boundary condition is $0$. In this case, we have not yet discovered $\ell^{u_{v_o,v_i}}$ and again by uniqueness of the TVS, $\ell^{u_{v_o,v_i}}$ is given by the first non-contractible loop $\ell_1$ of the iterated $\A_{-2\lambda,2\lambda}$ of the GFF $\Phi^A_0$ restricted to $\An_0$. \\
	
	We have now identified the loop $\ell^{u_{v_o,v_i}}$ in the three cases. Let us show that they correspond respectively to the cases where 
$\widehat{\tau}^{v_o,v_i}_{-2\lambda,2\lambda} = L$, where 
$\widehat{\tau}^{v_o,v_i}_{-2\lambda,2\lambda} = 0$ and where 
$\widehat{\tau}^{v_o,v_i}_{-2\lambda,2\lambda}$ is non-trivial. 
As in the last section, the Brownian motion comes from the exploration of $\A_{-2\lambda, 2\lambda}^{u_{v_o,v_i}}$ from the inner boundary, via Proposition \ref{p.LawELBddTVS}.
	
	Pick $N \in \Z$ such that $v_i \in (2\lambda (N-1), 2\lambda N]$ and explore $\overleftarrow A_0:=\A_{2\lambda(N-1),2\lambda N}^{u_{v_o,v_i},\partial_i}$, the connected component of $\A_{2\lambda(N-1),2\lambda N}^{u_{v_o,v_i}}$ connected to $\partial_i$. If $\overleftarrow A_0$ intersects $\partial_0$, then one can see that we are in the case $1$ and by Proposition \ref{p.LawELBddTVS} indeed 
$\widehat{\tau}^{v_o,v_i}_{-2\lambda,2\lambda}=L$. 
	
	Otherwise, there is a non-trivial annulus $\overleftarrow \An_0$ in 
$\An \backslash \overleftarrow A_0$. We can now repeat the construction of $\overleftarrow{A}_1^{u_{v_o,v_i}},...,\overleftarrow{A}_n^{u_{v_o,v_i}}$ as in Proposition \ref{rev:bg} above. The proof of Proposition \ref{rev:bg} then implies that $\overleftarrow{A}_n^{u_{v_o,v_i}}$ does not intersect $A_0$ and that the boundary value of $h_{\overleftarrow{A}_n^{u_{v_o,v_i}}}+u_{v_0,v_i}$ on $\overleftarrow\ell_n^{u_{v_o,v_i}}$ is equal to that of $h_{A_0}+u_{u_{v_o,v_i}}$ on $\ell_0$. We can now conclude:
\begin{itemize}
	\item In the case $2$, we have that $\widehat{\tau}^{v_o,v_i}_{-2\lambda,2\lambda} = 0$. Indeed, we know that the label corresponding to the last loop from the interior, $\overleftarrow{\ell}_n^{v_0,v_1}$, is equal to $2\lambda$. Hence the Brownian bridge describing the (reverse) labels by Lemma \ref{lembasic} does not attain $0$ before hitting $v_0$.
	\item In the case $3$, by Proposition \ref{rev:bg} we can identify $\ell^{u_{v_o,v_i}}$ with $\overleftarrow\ell_n^{u_{v_o,v_i}}$. Moreover, we know that its label is equal to $0$. Thus we deduce the law of $\widehat{\tau}^{v_o,v_i}_{-2\lambda,2\lambda}$ from Lemma \ref{lembasic}.
\end{itemize}
\end{proof}

\subsubsection{The case of two-valued sets with $a+b=2N\lambda$}

We now prove the generalisation of Proposition \ref{mlaw0} to the case where $a+b\in 2\lambda\N$. 

To do that, let $\widehat{B}^{v_o,v_i}$ be a Brownian bridge from 
$v_o$ to $v_i$, whose time-duration $L$ is given by the extremal distance of $\An$. Define
\begin{equation*}
\widehat T^{v_o, v_i}_{-a,b}
:=\inf\{s\geq 0: \widehat B^{v_o,v_i}_s \in \{-a, b\}\} \wedge L
\end{equation*}
If $\widehat{T}^{v_o, v_i}_{-a,b}=L$, we set 
$\widehat{\tau}^{v_o,v_i}_{-a,b}=L$. If, however 
$\widehat{T}^{v_o, v_i}_{-a,b}<L$, we define
\begin{equation*}
\widehat{\tau}^{v_o,v_i}_{-a,b}:=
\sup\{0\leq s \leq \widehat{T}^{v_o, v_i}_{-a,b}
:|\widehat B^{v_o,v_i}_s - \widehat B^{v_o,v_i}
_{\widehat{T}^{v_o, v_i}_{-a,b}}|=2\lambda\}\vee 0
\end{equation*}

As before, let $\ell^{u_{v_o,v_i}}_{-a,b}$ be the non-trivial loop of $\A_{-a,b}^{u_{v_o,v_i}, \partial_o}$ when it exists, and if it there is no non-trivial loop let $\ell^{u_{v_o,v_i}}_{k,l}$ be $\partial_i$. We can relate the law of the extremal distance between $\ell^{u_{v_o,v_i}}_{k,l}$ to $\partial_o$ to a Brownian motion as follows.
\begin{prop}\label{prop:generalTVS}
	When $a+b\in 2\lambda \N$, the law of $\ED(\ell^{u_{v_o,v_i}}_{-a,b}, \partial_o)$ equals that of 
	$\widehat{\tau}^{v_o,v_i}_{-a,b}$.
\end{prop}

First, in the case $a +b = 2\lambda$, it follows from Proposition \ref{p.intersection_boundary} that $\ell^{v_o,v_i}$ touches the inner boundary and thus the extremal distance is equal to zero. Notice that indeed in this case $\widehat{\tau}^{v_o,v_i}_{-a,b}$ is also zero so the proposition holds. So we can now assume that $a + b \geq 4\lambda$.

Let now $v_o\in [0,2\lambda)$ and recall the sequence $\ell_0^{u_{v_o,v_i}}, \ell_1^{u_{v_o,v_i}},.., \ell_n^{u_{v_o,v_i}},\ell_{n+1}^{u_{v_o,v_i}}$ introduced in the proof of Proposition \ref{mlawv}, where $\ell_0^{u_{v_o,v_i}}$ is the non-contractible loop associated to $\A_{0,2\lambda}^{u_{v_o,v_i},\partial_o}$, and \[\ell_1^{u_{v_o,v_i}},...,\ell_1^{u_{v_o,v_i}},.., \ell_n^{u_{v_o,v_i}},\ell_{n+1}^{u_{v_o,v_i}}\]
 are the non-contractible loops associated to the iterations of $\A_{-2\lambda,2\lambda}^{\partial_o}$ of $\Phi^{A_0^{u_{v_o,v_i}}}$ restricted to $\An_0^{u_{v_o,v_i}}$ as in Proposition \ref{rev:bg}. 
\begin{lemma}\label{l.find_ell}
	Take $k,l\in \N$ and $v_o \in[0,2\lambda)$, then the non-trivial loop $\ell=\ell_{-2k\lambda,2l\lambda}^{u_{v_o,v_i}}$ is equal to $\ell_{\mathbf {j}}^{u_{v_o,v_i}}$ where $\mathbf{j}$ is the first $j$ so that $\alpha_j^{v_o,v_i}$ is equal to either $-2k\lambda$ or $2l\lambda$. If such a $j$ does not exist, then take $\mathbf{j}=n+1$ i.e. $\ell=\partial_i$ . Recall that here $\alpha_j^{u_{v_o,v_i}}$ is the boundary value in $\ell_j^{u_{v_o,v_i}}$ of $h_{A_{j}^{v_o,v_i}}+u_{v_0,v_i}$ restricted to $\An_j^{v_o,v_i}$.
\end{lemma}

\begin{proof}
Observe that starting from $A_{\mathbf j}^{u_{v_o,v_i}}$ we can construct $\A_{-2k\lambda,2j\lambda}^{u_{v_o,v_i}}$ by sampling $\A_{-2k\lambda,2j\lambda}^{u'}$ of $\Phi^{A_{\mathbf{j}}^{u_{v_o,v_i}}}$ inside each connected component $O$ of $D\backslash A_{\mathbf{j}}^{u_{v_o,v_i}}$, where $u'$ restricted to $O$ is equal to $h_{A_{\mathbf{j}}^{u_{v_o,v_i}}}+u_{v_o,v_i}$. The resulting set satisfies the characterisation properties of the TVS $\A_{-a,b}^{u_{v_o,v_i}}$ and thus by Theorem \ref{t.tvs} it is equal to it. Furthermore, if $\mathbf{j} \leq n$, as $\alpha_{\mathbf{j}}^{u_{v_o,v_i}} \in \{-2k\lambda, 2l\lambda\}$, we  have that  $\A_{-2k\lambda,2l\lambda}^{u'}$ of $\Phi^{A_{\mathbf{j}}^{u_{v_o,v_i}}}$ inside $\An_{-2k\lambda,2l\lambda}^{u_{v_o,v_i}}$ does not intersect $\ell_{\mathbf j}^{u_{v_o,v_i}}$. By the same construction, if $\mathbf{j}=n+1$, it is clear that $\A_{-2k\lambda,2l\lambda}^{u_{v_o,v_i}}$ connects $\partial_o$ with $\partial_i$.
\end{proof}

Proposition \ref{prop:generalTVS} now follows.
\begin{proof}[Proof of Proposition \ref{prop:generalTVS}]
	First, we reduce the proposition to the case where $v_0\in[0,2\lambda)$ and $a=-2k\lambda$, $b=2j\lambda$. To do this, we notice the following
	\begin{itemize}
		\item For $a, b$ with $a+b = n2\lambda$ with $n \geq 2$, we can always find some $c \in \R$ such that $\A_{-a,b}^{u}=\A_{-2k\lambda,2j\lambda}^{u-c}$. Thus it suffices to treat the case of $a, b \in 2k \Z$. 
		\item By symmetry we can assume $v_o$ is positive, and again by a constant shift taking values in $2\lambda \Z$, we can assume $v_o \in [0,2\lambda)$.
	\end{itemize}
	
	The proof now follows as in Proposition \ref{mlawv}, by separating in the same cases as in that proposition and noting what these cases correspond for the respective Brownian motion. The only difference is that thanks to Lemma \ref{l.find_ell}, this time we are looking at the 
$\mathbf j$-th iteration instead of the first one.
\end{proof}

\subsubsection{The case of first passage sets}
Finally, let us consider the case of the first passage sets. On an annulus $\An$ with outer and inner boundaries denoted by $\partial_o$ and $\partial_i$, consider $\Phi$ a zero-boundary GFF. For $v \in \R$, again we let $u_v$ denote the bounded harmonic function that is equal to $0$ on $\partial_o$ and equal to a constant $v$ on $\partial_i$. Consider now, $\A^{u_v,\partial_o}_{-2\lambda}$, the component of the first passage set $\A^{u_v}_{-a}$ of $\Phi + u_v$ connected to $\partial_o$ and let $ \ell^{u_v}_{-a}$ be the outer boundary of the non-simple connected component of $\An\backslash\A^{u_v,\partial_o}_{-a}$  (in case there is no non-trivial loop surrounding $\partial_i$ we again set $\ell^{u_v}_{-a} = \partial_i$). We want to calculate $\ED(\ell^{u_v}_{-a}, \partial_o)$ as before. 

To do this, let $\widehat B^v$ be a Brownian bridge from $0$ to $v$, whose time-duration $L$ is given by the extremal distance of $\An$. Define
\begin{equation*}
\widehat T^v_{-a}:=\inf\{s\geq 0: \widehat B^v_s=-a\} \wedge L.
\end{equation*}
If $\widehat{T}^{v}_{-a}=L$, we also set 
$\widehat{\tau}^{v}_{-a} = L$. If, however 
$\widehat{T}^{v}_{-a}<L$, we define
\begin{equation*}
\widehat{\tau}^{v}_{-a}:= \sup\{0\leq s \leq \widehat{T}^{v}_{-a}:\widehat B^v_s=-a+2\lambda\}\vee 0.
\end{equation*}
The analogue of Proposition \ref{mlaw0} to first-passage sets is then the following.
\begin{prop}\label{mlawfps}
	For all $a > 0$, the law of $\ED(\ell^{u_v}_{-a}, \partial_o)$ equals that of $\widehat{\tau}^{v}_{-a}$.
\end{prop}
\begin{proof}
Consider a TVS $\A_{-a,-a+2k\lambda}^{u_v}$. Then, on the event where the label of $\ell^{u_v}_{-a,-a+2k\lambda}$ towards $\partial_i$ is equals $-a$, we have that $\ell^{u_v}_{-a}$ is equal to $\ell^{u_v}_{-a,-a+2k\lambda}$ . Moreover, from Proposition \ref{p.LawELBddTVS} we know that the probability of this event converges to $1$ as $k \to \infty$. Thus the proposition follows from Proposition \ref{prop:generalTVS} by taking 
$k \to \infty$.
\end{proof}

\subsection{The conditioned GFF and annulus CLE}
\label{SubSec cond GFF}

Consider $\A_{-2\lambda, 2\lambda}^{u_v,\partial_o}$ in an annulus $\An$ of a zero boundary GFF $\Phi$. As $\A_{-2\lambda, 2\lambda}^{u_v,\partial_o}$ is a local set, conditionally on the set $\A_{-2\lambda,2\lambda}^{u_v,\partial_o}$ and the harmonic function $h_{\A_{-2\lambda,2\lambda}^{u_v,\partial_o}}$, we can explicitly describe the law of the GFF $\Phi$ restricted to a connected component $O$ of $D\backslash \A_{-2\lambda,2\lambda}^{u_v,\partial_o}$ as a GFF with boundary conditions. 

In the following proposition we extend this in the case where $O$ is the outer boundary of the only annular connected component of  $\An_1 = \An\backslash \A_{-2\lambda,2\lambda}^{u_v, \partial_o}$. First, a simple extension is to determine the conditional law of the GFF in $O$ when conditioning only on the non-contractible loop $\ell^{u_v}$ on the boundary of $\An_1$; second, we also determine the conditional law in the interior of $O^c = \An \backslash \An_1$.

\begin{prop}
	\label{Prop cond GFF}
	On the event
	$\A^{u_{v},\partial_o}_{-2\lambda,2\lambda}\cap\partial_i=\emptyset$, let $\ell^{u_v}$ denote the non-contractible loop of $\A^{u_{v},\partial_o}_{-2\lambda,2\lambda}\cap\partial_i=\emptyset$ and $\alpha^{u_v}$ its label towards $\partial_i$. Let $\An_1^{u_v}$ denote the annulus between $\ell^{u_v}$ and $\partial_i$, and let $\overleftarrow \An_n^{u_v}$ denote the annulus between $\ell^{u_v}$ and $\partial_o$. Conditionally on $(\ell^{u_v},\alpha^{u_v})$ and the event 
	$\A^{u_{v},\partial_o}\cap\partial_i=\emptyset$, the fields
	$\Phi\vert_{ \An_{1}^{u_v}}$ and
	$\Phi\vert_{ \overleftarrow \An_n^{u_v}}$ are independent. Moreover,
	\begin{enumerate}
		\item The conditional distribution of 
	$(\Phi+u_{v})\vert_{ \An_{1}}$ is that of $\Phi'+u'$ where $\Phi'$ and $u'$ are independent, $\Phi'$ is a GFF in $\An_1$ and $u'$ is the bounded harmonic function with boundary values
	$\alpha^{u_v}$ on $\ell^{u_v}$ and $v$ on $\partial_i$; 
	\item The conditional distribution of the field
	$((\Phi+u_{v})\vert _{\overleftarrow \An_n^{u_v}})$ is that of a GFF
	in $\overleftarrow \An_n^{u_v}$ conditioned on the event that its TVS $\A_{-2\lambda,2\lambda}$ is connected.
	\end{enumerate}
\end{prop}

\begin{proof}
Let us start with (1). This is just a slight extension of the local set property. Indeed, if we condition on $\A_{-2\lambda,2\lambda}^{u_v, \partial_o}$ and $h_{\A_{-2\lambda,2\lambda}^{{u_v},\partial_o}}$, then the law of $(\Phi+u_{v})\vert_{ \An_{1}}$ only depends on $\ell^{{u_v}}$ and $\alpha^{u_v}$ and it is exactly the one described in (1), so we conclude.

The conditional independence of the fields
$\Phi\vert_{ \An_{1}^{u_v}}$ and
$\Phi\vert_{ \overleftarrow \An_n^{u_v}}$ follows similarly. Indeed, if we condition on $(\Phi+u_{v})\vert_{ \overleftarrow \An_n^{u_v}}$, $\A_{-2\lambda,2\lambda}^{u_v,\partial_o}$ and $h_{\A_{-2\lambda,2\lambda}^{u_v, \partial_o}}$ the law of $(\Phi+u_{v})\vert_{ \An_{1}^{u_v}}$ still only depends on $\ell^{u_v}$ and $\alpha^{u_v}$, so conditionally on $\ell^{u_v}$ and $\alpha^{u_v}$, it is independent of $(\Phi+u_{v})\vert_{ \overleftarrow \An_n^{u_v}})$.

We now prove (2). For simplicity, let us first consider the case when $u_v$ has boundary conditions $2N\lambda$ on the outer boundary. Using the notation of Proposition \ref{p.comm1} Let $n$ be the smallest number $k$ such that $A_{k+1}^N$ (the $k-$ times iterated $\A_{-2\lambda, 2\lambda}^{\partial_o}$) is connected to $\partial_i$. By Proposition \ref{p.comm1}, we see that $A_{n+1}^N = \overleftarrow A_{n+1}^N$, i.e. the $n$-times iterated $\A_{-2\lambda, 2\lambda}^{\partial_i}$ from the inner boundary. In particular, we know that $\ell^{u_v} = \ell_1^N = \overleftarrow \ell_{n}^N$. But now, for any fixed $m$, the iterated TVS, $\overleftarrow A_{m}^N$ is a local set and thus the field between $\partial_o$ and $\overleftarrow \ell_m^N$ is given by a GFF with relevant boundary conditions. Further, conditioning on the event  that $m = n$, amounts to conditioning that $\overleftarrow A_{m+1}^N \backslash \overleftarrow A_{m}^N$ is connected in the annulus between $\overleftarrow \ell_m^N$ and $\partial_o$. Moreover, by Lemma \ref{lembasic} the loop $\overleftarrow \ell_n^N$ corresponds to the last loop $\overleftarrow \ell$ discovered and its label $\overleftarrow \alpha_0^N$ equals $0$. The claim now follows.

The case of general boundary conditions on the inner boundary follows similarly by using Proposition \ref{rev:bg} instead of Proposition \ref{p.comm1}.

\end{proof}

\begin{rem}
In fact, it is also true that conditionally on only $\ell^{u_v}$, the law of $\Phi\vert_{ A_1^{u_v}}$ is independent of that of $\Phi\vert_{ \overleftarrow A_m^{u_v}}$. This is due to the fact that one that the law described in (2) does not depend on $\alpha^{u_v}$.
\end{rem}

To prove Theorem \ref{t.loop soup} we will need to generalize this proposition to several interfaces stemming from an iterated $\A_{-2\lambda, 2\lambda}$. In this respect, consider as in Subsection \ref{Ss.Extensions to non-zero} the exploration of the GFF
$\Phi+u_v$ by sampling successive local sets in the non-contractible connected component of the complementary of the preceding one, starting with 
$\A_{-2\lambda, 2\lambda}^{u_{v},\partial_o}$.
Let $\ell_{0}^{u_{v}}, \ell_{1}^{u_{v}},
 \dots, 
\ell_{n+1}^{u_{v}}$ be the sequence of non-contractible loops, with $\ell_{0}^{u_{v}}=\partial_o$ and $\ell_{n+1}^{u_v}=\partial_i$. Let
$\alpha_{j}^{u_{v}}\in 2\lambda\mathbb{Z}$ be the label of the loop 
$\ell_{j}^{u_{v}}$ towards $\partial_i$.
By construction, $\alpha_{0}^{u_{v}}=0$,
$\alpha_{1}^{u_{v}}\in\{-2\lambda,2\lambda\}$ and
$\vert \alpha_{j+1}^{u_{v}} - \alpha_{j}^{u_{v}}\vert = 2\lambda$ 
for every $j\in\{1,2,\dots,n-1\}$. Further, $\mathbf{A}^{j,u_v}$ will denote the annular domain delimited by
$\ell_j^{u_{v}}$ and $\partial_i$ and
$ \An_{j-1,j}^{u_{v}}$ will denote the annular domain delimited by
$\ell_j^{u_{v}}$ and $\ell_{j-1}^{u_{v}}$. The generalized version that follows by standard arguments is then as follows.

\begin{cor}
	\label{Cor cond GFF several}
	Let $J$ be a stopping time
	(with values in $\mathbb{N}$) for the filtration of
	$(\ell_{j}^{u_v},\alpha_{j}^{u_v})_{j\geq 1}$. Conditionally on
	$\ell_{J}^{u_v}\cap\partial_i=\emptyset$ and on
	$(\ell_{j}^{u_v},\alpha_{j}^{u_v})_{1\leq j\leq J}$, the fields
	$(\Phi\vert_{ \An_{j-1,j}^{u_{v}}})_{1\leq j\leq J}$ are independent.
	The conditional law of
	$(\Phi+u_{v}-\alpha_{j-1}^{u_v})
	\vert_{ \An_{j-1,j}^{u_{v}}}
	$ is that of the GFF conditioned on the event that
	its TVS $\A_{-2\lambda,2\lambda}$
	is connected.
\end{cor}

To finish this section, let us state a corollary that is not used in this paper, but is of independent interest.
\begin{cor}
Consider a zero boundary GFF $\Phi$ in $\D$ and $\A_{-2\lambda, 2\lambda}$. Then conditioned on the loop $\ell$ of $\A_{-2\lambda, 2\lambda}$ surrounding the origin, the restriction of $\Phi$ to the annulus between $\ell$ and $\partial_o$ has the law of a zero boundary GFF conditioned on the even that its corresponding two-valued set $\A_{-2\lambda, 2\lambda}$ is connected. 
\end{cor}

\begin{proof}
Proposition \ref{prop:coupling} implies that we can couple the local set $\A_{-2\lambda, 2\lambda}^{\partial_o}$ for the GFF in the annulus $\D \backslash r\D$ and the TVS $\A_{-2\lambda, 2\lambda}$ for the GFF in $\D$ such that as $r\to 0$ the probability that $\A_{-2\lambda, 2\lambda}^{\partial_o}$ is equal to $\A_{-2\lambda, 2\lambda}$ converges to $1$. Thus the corollary follows from Proposition \ref{Prop cond GFF}.
\end{proof}

Thus, as CLE$_4$ has the law of $\A_{-2\lambda, 2\lambda}$, we can use Prop 3.5 in \cite{HSW} to deduce that the annulus CLE defined in that paper (Section 3.1) agrees with the conditional law of $\A_{-2\lambda, 2\lambda}$, conditioned on the event that it is connected. Using the definition of the annulus CLE via BTLS, we further deduce the following connection between the critical Brownian loop soup and the TVS $\A_{-2\lambda, 2\lambda}$ in the annulus.

\begin{cor}
Take a critical Brownian loop soup in an annulus $\An$ and let $E$ be the event on which this loop soup has no cluster that disconnects from $\partial_i$ from $\partial_o$. Then conditionally on $E$, the law of the outermost boundaries of its outermost clusters is the same as that of a TVS $\A_{-2\lambda, 2\lambda}$ in $\An$, conditioned to be connected.
\end{cor}

\section{The joint laws}
\label{SecLaws}
In this section we will calculate the joint laws of extremal distances and conformal radii. This is done in two steps.

\begin{itemize}
	\item First, we obtain a characterization of certain joint laws for the Brownian motion and the Brownian bridge in terms of conditional laws. These characterizations could be of independent interest, and are based purely on Brownian bridge arguments.
	\item Second, we argue that the joint laws of extremal distances and conformal radii satisfy these characterizations. 
\end{itemize}

This is also the rough outline of this chapter, although we will start instead with more general statements of our main theorems, and finish with the proof of Theorem 1.2 in a separate subsection.

\subsection{Statements}
\label{SubsecLawsStat}

To describe the relevant laws of conformal radii and extremal distances, we need to introduce the following random times of a Brownian motion $B_t$. 
\begin{align*}
&T_{-a,b}:=\inf\{t\geq 0: B_t=-a \text{ or } B_t=b\},\\
&\tau_{-a,b}:=\sup\{0\leq t \leq T_{-a,b}:|B_t-B_{T_{-a,b}}|=2\lambda\}\vee 0,
\end{align*} 

Theorem \ref{t. main} is a consequence of the following generalization.
\begin{thm}\label{t. joint law TVS}
	Let $a,b>0$ with $a+b = 2k\lambda$, 
$k\in\mathbb{N}\backslash\{0\}$, and $\ell$ be the loop of a two-valued set of level $-a$ and $b$, $\A_{-a,b}$, in the unit disk surrounding the origin. Consider the extremal distance $\ED(\ell, \partial \D)$ between $\ell$ and $\partial \D$, and 
$\crad(0,\D\backslash \ell)$ the conformal radius of the origin on the bounded domain surrounded by $\ell$. Then, 
$(\ED(\ell,\partial \D), -(2\pi)^{-1}\log\crad(0,\D\backslash \ell) )$ is equal in law to $(\tau_{-a,b},T_{-a,b})$.
	\end{thm}

	The fact that we can obtain the joint laws for a large family of two-valued sets motivates the study of another important set related to the continuum GFF - the first passage set of the GFF introduced in \cite{ALS1}. The geometry of the loop of this set is related to the following times of a Brownian motion $B$,
	\begin{align*}
	&T_{-a}:= \inf\{t\geq 0: B_t=-a\},\\
	&\tau_{-a}:=\sup\{0 \leq t \leq T_{-a}: B_t= -a+2\lambda\} \vee 0.
	\end{align*}
	In this context, an analogue of Theorem \ref{t. main} takes the following form.
	\begin{thm}\label{t. law FPS}
		Let $a>0$ and $\ell$ be the loop of the first passage set of level $-a$, $\A_{-a}$, in the unit disk surrounding the origin. Take 
$\ED(\ell, \partial \D)$ the extremal distance between $\ell$ and 
$\partial \D$, and $\crad(0,\D\backslash \ell)$ the conformal radius of the origin on the bounded domain surrounded by $\ell$. Then, 
$(\ED(\ell,\partial \D), -(2\pi)^{-1}\log\crad(0,\D\backslash \ell) )$ is equal in law to $(\tau_{-a},T_{-a})$.
		\end{thm}

This theorem is also of interest due to the relation of the Brownian loop soup and FPS, as given in Proposition \ref{PropClusterFPS}, and is used to prove Theorem 1.2, i.e. to describe the extremal distances related to a cluster of a critical Brownian loop soup.

In fact, strictly speaking, one could deduce Theorem 5.2 from Theorem 5.1 by using the fact that the non-contractible loop of $\A_{-a}$ agrees with the non-contractible loop of $\A_{-a,b}$ with probability $1-o(1)$ as $b \to \infty$ (see e.g. the proof of Proposition \ref{mlawfps}). However, several proofs are easier to explain in the case of the FPS, so we will still prove both theorems in this section.

\subsection{Characterizations of the joint laws}
\label{SubsecLawsCharac}
For $L>0$ and $v\in \R$, denote $(\widehat{B}^{v}_{t})_{0\leq t\leq L}$ the standard Brownian bridge from $0$ to $v$ in time $L$. Take 
$c\in \R$ and define
\begin{align*}
&\widehat{T}^{v}_{c}:=
\inf\lbrace 0\leq t\leq L: \widehat{B}^{v}_{t}
=c\rbrace \wedge L,\\
&\widehat{\tau}^{v}_{c}:=
\left\lbrace
\begin{array}{ll}
\sup\lbrace 0 \leq \widehat{T}^{v}_{c} :
\widehat{B}^{v}_{t}=c+2\lambda \text{ or } \widehat{B}^{v}_{t}=c-2\lambda \rbrace \vee 0 & 
\text{if } \widehat{T}^{v}_{c}<L, \\ 
L & \text{if } \widehat{T}^{v}_{c}=L.
\end{array} 
\right.
\end{align*}
Additionally, for $a,b>0$, denote 
$\widehat{T}^{v}_{-a,b}=\widehat{T}^{v}_{-a}\wedge T^{v}_ b$, and
\begin{displaymath}
\widehat{\tau}^{v}_{-a,b}=\left\lbrace\begin{array}{ll}
\widehat{\tau}^{v}_{ -a} & 
\text{if } \widehat{T}^{v}_{-a,b}=\widehat{T}^{v}_{-a} <L, \\ 
\widehat{\tau}^{v}_{ b} & \text{if } \widehat{T}^{v}_{-a,b} 
=\widehat{T}^{v}_{ b} <L, \\
L & \text{if } \widehat{T}^{v}_{-a,b}=L.
\end{array} \right.
\end{displaymath}

The following proposition gives a characterization of the joint law of 
$(\widehat{\tau}^{v}_{-a}, L-\widehat{T}^{v}_{-a})$. 
This will help us identify,
on an annulus,
the joint law of the extremal distance of the non-contractible loop $\ell^{u_v}_{-a}$ of 
$\A_{-a}^{u_{v},\partial_o}$ 
to both boundaries.
\begin{prop}\label{p. joint law FPS}
	Let $a>0$. Assume that for every $v\in\R$, there is a couple of random variables 
	$(\sigma_{o}^{v},\sigma_{i}^{v})$, with 
	$\sigma_{o}^{v},\sigma_{i}^{v}\in [0,L]$ and
	$\sigma_{o}^{v}+\sigma_{i}^{v}\leq L$, such that
	\begin{enumerate}
		\item $\sigma_{o}^{v}$, respectively $\sigma_{i}^{v}$, has same law as
		$\widehat{\tau}^{v}_{-a}$, respectively 
		$L-\widehat{T}^{v}_{ -a}$,
		\item for every $v\in \R$ and on the event 
		$\sigma_{i}^{v}> 0$, the conditional law of 
		$\sigma_{o}^{v}$ given $\sigma_{i}^{v}$ is the same as the conditional law of
		$\sigma_{o}^{0}$ given $\sigma_{i}^{0}$.
	\end{enumerate}
	Then, for every $v\in\R$, the couple
	$(\sigma_{o}^{v},\sigma_{i}^{v})$ has the same law as
	$(\widehat{\tau}^{v}_{ -a},L-\widehat{T}^{v}_{ -a})$.
\end{prop}

Similarly, to identify the joint laws of extremal distances for two-valued sets, we obtain the following characterization of the joint law of 	
$(\widehat{\tau}^{v}_{-a,b},L-\widehat{T}^{v}_{-a,b},
\widehat{B}^{v}_{\widehat{T}^{v}_{-a,b}}
\1_{\widehat{T}^{v}_{-a,b}<L})$.

\begin{prop}\label{p. joint law TVS}
	Let $a,b>0$. Assume that for every $v\in\R$, there is a triple of random variables 
	$(\sigma_{o}^{v},\sigma_{i}^{v},X^{v})$, with 
	$\sigma_{o}^{v},\sigma_{i}^{v}\in [0,L]$,
	$\sigma_{o}^{v}+\sigma_{i}^{v}\leq L$, and $X^{v}\in\lbrace -a,b,v\rbrace$, such that
	\begin{enumerate}
		\item the joint law of
		$(\sigma_{i}^{v},X^{v})$ is the same as for
		$(L-\widehat{T}^{v}_{-a,b},
		\widehat{B}^{v}_{\widehat{T}^{v}_{-a,b}})$,
		\item the law of $(\sigma_{o}^{v},X^{v})$ is the same as that of
		$(\widehat{\tau}^{v}_{-a,b},
		\widehat{B}^{v}_{\widehat{T}^{v}_{-a,b}})$,
		\item for every $v\in \R$, on the event $\sigma_{i}^{v}>0$, the conditional law of 
		$\sigma_{o}^{v}$ given $(\sigma_{i}^{v},X^{v})$ is the same as the conditional law of
		$\sigma_{o}^{0}$ given $(\sigma_{i}^{0},X^{0})$.
	\end{enumerate}
	Then, for every $v\in\R$ the triple
	$(\sigma_{o}^{v},\sigma_{i}^{v},X^{v}
	\1_{\sigma_{i}^{v}>0})$ has same law as
	$(\widehat{\tau}^{v}_{-a,b},L-\widehat{T}^{v}_{-a,b},
	\widehat{B}^{v}_{\widehat{T}^{v}_{-a,b}}
	\1_{\widehat{T}^{v}_{-a,b}<L})$.
\end{prop}

The proofs of these propositions make use of certain explicit formulas. Firstly, recall the heat kernel 
on $\R$:
\begin{align}
\label{HeatKernel}
p(t,x,y):=\dfrac{1}{\sqrt{2\pi t}}e^{-\frac{(y-x)^{2}}{2t}}.
\end{align}
Also, denote by $q_{-a}$ the density function of the first hitting times of Brownian motion of level $-a$ (see \cite{BorodinSalminen2015}, Section 1.2, Formula 2.0.2)
\begin{align}
\label{Eqqa}
q_{-a}(t):=\dfrac{a}{\sqrt{2\pi t^{3}}}e^{-\frac{a^{2}}{2t}}
=\mathbb{P}(T_{-a}\in (t,t+dt))/dt. 
\end{align}
Finally, let us also state the explicit laws for the hitting times of the Brownian bridge.
\begin{lemma}
\label{l. law T-a}
	For all $v\in\R$ and $a>0$, the law of $\widehat{T}^{v}_{-a}$ is given by
	\begin{align}
	\label{EqT-aBridge}
	\mathbb{P}(\widehat{T}^{v}_{ -a}
	\in (t,t+dt), \widehat{T}^{v}_{ -a}<L)=q_{-a}(t)
	\dfrac{p(L-t,-a,v)}{p(L,0,v)}\1_{t<L}dt.
	\end{align}
\end{lemma}

\begin{proof}
For $t\in [0,L)$, the Brownian bridge
$(\widehat{B}_{s}^{v})_{0\leq s\leq t}$ is absolutely continuous with respect to the Brownian motion
$(B_{s})_{0\leq s\leq t}$, with the Radon-Nikodym derivative
\begin{align*}
\mathcal{D}_{t}=\dfrac{p(L-t,B_{t},v)}{p(L,0,v)}.
\end{align*}
Then \eqref{EqT-aBridge} follows from \eqref{Eqqa} and the expression above.
\end{proof}

We are now ready to prove Proposition \ref{p. joint law FPS}.
\begin{proof}[Proof of Proposition \ref{p. joint law FPS}]
	Denote by $F(\nu,v)$ the Laplace transform of $\sigma_o^v$:
	\begin{displaymath}
	F(\nu,v)=\E[e^{-\nu\sigma_{o}^{v}}],~~\nu\geq 0.
	\end{displaymath}
	By assumption (1), this is also the Laplace transform of 
	$\widehat{\tau}^{v}_{-a}$, hence already determined.
	Take $\rho>0$ and  let $F_{c}(\nu,\rho)$ be the conditional Laplace transform of $\sigma_o^v$ given $\sigma_i^v$
	\begin{displaymath}
	F_{c}(\nu, \rho)
	=\E[e^{-\nu\sigma_{o}^{v}}\vert \sigma_{i}^{v}=\rho],
	\end{displaymath}
	defined, for every $\nu\geq 0$,
	$\1_{(0,L)} d\rho$-almost everywhere. Note that as a consequence of assumption (3), $F_c(\nu, \rho)$ does not depend on $v$.
	
	It suffices to show that $F_c(\nu,\rho)$ equals 	
	$\E[e^{-\nu \widehat{\tau}^{v}_{-a}}\vert 
	\widehat{T}^{v}_{-a}=\rho]$ for any $\nu \geq 0$ and for $\d\rho$-almost every $\rho \in [0,L)$. To do this, we argue that due to our conditions, $F_c(\nu, \cdot)$ is entirely determined by $F(\nu,\cdot)$. 
Indeed, note that 	for $v\leq -a$, we have that 
$\widehat{T}^{v}_{ -a}<L$ a.s., and thus
	\begin{eqnarray*}
		F(\nu,v)&=&\E[F_{c}(\nu,\sigma_{i}^{v})]\\
		&=&
		\dfrac{1}{p(L,0,v)}\int_{0}^{L}F_{c}(\nu, \rho)
		q_{-a}(L-\rho)
		\dfrac{1}{\sqrt{2\pi \rho}} 
		e^{-\frac{(v+a)^{2}}{2\rho}}d\rho\\
		&=&
		\dfrac{1}{p(L,0,v)}\int_{L^{-1}}^{+\infty}F_{c}(\nu, s^{-1})
		q_{ -a}(L-s^{-1})
		\dfrac{1}{\sqrt{2\pi s^{3}}} e^{-\frac{(v+a)^{2}}{2}s}ds.
	\end{eqnarray*}
	As a consequence, we obtain that 
	$F(\nu,-\sqrt{2u}-a)p(L,0,-\sqrt{2u}-a)$ 
	is the Laplace transform of
	\begin{displaymath}
	s\mapsto \1_{s>L^{-1}}
	F_{c}(\nu, s^{-1})
	q_{-a}(L-s^{-1})
	\dfrac{1}{\sqrt{2\pi s^{3}}},
	\end{displaymath}
	evaluated in $u\geq 0$. It follows that the above function is uniquely determined $ds$-almost everywhere, 
	and consequently, for a fixed $\nu$,
	$F_{c}(\nu, \rho)$ is uniquely determined for 
	$d\rho$-almost every
	$\rho\in (0,L)$, and thus equals
	\begin{align*}
	\E[e^{-\nu \widehat{\tau}^{v}_{-a}}\vert 
	\widehat{T}^{v}_{-a}=\rho].
	\end{align*}
By taking a countable intersection, 
	$F_{c}(\nu, \rho)$  is uniquely determined for 
	$d\rho$-almost every
	$\rho\in (0,L)$ and for all $\nu$ positive rational. Thus by continuity of the Laplace transform, for $d\rho$-almost every $\rho\in (0,L)$, $F_{c}(\nu,\rho)$
	equals 	$\E[e^{-\nu \widehat{\tau}^{v}_{-a}}\vert 
	\widehat{T}^{v}_{-a}=\rho]$ for all $\nu \geq 0$.
	
	It remains to treat the $\rho = 0$ case, which can happen with positive probability. To do this, note that when $v>-a$,
	\begin{eqnarray*}
		\label{EqLT1}
		F(\nu,v)&=&\E[F_{c}(\nu,\sigma_i^v)\1_{\sigma_i^v > 0}]+\E\left[e^{-\nu \sigma_{o}^v}\mid \sigma_i^v=0 \right] \P(\sigma_i^v=0) 
		\nonumber\\
		&=&
		\int_{0}^{L}F_{c}(\nu, \rho)
		q_{-a}(L-\rho)
		\dfrac{p(\rho,-a,v)}{p(L,0,v)} d\rho+
		\E\left[e^{-\nu \sigma_{o}^v}\mid \sigma_i^v=0 \right]\mathbb{P}(T^{v}_{ -a}=L).
	\end{eqnarray*}
	Thus, as $\P(T^{v}_{ -a}=L)$ is positive, $\E\left[e^{-\nu \sigma_{o}^v}\mid \sigma_i^v=0 \right]$ is uniquely determined, too, and the proposition follows.
\end{proof}

Next we move to the proof of Proposition \ref{p. joint law TVS}. The proof is analogous to that of Proposition \ref{p. joint law FPS}. 
Let us define the following density
\begin{align*}\label{Eqqab}
q_{-a,b}(t,c):=
\mathbb{P}(T_{-a,b}\in(t,t+dt), B_{T_{-a,b}}=c)/dt,
~~c\in\{ -a,b\}.
\end{align*}
For the exact expression of $q_{-a,b}(t,c)$
we refer to Proposition \ref{Prop A1 laws} (2), which stems from Formula 3.0.6, Section 1.3, in \cite{BorodinSalminen2015}.

We can, now, write the analogue of Lemma \ref{l. law T-a} in this context.
\begin{lemma}
\label{l. law T-a b}
	The joint law of
	$(\widehat{T}^{v}_{-a,b},\widehat{B}^{v}_{\widehat{T}^{v}_{-a,b}})$ on the event
	$\widehat{T}^{v}_{-a,b}<L$ is given by
	\begin{eqnarray*}
	\mathbb{P}\left (\widehat{T}^{v}_{-a,b}
	\in (t,t+dt), \widehat{T}^{v}_{-a,b}<L,
	\widehat{B}^{v}_{\widehat{T}^{v}_{-a,b}}=-a\right )&=&
	q_{-a,b}(t,-a)
	\dfrac{p(L-t,-a,v)}{p(L,0,v)}\1_{t<L} dt,\\
	\mathbb{P}\left (\widehat{T}^{v}_{-a,b}
	\in (t,t+dt), \widehat{T}^{v}_{-a,b}<L,
	\widehat{B}^{v}_{\widehat{T}^{v}_{-a,b}}=b\right )&=&
	q_{-a,b}(t,b)
	\dfrac{p(L-t,b,v)}{p(L,0,v)}\1_{t<L} dt.
	\end{eqnarray*}
\end{lemma}

The proof of Proposition \ref{p. joint law TVS} now follows almost line by line the proof of Proposition \ref{p. joint law FPS}, so instead of detailing with we just highlight one difference.

\begin{proof}[Proof of Proposition \ref{p. joint law TVS}:]
 The only difference w.r.t. to the proof of Proposition \ref{p. joint law FPS} is the fact that we calculate the Laplace transforms, conditioned on the value of $X^v$, i.e. we consider the Laplace transforms
	\begin{displaymath}
	\E[e^{-\nu \sigma_{o}^{v}}\1_{ X^{v}=-a}],
	\qquad
	\E[e^{-\nu \sigma_{o}^{v}}\1_{X^{v}=b}],
	\end{displaymath}
	and their conditional versions given 
	$\sigma_{i}^{v}=\rho>0$ and $X^v$ equal either $-a$ or $b$. Additionally, one needs to note that $X^v=v$ if and only if $\sigma_i^{v}=0$.
\end{proof}

\subsection{Calculating the joint laws for TVS and FPS}
\label{SubsecLawsFPSTVS}

We will now verify that the laws of extremal distances for the FPS and the relevant TVS in an annulus satisfy the characterization of the joint law. This will prove versions of Theorems \ref{t. joint law TVS}
and \ref{t. law FPS} in an annulus. We thereafter explain how to deduce the case of the disk, i.e. the theorems themselves. 

Now, let $\An_r := \D\backslash r\overline{\D}$ ($r\in(0,1)$) be an annulus with outer boundary 
$\partial_o:=\partial\D$ and inner boundary
$\partial_i:=r\partial\D$. Let $L$ be the extremal distance
between $\partial_o $ and $\partial_i$:
\begin{align*}
L:=\ED(\partial_o,\partial_i) = 
\dfrac{1}{2\pi}\log(r^{-1}).
\end{align*}

Define, for $v\in\R$, $u_v$ as the harmonic function 
in $\An_r$ with boundary values $0$ on 
$\partial_o$ and $v$ on $\partial_i$. 
Take $\Phi$ a (zero-boundary) GFF in $\An_r$ and let 
$\A_{-a,b}^{u_v}$ ($a,b > 0, a+b\geq 2\lambda$)
be a two-valued local set of $\Phi + u_v$. Denote by $\A_{-a,b}^{u_v,\partial_o}$ the connected component of
$\A_{-a,b}^{u_v}$ containing $\partial_o$. Let 
$\ell^{u_v}_{-a,b}$ be the non-contractible loop of 
$\A_{-a,b}^{u_v, \partial_o}$
on the event 
$\A_{-a,b}^{u_v, \partial_o}\cap\partial_{i}=\emptyset$. 
Moreover, $\alpha^{u_v}_{-a,b}\in\{-a,b\}$ denotes the (random, constant) boundary value of $\Phi + u_v$ on $\ell^{u_v}_{-a,b}$ from the interior side.
On the event
$\A_{-a,b}^{u_v, \partial_o}\cap\partial_{i}\neq\emptyset$,
the convention is
$\ell^{u_v}_{-a,b}=\partial_{i}$.

We can now obtain the analogue of Theorem \ref{t. joint law TVS}  in the case of the annulus, i.e. identify the joint laws of the extremal distances of loop of a TVS to both inner and outer boundaries.

\begin{prop}
\label{p. law of TVS}
Take $a,b>0$, with $a+b=2k\lambda$, 
$k\in\mathbb{N}\backslash\{0\}$.
Let $(\widehat{B}^{v}_{t})_{0\leq t\leq L}$ be a standard Brownian bridge from $0$ to $v$ in time $L$. With the notations of Subsection \ref{SubsecLawsCharac}, the joint distribution of
\begin{align*}
(\ED(\ell^{u_v}_{-a,b},\partial_{o}),
\ED(\ell^{u_v}_{-a,b},\partial_{i}),
\alpha^{u_v}_{-a,b}\1_{\A_{-a,b}^{u_v, \partial_o}\cap\partial_{i}=\emptyset})
\end{align*}
is that of
$(\widehat{\tau}^{v}_{-a,b},L-\widehat{T}^{v}_{-a,b},
\widehat{B}^{v}_{\widehat{T}^{v}_{-a,b}}
\1_{\widehat{T}^{v}_{-a,b}<L})$.
\end{prop}
\begin{proof}
By Proposition
\ref{p.LawELBddTVS}, for all
$v\in\R$, we have that
\begin{align*}
\P(\A_{-a,b}^{u_v, \partial_o}\cap\partial_{i}=\emptyset)=
\P(\widehat{T}^{v}_{-a,b}<L),
\end{align*}
and the joint law of
$(\ED(\ell^{u_v}_{-a,b},\partial_{i}),\alpha^{u_v}_{-a,b})$ on the event
$\A_{-a,b}^{u_v, \partial_o}\cap\partial_{i}=\emptyset$ is that of
$(L-\widehat{T}^{v}_{-a,b},\widehat{B}^{v}_{\widehat{T}^{v}_{-a,b}})$
on the event $\widehat{T}^{v}_{-a,b}<L$.
Moreover, by Proposition \ref{prop:generalTVS},
for all $v\in\R$, $\ED(\ell^{u_v}_{-a,b},\partial_{o})$ has same law as
$\widehat{\tau}^{v}_{-a,b}$. Finally, from Proposition
\ref{p.change of measure TVS}, we deduce that the family
\[((\ED(\ell^{u_v}_{-a,b},\partial_{o}),
\ED(\ell^{u_v}_{-a,b},\partial_{i}),
\alpha^{u_v}_{-a,b}\1_{\A_{-a,b}^{u_v, \partial_o}\cap\partial_{i}=\emptyset}
))_{v\in\R}\] satisfies the assumptions of 
Proposition \ref{p. joint law TVS} and thus we conclude.
\end{proof}

Now, consider $\A_{-a}^{u_v,\partial_o}$ 
($a>0$)
to be the first passage set starting from $\partial_o $ 
of the GFF $\Phi$.
Let $\ell^{u_v}_{-a}$ be the non-contractible loop of 
$\A_{-a}^{u_v,\partial_o}$ 
(on the event $\A_{-a}^{u_v,\partial_o}\cap\partial_i=\emptyset$,
otherwise $\ell^{u_v}_{-a}=\partial_{i}$).We can now state the analogue of Theorem \ref{t. law FPS} in an annulus. We omit the proof that is exactly the same as that of 
Proposition \ref{p. law of TVS}.
\begin{prop} \label{p. law of FPS}
For every $a>0$, the joint distribution of
\begin{align*}
(\ED(\ell^{u_v}_{-a},\partial_{o}),
\ED(\ell^{u_v}_{-a},\partial_{i}))
\end{align*}
is that of
$(\widehat{\tau}^{v}_{-a},L-\widehat{T}^{v}_{-a})$.
\end{prop}

{Theorems \ref{t. joint law TVS} 
and \ref{t. law FPS} now follow from these propositions by letting the inner radius of the annulus tend to zero.
\begin{proof} [Proof of Theorems \ref{t. joint law TVS}
and \ref{t. law FPS}]
We only write the proof of Theorem \ref{t. joint law TVS}, as
Theorem \ref{t. law FPS} can be proven by the same technique.

Here $\Phi$, 
respectively $\Phi^r$ will denote 
the (zero-boundary) GFF
in $\D$, respectively $\An_r$. 
$\A_{-a,b}(\Phi)$ 
will denote the corresponding two-valued
local set of $\Phi$ and 
$\ell$ will denote the loop of $\A_{-a,b}(\Phi)$ surrounding $0$.
$\A_{-a,b}(\Phi^r)$ 
will denote the two-valued local set of
$\Phi^r$, and 
$\A_{-a,b}^{\partial \D}(\Phi^r)$ its connected component containing $\partial \D =\partial_o$, the outer boundary of
$\An_r$. Finally, on the event
$\A_{-a,b}^{\partial \D}(\Phi^r)\cap r\partial\D=\emptyset$,
$\ell(r)=\ell^{0}_{-a,b}(\Phi^r)$ will denote the non-contractible (in $\overline{\An_r}$) loop of 
$\A_{-a,b}^{\partial \D}(\Phi^r)$.

The probability
$\P(\A_{-a,b}^{\partial \D}(\Phi^r)
\cap r\partial\D=\emptyset)$
equals the probability that a Brownian bridge from $0$ to $0$ of length
$(2\pi)^{-1}\log(r^{-1})$ does not exit
$[-a,b]$. Thus,
\begin{align*}
\lim_{r\to 0}
\P(\A_{-a,b}^{\partial \D}(\Phi^r)\cap r\partial\D=\emptyset)=1.
\end{align*}
Moreover, from Proposition 
\ref{prop:coupling} it
follows that the law of
$\ell(r)$ converges in total variation to the law of
$\ell$ as $r\to 0$. In particular,
$(\ED(\ell(r),\partial\D), 
-(2\pi)^{-1}\log\crad(0,\D\backslash\ell(r)))$
converges in law to
\\$(\ED(\ell,\partial\D), 
-(2\pi)^{-1}\log\crad(0,\D\backslash\ell))$.
By Lemma \ref{LemCREDlim},
\begin{align*}
\lim_{r\to 0}\Big\vert -\dfrac{1}{2\pi}\log\crad(0,\D\backslash\ell)
-\ED(\partial\D,r\partial \D)
+\ED(\ell,\partial \D)\Big\vert =
0 \text{ a.s.}
\end{align*}
The convergence in total variation of $\ell(r)$ to $\ell$ implies the convergence in probability to $0$ of
\begin{align*}
\Big\vert -\dfrac{1}{2\pi}\log\crad(0,\D\backslash\ell(r))
-\ED(\partial\D,r\partial \D)
+\ED(\ell(r),\partial \D)\Big\vert.
\end{align*}
Thus, $(\ED(\ell,\partial\D), 
-(2\pi)^{-1}\log\crad(0,\D\backslash\ell))$ is a limit in law of
\begin{align}
\label{EqApproxLaw}
(\ED(\ell(r),\partial\D), 
\ED(\partial\D,r\partial \D)
-\ED(\ell(r),\partial \D)).
\end{align}
To identify the limiting law of \eqref{EqApproxLaw}, we finally use 
Proposition \ref{p. law of TVS}: indeed the limit $r\to 0$ corresponds to the limit  $L\to +\infty$ on the side of the Brownian bridge.
\end{proof}

\subsection{Proof of Theorem \ref{t.loop soup}}
\label{Subsec Thm_1_2}
Exactly as in the proofs of Theorems \ref{t. joint law TVS} and
\ref{t. law FPS}, Theorem \ref{t.loop soup} (i.e. a statement in the unit disk) will follow from analogous statement in an annulus $\An_r$ by letting $r\to 0$. 

So to avoid extra notations, let us here concentrate only on the proof in the annulus. To do this, first recall the construction of relevant local set in the annulus, as given at the end of Subsection \ref{ss.RN for annulus}. 

In the annulus $\An_r$ we first sample the local set 
$\A^{u_{v},\partial_{o}}_{-2\lambda,2\lambda}$, and if
$\A^{u_{v},\partial_{o}}_{-2\lambda,2\lambda}
\cap\partial_i=\emptyset$, we further sample an FPS 
$\A_{-2\lambda}(\Phi^{\A^{u_{v},\partial_{o}}_{-2\lambda,2\lambda}})$ or $\A_{-2\lambda}(-\Phi^{\A^{u_{v},\partial_{o}}_{-2\lambda,2\lambda}})$ inside the 
non-contractible connected component of 
$\An_r\backslash \A^{u_{v},\partial_{o}}_{-2\lambda,2\lambda}$, depending on whether the sign of the GFF is equal to the sign of the label 
$\alpha^{u_{v}}=\alpha^{u_{v}}_{-2\lambda, 2\lambda}$. The explored set is, in fact, a local set and is denoted $\widecheck{A}^{u_{v}}_{0}$.
Further, let 
$\ell^{u_v}=\ell^{u_v}_{-2\lambda,2\lambda}$ 
denote the non-contractible loop of 
$\A^{u_{v},\partial_{o}}_{-2\lambda,2\lambda}$
(if $\A^{u_{v},\partial_{o}}_{-2\lambda,2\lambda}
\cap\partial_i=\emptyset$, otherwise 
$\ell^{u_v}=\partial_i$), and
$\check{\ell}^{u_v}_{0}$ the non-contractible loop of 
$\widecheck{A}^{u_v}_{0}$
(if $\widecheck{A}^{u_v}_{0}
\cap\partial_i=\emptyset$, otherwise 
$\check{\ell}^{u_v}_{0}=\partial_i$). To emphasize the dependence on the domain, we will also use the notation
$\widecheck{A}^{u_{v}}_{0}(\Phi^r)$, where
$\Phi^r$ is a zero boundary GFF in $\An_r$. 

To state all the laws related to 
$(\ell^{u_{v}},\check{\ell}^{u_{v}}_{0})$ in the annulus $\An_r$, we introduce two further times related to the Brownian bridge 
$(\widehat{B}^{v}_{t})_{0\leq t\leq L}$:
\begin{align*}
\widecheck{T}^{v}_{0}:=&
\inf\{\widehat{T}^{v}_{-2\lambda,2\lambda}\leq t\leq L:
\widehat{B}^{v}_{t}=0\}\wedge L,\\
\widecheck{\tau}^{v}_{0}:=&
\left\lbrace
\begin{array}{ll}
\sup\left\{0\leq t\leq \widecheck{T}^{v}_{0}
-\widehat{T}^{v}_{-2\lambda,2\lambda}:
\widehat{B}^{v}_{\widehat{T}^{v}_{-2\lambda,2\lambda}+t}
=
\widehat{B}^{v}_{\widehat{T}^{v}_{-2\lambda,2\lambda}}
\right\}
 & 
\text{if } \widecheck{T}^{v}_{0}<L, \\ 
L & \text{if } \widecheck{T}^{v}_{0}=L.
\end{array} 
\right.
\end{align*}

We can now state the relevant proposition for the annulus, which implies Theorem \ref{t.loop soup} by taking $r \to 0$.
\begin{prop}
\label{Prop cluster annulus}
For every $v\in\mathbb{R}$ one has the following joint laws.
\begin{enumerate}
\item The quadruple
$(\ED(\partial_o,\ell^{u_{v}}),
\ED(\ell^{u_{v}},\partial_i),
\ED(\ell^{u_{v}},\check{\ell}^{u_{v}}_{0}),
\ED(\check{\ell}^{u_{v}}_{0},\partial_i))$
has same joint law as
$(\widehat{\tau}^{v},
L-\widehat{T}^{v}_{-2\lambda,2\lambda},
\widecheck{\tau}^{v}_{0},L-\widecheck{T}^{v}_{0})$.
\item The quadruple
$(\ED(\partial_o,\ell^{u_{v}}),
\ED(\ell^{u_{v}},\check{\ell}^{u_{v}}_{0}),
\ED(\partial_o,\check{\ell}^{u_{v}}_{0})
\ED(\check{\ell}^{u_{v}}_{0},\partial_i))$
has same joint law as the quadruple
$(\widehat{\tau}^{v},
\widecheck{\tau}^{v}_{0},
\widehat{T}^{v}_{-2\lambda,2\lambda}+\widecheck{\tau}^{v}_{0},
L-\widecheck{T}^{v}_{0})$.
\end{enumerate}
\end{prop}

Recalling the definitions of $q_{-a}$ from Equation \eqref{Eqqa} and 
$q_{-a,b}$ from Equation \eqref{Eqqa}, we set 
\begin{displaymath}
\widecheck{q}_{0}(t):=
\int_{0}^{t}
(q_{-2\lambda,2\lambda}(s,2\lambda)+
q_{-2\lambda,2\lambda}(s,-2\lambda))
q_{-2\lambda}(t-s) ds.
\end{displaymath}

Similarly to Lemmas \ref{l. law T-a} and
\ref{l. law T-a b}, one then has the following description of the law of $\widecheck{T}^{v}_{0}$.
\begin{lemma}
\label{l. law checkT}
For all $v\in\R$, the law of $\widecheck{T}^{v}_{0}$ is given by
\begin{align*}
\mathbb{P}(\widecheck{T}^{v}_{0}\in (t,t+dt),
\widecheck{T}^{v}_{0}<L)=
\widecheck{q}_{0}(t)
\dfrac{p(L-t,0,v)}{p(L,0,v)}\1_{t<L} dt.
\end{align*}
\end{lemma}
Moreover, using the same proof as for Proposition \ref{p. joint law FPS}, with $q_{-a}(t)$ replaced by $\widecheck{q}_{0}(t)$ and
$-a$ by $0$, we have the following characterization of the joint law of $(\widecheck{\tau}^{v}_{0},L-\widecheck{T}^{v}_{0})$.
\begin{prop}\label{p. joint law plus}
	Assume that for every $v\in\R$, there is a couple of random variables 
	$(\sigma_{o}^{v},\sigma_{i}^{v})$, with 
	$\sigma_{o}^{v},\sigma_{i}^{v}\in [0,L]$ and
	$\sigma_{o}^{v}+\sigma_{i}^{v}\leq L$, such that
	\begin{enumerate}
		\item $\sigma_{o}^{v}$, respectively $\sigma_{i}^{v}$, has same law as
		$\widehat{T}^{v}_{-2\lambda,2\lambda}+\widecheck{\tau}^{v}_{0}$, respectively 
		$L-\widecheck{T}^{v}_{0}$,
		\item for every $v\in \R$ and on the event 
		$\sigma_{i}^{v}> 0$, the conditional law of 
		$\sigma_{o}^{v}$ given $\sigma_{i}^{v}$ is the same as the conditional law of
		$\sigma_{o}^{0}$ given $\sigma_{i}^{0}$.
	\end{enumerate}
	Then, for every $v\in\R$, the couple
	$(\sigma_{o}^{v},\sigma_{i}^{v})$ has the same law as
	$(\widehat{T}^{v}_{-2\lambda,2\lambda}
	+\widecheck{\tau}^{v}_{0},L-\widecheck{T}^{v}_{0})$.
\end{prop}
We now turn towards the proof of Proposition \ref{Prop cluster annulus}. In fact the point (1) follows easily:

\begin{proof}[Proof of Proposition \ref{Prop cluster annulus} (1)] The first point follows simply by first applying Proposition \ref{p. law of TVS} to get the law of 
$(\ED(\partial_o,\ell^{u_{v}}),
\ED(\ell^{u_{v}},\partial_i))$, and then, conditionally 
$\A^{u_{v},\partial_{o}}_{-2\lambda,2\lambda}$, 
by further applying Proposition \ref{p. law of FPS} to get the conditional law of
$(\ED(\ell^{u_{v}},\check{\ell}^{u_{v}}_{0}), 
\ED(\check{\ell}^{u_{v}}_{0},\partial_i))$ given 
$\ell^{u_{v}}$.
\end{proof}

The proof of the second point is more involved. We would like to now first condition on $\check{\ell}^{u_{v}}_{0}$ and the GFF in the annulus between $\check{\ell}^{u_{v}}_{0}$ and $\partial_i$, and then
describe the conditional law of $\ell^{u_{v}}$. However, compared to the point (1) the difficulty is the conditional law in the annulus between 
$\partial_o$ and $\check{\ell}^{u_{v}}_{0}$ is not that of a GFF, but rather of a conditioned GFF.
So we will need to use a conditioned version of Proposition \ref{p. law of TVS}. So let us state this before giving the proof of (2).
Recall that $\Phi^r$ is the zero boundary GFF in the annulus $\An_r$ and
$\Phi^r+u_{2\lambda}$ is the GFF with boundary condition
$0$ on $\partial_o$, and $2\lambda$ on $\partial_i$.
As in Subsection \ref{ssrv},
and in particular in Lemma \ref{lembasic} in the case $N=1$, we consider the exploration obtained by sampling successive local sets in the non-contractible connected component of the complementary of the preceding one,
starting with 
$\A_{-2\lambda, 2\lambda}^{u_{2\lambda},\partial_o}$.
Let $\ell_{0}^{1}, \ell_{1}^{1}, \ell_{2}^{1}, \dots, 
\ell_{n}^{1}, \ell_{n+1}^{1}$ 
be the sequence of created loops, with
$\ell_{0}^{1}=\partial_o$ and $\ell_{n+1}^{1}=\partial_i$. Let
$\alpha_{j}^{1}\in 2\lambda\mathbb{Z}$ be the label of the loop 
$\ell_{j}^{1}$ towards $\partial_i$.

By construction, $n\geq 1$,
$\alpha_{1}^{1}\in\{-2\lambda,2\lambda\}$,
$\vert \alpha_{j+1}^{1} - \alpha_{j}^{1}\vert = 2\lambda$ 
for every
$j\in\{1,2,\dots,n-1\}$, and $\alpha_{n}^{1} =2\lambda$.
Let $\mathtt{E}(\Phi^r)$ be the event on the GFF 
$\Phi^r$ defined by requiring that 
$\alpha_{1}^{1}=2\lambda$ and 
$\alpha_{j}^{1} > 0$ for every $j\in\{1,2,\dots,n-1\}$.
Given $(\widehat{B}^{2\lambda}_{t})_{0\leq t\leq L}$ the Brownian bridge from $0$ to $2\lambda$,
$\mathtt{E}(\widehat{B}^{2\lambda})$ will denote the event that
$\widehat{B}^{2\lambda}$ hits $2\lambda$ before $-2\lambda$
($\widehat{B}^{2\lambda}
_{\widehat{T}^{2\lambda}_{-2\lambda,2\lambda}}=2\lambda$)
and that $\widehat{B}^{2\lambda}$ does not hit $0$ after time
$\widehat{T}^{2\lambda}_{-2\lambda,2\lambda}$.

\begin{lemma}
\label{Lem law exc}
We have that 
$\mathbb{P}(\mathtt{E}(\Phi^r)
=\mathbb{P}(\mathtt{E}(\widehat{B}^{2\lambda}))$.
Moreover, the law of
$(\ED(\ell_{1}^{1},\partial_{o}),
\ED(\ell_{1}^{1},\partial_{i}))$ conditional on
$\mathtt{E}(\Phi^r)$ is the same as the law of
$(\widehat{\tau}^{2\lambda}_{-2\lambda,2\lambda},
L-\widehat{T}^{2\lambda}_{-2\lambda,2\lambda})$
conditional on $\mathtt{E}(\widehat{B}^{2\lambda})$.
\end{lemma}

Notice that to prove the lemma, we cannot couple the GFF $\Phi$ and the Brownian bridge $\widehat{B}^{2\lambda}$ directly through a local set process as we did in Lemma \ref{lembasic} via Proposition \ref{p.LawELBddTVS}. Indeed, this would not relate $\ED(\ell_{1}^{1},\partial_{o})$ and
$\widehat{\tau}^{2\lambda}_{-2\lambda,2\lambda}$.
\begin{proof}
	By Proposition \ref{p. law of TVS} we can couple
	$(\A_{-2\lambda, 2\lambda}^{u_{2\lambda},
		\partial_o},\alpha_{1}^{1})$
	with 
	$(\widehat{\tau}^{2\lambda}_{-2\lambda,2\lambda},
	\widehat{T}^{2\lambda}_{-2\lambda,2\lambda},
	\widehat{B}^{2\lambda}
	_{\widehat{T}^{2\lambda}_{-2\lambda,2\lambda}})$,
	such that
	$\ED(\ell_{1}^{1},\partial_{o})= 
	\widehat{\tau}^{2\lambda}_{-2\lambda,2\lambda}$,
	$\ED(\ell_{1}^{1},\partial_{i})=
	L-\widehat{T}^{2\lambda}_{-2\lambda,2\lambda}$, and
	$\alpha_{1}^{1}=
	\widehat{B}^{2\lambda}_{\widehat{T}
	^{2\lambda}_{-2\lambda,2\lambda}}.$ 
	Now, using the strong Markov property of the Brownian motion w.r.t 
$\widehat{T}^{2\lambda}_{-2\lambda,2\lambda}$ and the fact that 
$\A_{-2\lambda, 2\lambda}^{u_{2\lambda}}$ is a local set, we can couple
	\begin{displaymath}
	(\widehat{B}^{2\lambda}_{\widehat{T}^{2\lambda}_{-2\lambda,2\lambda}+t}-
	\widehat{B}^{2\lambda}_{\widehat{T}^{2\lambda}_{-2\lambda,2\lambda}})_{0\leq t\leq 
		L-\widehat{T}^{2\lambda}_{-2\lambda,2\lambda}}
	\end{displaymath}
	with a local set process exploring
	$\Phi^{\A_{-2\lambda, 2\lambda}^{u_{2\lambda},
			\partial_o}}$ in the non-contractible connected
	component of 
	$\An_{r}\backslash 
	\A_{-2\lambda, 2\lambda}^{u_{2\lambda},\partial_o}$ exactly as in the proof of Proposition \ref{p.LawELBddTVS} .
	As in this construction,
	the events $\mathtt{E}(\Phi^r)$ and 
	$\mathtt{E}(\widehat{B}^{2\lambda})$
	coincide, we deduce the lemma.
\end{proof}

We are now ready to prove the second point of Proposition \ref{Prop cluster annulus}.

\begin{proof}[Proof of Proposition \ref{Prop cluster annulus} (2)]

From the point (1) we know that 
$\ED(\check{\ell}^{u_{v}}_{0},\partial_i)$
is distributed as $L-\widecheck{T}^{v}_{0}$. Moreover, from the general reversibility statement, i.e. Proposition \ref{rev:bg} in Subsection \ref{Ss.Extensions to non-zero}, it follows that the marginal law of $\ED(\partial_o,\check{\ell}^{u_{v}}_{0})$ equals to that of
$\widehat{T}^{v}_{-2\lambda,2\lambda}+\widecheck{\tau}^{v}_{0}$.

As Proposition \ref{p.change of measure cluster} ensures that the assumptions of Proposition \ref{p. joint law plus} are met, we can apply Proposition \ref{p. joint law plus}
to conclude that the joint law of the pair
$(\ED(\partial_o,\check{\ell}^{u_v}_{0}),
\ED(\check{\ell}^{u_v}_{0},\partial_i))$
is distributed as
$(\widehat{T}^{v}_{-2\lambda,2\lambda}+\widecheck{\tau}^{v}_{0},
L-\widecheck{T}^{v}_{0})$. It thus remains to deal with the conditional law of
$(\ED(\partial_o,\ell^{u_{v}}),
\ED(\ell^{u_{v}},\check{\ell}^{u_v}_{0}))$
given $\check{\ell}^{u_{v}}_{0}$.

First we claim that conditionally on 
$(\widehat{T}^{v}_{-2\lambda,2\lambda}
+\widecheck{\tau}^{v}_{0},
(\widehat{B}^{v}_{t})_{\widehat{T}^{v}_{-2\lambda,2\lambda}
+\widecheck{\tau}^{v}_{0}\leq t\leq L})$, the process
$$(\operatorname{sign}(\widehat{B}^{v}_{\widehat{T}^{v}_{-2\lambda,2\lambda}})\widehat{B}^{v}_{t})
_{0\leq t\leq\widehat{T}^{v}_{-2\lambda,2\lambda}
+\widecheck{\tau}^{v}_{0}}$$
has the law of a Brownian bridge
$\widehat{B}^{2\lambda}$ from $0$ to
$2\lambda$ of duration $\widehat{T}^{v}_{-2\lambda,2\lambda}
+\widecheck{\tau}^{v}_{0}$,
conditioned on the event
$\mathtt{E}(\widehat{B}^{2\lambda})$ above. We omit the proof, as this is a standard argument using for example discrete time random walks, and convergence, or can be also deduced from decompositions of Brownian trajectories, as in Section 1 in \cite{SVY}, Chapter XII § 3 in \cite{RevuzYor1999BMGrundlehren},
Chapter IV Section 3.18 in
\cite{BorodinSalminen2015}.

Further, let $\widecheck{O}^{u_{v}}_{o}$ be the connected component of
$\An_r\backslash \check{\ell}^{u_{v}}_{0}$ containing
$\partial_o$ on its boundary.
Let $\widecheck{u}^{v}_{o}$ be the harmonic function on
$\widecheck{O}^{u_v}_{o}$ with boundary values
$0$ on $\partial_o$ and $\alpha^{u_v}$ on
$\check{\ell}^{u_v}_{0}$. 
Then, from Corollary \ref{Cor cond GFF several} it follows that conditionally on
$\check{\ell}^{u_v}_{0}$, $\Phi^{\An_{r}}$ restricted to $\An_r \backslash \widecheck{O}^{u_v}_{o}$,
$\ED(\partial_o,\check{\ell}^{u_v}_{0})=
\frac{1}{2\pi}\log(1/r')$, 
$\ell^{u_v}\neq \partial_i$ and
$\alpha^{u_v}=2\lambda$, the field
$\Phi^r+u_{v}-\widecheck{u}^{v}_{o}$
under a conformal map $\phi$ uniformizing
$\widecheck{O}^{u_v}_{o}$ to the annulus of form
$\An_{r'}$, 
is distributed as a zero boundary GFF $\Phi^{r'}$ in $\widecheck{O}^{u_{v}}_{o}$ conditioned on the event
$\mathtt{E}(\Phi^{r'})$.

Hence, Lemma \ref{Lem law exc} ensures that the conditional law of 
$(\ED(\partial_o,\ell^{u_{v}}),
\ED(\ell^{u_{v}},\check{\ell}^{u_{v}}_{0}))$
given
\\$(\ED(\partial_o,\check{\ell}^{u_{v}}_{0}),
\ED(\check{\ell}^{u_{v}}_{0},\partial_i))$
is that of
$(\widehat{\tau}^{v}_{-2\lambda,2\lambda},
\widecheck{\tau}^{v}_{0})$
given
$(\widehat{T}^{v}_{-2\lambda,2\lambda}+\widecheck{\tau}^{v}_{0},
L-\widecheck{T}^{v}_{0})$ and we conclude.
\end{proof}

\subsection{Proof of Corollary \ref{CorAnnularExp}}
\label{SebSec Ann_exp}
Let $\ell$ be the CLE$_{4}$ loop in $\D$ surrounding $0$.
Recall the Brownian times $T$ and $\tau$ from Equation \eqref{Eq T tau}:
\begin{align*}
&T:=\inf\{t\geq 0: |B_t|=\pi\},\\
&\tau:=\sup\{0\leq t \leq T: B_t=0\}.
\end{align*}
According to Theorem \ref{t. main}, there is a coupling such that
\begin{align*}
(2\pi \ED(\ell,\partial \D) , -\log\crad(0, \D\backslash \ell) )
=(\tau, T)~~\text{ a.s.}
\end{align*}
Thus, by
\eqref{EqKoebeQuarter}, Proposition \ref{PropRplusED} and
 Corollary \ref{CorRatio},
\begin{align*}
\dfrac{1}{4} e^{\tau}\leq \dfrac{1}{r_{+}(\ell)}
\leq e^{\tau},\qquad
e^{T}\leq \dfrac{1}{r_{-}(\ell)}
\leq 4 e^{T},\qquad
		e^{T-\tau}
		\leq
		\dfrac{r_{+}(\ell)}{r_{-}(\ell)}
		\leq 
		16 e^{T-\tau}~~\text{ a.s.}
\end{align*}

One has
\begin{align*}
\mathbb{P}(\tau>t)=e^{-\frac{1}{8}t + o(t)},
\qquad
\mathbb{P}(T>t)=e^{-\frac{1}{8}t + o(t)}
\end{align*}
(see Appendix A, in particular Proposition \ref{Prop A1 laws} (2)).
Indeed, $-1/8$ is the first eigenvalue of
$\frac{1}{2}\frac{d^{2}}{dx^{2}}$ with zero Dirichlet boundary conditions on the interval $(-\pi,\pi)$. This gives the exponents for
$r_{+}(\ell)$ and $r_{-}(\ell)$.

Further, the path $(B_{\tau + t})_{t\geq 0}$ is independent from
$(\tau,(B_{t})_{0\leq t\leq\tau})$ (this can be seen for instance with the excursion theory of 
Brownian motion.)
Now, the path 
$(\vert B_{\tau + t}\vert)_{0\leq t\leq T-\tau}$
is distributed like a Bessel 3 process started from $0$ and run until the first hitting time of level $\pi$. 
Thus from Proposition \ref{PropExpTailBessel},
\begin{align*}
\mathbb{P}(T-\tau >t) = e^{-\frac{1}{2}t + o(t)}.
\end{align*}
It follows that
\begin{align*}
\mathbb{P}(r_{+}(\ell) / r_{-}(\ell)>R)
= R^{-\frac{1}{2} + o(1)}.
\end{align*}

To obtain the statement about the stationary CLE$_4$ loop, observe further that the conclusions of 
Corollary \ref{CorAnnularExp} still hold if one conditions
$\ell$ on $\ED(\ell,\partial\D)>L$. This is because
$\tau$ and $T-\tau$ are independent.
The convergence of Theorem \ref{ThmConvStatCLE} 
ensures that
Corollary \ref{CorAnnularExp} also holds if $\ell$ is distributed according to $\mathbb{P}^{\rm stat}_{\operatorname{CLE}_{4}}$,
the stationary CLE$_{4}$ distribution in $\C$.
See Subsection \ref{SubSecStationary}.

\section*{Appendix A: list of laws for Brownian motion, 
Brownian bridge and Bessel 3}

The joint laws appearing in Theorems \ref{t. joint law TVS}
and \ref{t. law FPS} and Propositions \ref{p. law of TVS} and
\ref{p. law of FPS} can be expressed explicitly. For this we introduce
\begin{align*}
&\beta(t,x):=\dfrac{\sqrt{2}x}{\sqrt{\pi t^{5}}}
\sum_{k=0}^{+\infty}((2k+1)^{2}x^{2}-t)
e^{-\frac{(2k+1)^{2}x^{2}}{2t}},
\qquad t>0, x>0.
\end{align*}
$\1_{t>0} \beta(t,x)$ is the density of the first hitting time of level 
$x$ of a Bessel 3 process started at $0$
(Formula 2.0.2, Section 5.2 in \cite{BorodinSalminen2015}).
We recall that $p(t,x,y)$ denotes the heat kernel \eqref{HeatKernel}
on $\R$.
We will also need $p_{-a}(t,x,y)$, the heat kernel on
$(-a,+\infty)$ with zero Dirichlet condition in $-a$, and
$p_{-a,b}(t,x,y)$ the heat kernel on
$(-a,b)$ with zero Dirichlet condition in $-a$ and $b$:

\begin{eqnarray*}
	p_{-a}(t,x,y)&:=&\dfrac{1}{\sqrt{2\pi t}}
	(e^{-\frac{(y-x)^{2}}{2t}}-
	e^{-\frac{(y+x+2a)^{2}}{2t}}),\\
	p_{-a,b}(t,x,y)&:=&\dfrac{1}{\sqrt{2\pi t}}
	\sum_{k\in\mathbb{Z}}
	(e^{-\frac{(y-x+2k(b+a))^{2}}{2t}}-
	e^{-\frac{(y+x+2a+2k(b+a))^{2}}{2t}})\\
	&=&\dfrac{2}{b+a}\sum_{j\geq 0}
	\cos\Big(\dfrac{(2j+1)\pi}{b+a}\Big(x-\dfrac{b-a}{2}\Big)\Big)
	\cos\Big(\dfrac{(2j+1)\pi}{b+a}\Big(y-\dfrac{b-a}{2}\Big)\Big)
	e^{- \frac{(2j+1)^{2}\pi^{2}}{2(b+a)^{2}}t}
	\\
	&& +
	\dfrac{2}{b+a}\sum_{j\geq 1}
	\sin\Big(\dfrac{2j\pi}{b+a}\Big(x-\dfrac{b-a}{2}\Big)\Big)
	\sin\Big(\dfrac{2j\pi}{b+a}\Big(y-\dfrac{b-a}{2}\Big)\Big)
	e^{- \frac{2j^{2}\pi^{2}}{(b+a)^{2}}t}.
\end{eqnarray*}

The heat kernel $p_{-a,b}(t,x,y)$ is related to theta functions of imaginary argument (Appendix 2.13 in \cite{BorodinSalminen2015}).
Its first expression 
(Formula 3.0.2, Section 1.3 in \cite{BorodinSalminen2015})
comes from the reflection principle. 
The second one corresponds to a decomposition along the eigenfunctions of $\frac{1}{2}\frac{d^{2}}{dy^{2}}$
(see also \cite{AbramowitzStegun64}, Section 16.27).
The two expressions are related by the Poisson summation formula.

Next we give the joint laws for
$(\tau_{-a}, T_{-a}), (\tau_{-a,b}, T_{-a,b}, B_{T_{-a,b}}),
(\widehat{\tau}_{-a}^{v}, \widehat{T}_{-a}^{v})$
and
$(\widehat{\tau}_{-a,b}^{v}, \widehat{T}_{-a,b}^{v},
\widehat{B}^{v}_{\widehat{T}_{-a,b}^{v}})$.
We refer to Section 1 in \cite{SVY},
Chapter II Section 3.20 in \cite{BorodinSalminen2015},
and the references therein for the notion of \textit{last exit time decomposition}.

\begin{propA}
\label{Prop A1 laws}
	One has the following explicit laws.
	\begin{enumerate}
		\item With the notations of Theorem \ref{t. law FPS}, the joint law of
		$(\tau_{-a}, T_{-a})$ is given by
		\begin{align*}
		\mathbb{P}(\tau_{-a}\in (t_{1},t_{1}+dt_{1}),
		T_{-a}\in (t_{2},t_{2}+dt_{2})) =
		\dfrac{\1_{0<t_{1}<t_{2}}}{4\lambda}
		p_{-a}(t_{1},0,-a+2\lambda)
		\beta(t_{2}-t_{1},2\lambda) dt_{1} dt_{2}.
		\end{align*}
		\item With the notations of Theorem \ref{t. joint law TVS}, 
		the joint law of $(\tau_{-a,b}, T_{-a,b}, B_{T_{-a,b}})$ is given by
		\begin{align*}
		&\mathbb{P}(
		\tau_{ -a,b}\in (t_1,t_1+dt_1),	
		T_{-a,b}\in (t_{2}, t_{2}+dt_2),
		B_{T_{-a,b}}=-a)\\
		&\hspace{0.14\textwidth}=
		\1_{0<t_{1}<t_{2}}
		\dfrac{b+a}
		{4b\lambda}
		p_{-a,b}(t_{1},0,-a+2\lambda)
		\beta(t_{2}-t_{1},2\lambda)dt_1dt_2,\\
		&\mathbb{P}(
		\tau_{ -a,b}\in (t_1,t_1+dt_1),
		T_{-a,b}\in (t_{2}, t_{2}+dt_2),
		B_{T_{-a,b}}=b)\\
		&\hspace{0.14\textwidth}=
		\1_{0<t_{1}<t_{2}}
		\dfrac{b+a}
		{4a\lambda}
		p_{-a,b}(t_{1},0,b-2\lambda)
		\beta(t_{2}-t_{1},2\lambda) 
		dt_{1} dt_{2}.
		\end{align*}
	In particular,
		\begin{align*}
		&\mathbb{P}(
		\tau_{ -a,b}\in (t_1,t_1+dt_1),B_{T_{-a,b}}=-a)=
		\1_{0<t_{1}}
		\dfrac{b+a}{4b\lambda}
		p_{-a,b}(t_{1},0,-a+2\lambda) dt_{1},\\
		&\mathbb{P}(
		\tau_{ -a,b}\in (t_1,t_1+dt_1),B_{T_{-a,b}}=b)=
		\1_{0<t_{1}}
		\dfrac{b+a}{4a\lambda}
		p_{-a,b}(t_{1},0,b-2\lambda) dt_{1},
		\end{align*}
		and
		$$\mathbb{P}(
		T_{ -a,b}\in (t_2,t_2+dt_2),B_{T_{-a,b}}=c)
		=\1_{0<t_{2}}q_{-a,b}(t_{2},c)dt_2,
		c\in\{-a,b\},$$ where
		\begin{align*}
		q_{-a,b}(t,c)=\sum_{k\in\mathbb{Z}}
		\dfrac{|c|+2k(b+a)}{\sqrt{2\pi t^{3}}}
		e^{-\frac{(|c|+2k(b+a))^{2}}{2t}}.
		\end{align*}
		Moreover,
		\begin{align*}
		q_{-a,b}(t,-a)+q_{-a,b}(t,b)=
		-\dfrac{d}{dt}\int_{-a}^{b}p_{-a,b}(t,0,y) dy.
		\end{align*}
		
		\item With the notations of Proposition \ref{p. law of FPS},
		the joint law of $(\widehat{\tau}_{-a}^{v}, \widehat{T}_{-a}^{v})$
		is given by
		\begin{multline*}
		\mathbb{P}(
		\widehat{\tau}^{v}_{-a}\in (t_1,t_1+dt_1),
		\widehat{T}^{v}_{ -a}\in (t_{2}, t_{2}+dt_2))\\=
		\dfrac{\1_{0<t_{1}<t_{2}<L}}{4\lambda}p_{ -a}(t_{1},0,-a+2\lambda)
		\beta(t_{2}-t_{1},2\lambda) 
		\dfrac{p(L-t_{2},-a,v)}{p(L,0,v)} dt_{1} dt_{2}.
		\end{multline*}
		\item With the notations of Proposition \ref{p. law of TVS}
		the joint law of
		$(\widehat{\tau}_{-a,b}^{v}, \widehat{T}_{-a,b}^{v},
		\widehat{B}^{v}_{\widehat{T}_{-a,b}^{v}})$
		is given by
		\begin{align*}
		&\mathbb{P}(
		\widehat{\tau}^{v}_{ -a,b}\in (t_1,t_1+dt_1),
		\widehat{T}^{v}_{-a,b}\in (t_{2}, t_{2}+dt_2),
		\widehat{B}^{v}_{\widehat{T}^{v}_{-a,b}}=-a)\\
		&\hspace{0.14\textwidth}=
		\1_{0<t_{1}<t_{2}<L}
		\dfrac{b+a}
		{4b\lambda}
		p_{-a,b}(t_{1},0,-a+2\lambda)
		\beta(t_{2}-t_{1},2\lambda) 
		\dfrac{p(L-t_{2},-a,v)}{p(L,0,v)}
		dt_{1} dt_{2},\\
		&\mathbb{P}(
		\widehat{\tau}^{v}_{ -a,b}\in (t_1,t_1+dt_1),
		\widehat{T}^{v}_{-a,b}\in (t_{2}, t_{2}+dt_2),
		\widehat{B}^{v}_{\widehat{T}^{v}_{-a,b}}=b)\\
		&\hspace{0.14\textwidth}=
		\1_{0<t_{1}<t_{2}<L}
		\dfrac{b+a}
		{4a\lambda}
		p_{-a,b}(t_{1},0,b-2\lambda)
		\beta(t_{2}-t_{1},2\lambda) 
		\dfrac{p(L-t_{2},-a,v)}{p(L,0,v)} 
		dt_{1} dt_{2}.
		\end{align*}
	\end{enumerate}
\end{propA}

For $x>0$, let $\mathcal{T}_{x}$ denote the first time a Bessel 3 processes started from 0 hits the level $x$. Its distribution is
$\1_{t>0}\beta(t,x) dt$. One has the following exponential asymptotics for the tail
of $\mathcal{T}_{x}$.

\begin{propA}
\label{PropExpTailBessel}
As $t\to +\infty$,
\begin{align*}
\mathbb{P}(\mathcal{T}_{x}>t) = 
e^{-\frac{\pi^{2}}{2 x^{2}} t + o(t)}.
\end{align*}
\end{propA}
\begin{proof}
The point $0$ maximizes the hitting time of level $x$ by a Bessel 3 among the starting points in $[0,x]$. 
Thus, for every $t,s>0$,
\begin{align*}
\mathbb{P}(\mathcal{T}_{x}>t+s)
\leq \mathbb{P}(\mathcal{T}_{x}>t)\mathbb{P}(\mathcal{T}_{x}>s).
\end{align*}
Moreover, the function
$t\mapsto \mathbb{P}(\mathcal{T}_{x}>t)$ is continuous.
By Fekete's subadditivity lemma,
\begin{align*}
\lim_{t\to +\infty}\dfrac{1}{t}\log\mathbb{P}(\mathcal{T}_{x}>t)=
\inf_{t>0}\dfrac{1}{t}\log\mathbb{P}(\mathcal{T}_{x}>t)
\in [-\infty, 0).
\end{align*}
According to Formula 2.0.1, Section 5.2 in \cite{BorodinSalminen2015},
for every $\nu>0$,
\begin{align*}
\E[e^{-\nu \mathcal{T}_{x}}] =
\dfrac{x\sqrt{2\nu}}{\sinh (x\sqrt{2\nu})}.
\end{align*}
By analytic continuation,
\begin{align*}
\E[e^{\nu \mathcal{T}_{x}}] =
\dfrac{x\sqrt{2\nu}}{\sin (x\sqrt{2\nu})}
\end{align*}
for $\nu\in (0,\frac{\pi^{2}}{2 x^{2}})$, and
\begin{align*}
\lim_{\nu\to \frac{\pi^{2}}{2 x^{2}}}\E[e^{\nu \mathcal{T}_{x}}] =
+\infty .
\end{align*}
Thus,
\begin{displaymath}
\lim_{t\to +\infty}\dfrac{1}{t}\log\mathbb{P}(\mathcal{T}_{x}>t)
=-\dfrac{\pi^{2}}{2 x^{2}}.
\qedhere
\end{displaymath}
\end{proof}

\section*{Appendix B: proofs of some basic results}

\subsection*{Proof of Lemma \ref{LemCREDlim}}

\begin{proof}
	Let $f$ be a holomorphic function uniformizing the interior surrounded by $\ell$ into $\D$, with $f(0)=0$. Then,
	$\crad(0,\D\backslash\ell) = \vert f'(0)\vert^{-1}$.
	For $r< d(0,\ell)$, let be
	\begin{align*}
	R_{-}(r):=d(0,f(r\partial \D)),\qquad
	R_{+}(r):=\max\{\vert f(z)\vert\ : \vert z\vert = r \}.
	\end{align*}
	Then
	\begin{align*}
	R_{-}(r)=\vert f'(0)\vert r + O(r^{2}),\qquad
	R_{+}(r)=\vert f'(0)\vert r + O(r^{2}).
	\end{align*}
	By conformal invariance,
	\begin{align*}
	\ED(\ell,r\partial\D) =
	\ED(f(\ell), f(r\partial\D)) =
	\ED(\partial\D, f(r\partial\D)).
	\end{align*}
	By a comparison principle (see Theorem 4-1 in \cite{Ahlfors2010ConfInv}), we have the following bounds
	\begin{align*}
	\ED(\partial\D, R_{+}(r)\partial\D)
	\leq\ED(\partial\D, f(r\partial\D))\leq
	\ED(\partial\D, R_{-}(r)\partial\D).
	\end{align*}
	It follows that
	\begin{align*}
	\ED(\partial\D, f(r\partial\D)) =
	\dfrac{1}{2\pi}\log(\vert f'(0)\vert^{-1} r^{-1})
	+ O(r),
	\end{align*}
	which implies the lemma.
\end{proof}

\subsection*{Proof of Proposition \ref{PropRplusED}}

\begin{proof}
The lower bound comes from monotonicity. 
		Since $\wp\subset r_{+}(\wp)\overline{\D}$,
		\begin{align*}
		\ED(\partial\D,\wp)\geq
		\ED(\partial\D, r_{+}(\wp)\partial\D)
		=\dfrac{1}{2\pi}\log \dfrac{1}{r_{+}(\wp)} .
		\end{align*}
		For the upper bound, consider the segment
		$[0,r_{+}(\wp)]$ and the annular domain
		$\D\backslash [0,r_{+}(\wp)]$. The latter is extremal in the following sense:
		\begin{align}
		\label{EqGro1}
		\ED(\partial\D,\wp)\leq \ED(\partial\D,[0,r_{+}(\wp)]).
		\end{align}
		The latter inequality is equivalent to Grötzsch's theorem
		(Theorem 4-6 in \cite{Ahlfors2010ConfInv}.)
		The inversion $z\mapsto 1/z$ sends 
		$\D\backslash [0,r_{+}(\wp)]$ to the two-connected domain
		$\C\backslash(\overline{\D}\cup [R_{\rm Gr},+\infty))$,
		with $R_{\rm Gr}=r_{+}(\wp)^{-1}$, known as Grötzsch's annulus.
		By conformal invariance,
		\begin{align}
		\label{EqGro2}
		\ED(\partial\D,[0,r_{+}(\wp)])=
		\ED(\partial\D,[R_{\rm Gr},+\infty)).
		\end{align}
		Further, consider the Teichmüller annulus 
		$\C\backslash([-1,0]\cup [R_{\rm Tch},+\infty))$,
		with $R_{\rm Tch}=R_{\rm Gr}^{2}-1$.
		According to the Formula (4-14), Section 4-11 in
		\cite{Ahlfors2010ConfInv},
		\begin{align}
		\label{EqGro3}
		\ED([-1,0],[R_{\rm Tch},+\infty))=2\ED(\partial\D,[R_{\rm Gr},+\infty)).
		\end{align}
		According to the inequality (4-21), Section 4-12 in
		\cite{Ahlfors2010ConfInv},
		\begin{align}
		\label{EqGro4}
		e^{2\pi\ED([-1,0],[R_{\rm Tch},+\infty))}
		\leq 16(R_{\rm Tch}+1).
		\end{align}
		By combining \eqref{EqGro1}, \eqref{EqGro2}, \eqref{EqGro3}
		and \eqref{EqGro4}, one gets the desired upper bound.
\end{proof}

\subsection*{Proof of Theorem \ref{t. change of measure local sets}}

\begin{proof}
	First, use the Cameron-Martin theorem (e.g. Theorem 11 of \cite{ASW}), to see that the law of $\Phi$ under $\widetilde \P$ is that of $\widetilde \Phi + g$, where $\widetilde \Phi$ is a GFF under 
	$\widetilde{\P}$. Now, define \[\widetilde \Phi^A:=\widetilde \Phi -\Phi_A-g_A.\]
	To conclude, we need to show that under the law of $\widetilde \P$ and conditionally on $(A, \widetilde \Phi_A)=(A,\Phi_A-g_A)$ the law of $\widetilde\Phi^A$ is that of a GFF on $D\backslash A$. To do this, we just need to take $F$ a bounded measurable function and compute
	\begin{align*}
	\widetilde{\E}\left[F(\widetilde \Phi^A) \mid(A,\Phi_A-g_A)\right]&=Z_{A,\Phi_A} \E\left[F(\widetilde \Phi^A)\exp((\Phi,g)_{\nabla}-\frac{1}{2}(g,g)_\nabla)\mid (A, \Phi_A) \right]  \\
	&=Z'_{A,\Phi_A} \E\left[F(\Phi^A+g-g_A)\exp((\Phi^A,g-g_A)_{\nabla}) \mid(A,\Phi_A)\right],
	\end{align*}
	where we used that $(\Phi^A,g_A)_{\nabla}$ is a.s. equal to $0$. Note that conditionally on $(A,\Phi_A)$, $\Phi^A$ is a GFF in $D\backslash A$, thus by Cameron-Martin theorem we have that
	\begin{align*}
	\widetilde{\E}\left[F(\widetilde \Phi^A) \mid(A,\Phi_A-g_A)\right]= Z''_{A,\Phi_A} \E\left[ F(\Phi^A)\mid A \right].
	\end{align*}
	Note that $Z''_{A,\Phi_A}=1$ thanks to the fact that one can take $F=1$. Thus, we obtained that under the law $\widetilde \P$ and conditionally on $(A,\widetilde \Phi_A)$,  $\widetilde \Phi^A$ has the law of a GFF in $D\backslash A$.	 
\end{proof}

\section*{Acknowledgements}The authors are thankful to Wendelin Werner for inspiring discussions and for pointing out that the joint laws of multiple
nested interfaces cannot be read out from a Brownian motion in the naive way. This work was partially supported by the SNF grants SNF-155922 and SNF-175505. A. Sepúlveda was supported by the ERC grant LiKo 676999. J. Aru is a member of NCCR Swissmap. T. Lupu acknowledges the support of the French National Research Agency (ANR) grant within the project MALIN (ANR-16-CE93-0003). A. Sepúlveda would also like to thank the hospitality of N\'ucleo Milenio ``Stochastic models of complex and disordered systems'' for repeated invitations to Santiago, where part of this paper was written.

\bibliographystyle{alpha}	
\bibliography{biblio}

\end{document}